\documentclass[12pt, twoside]{article}
\usepackage{amsmath,amsthm,amssymb}
\usepackage{times}
\usepackage{enumerate}

\usepackage{graphics}
\usepackage{cite}
\usepackage{amsfonts}
\usepackage{float}
\usepackage{epic}
\usepackage{eepic}
\usepackage{color}
\usepackage{ae}
\usepackage{mathrsfs,array,graphicx}
\usepackage[all,ps]{xy}
\usepackage{epstopdf}
\usepackage{a4wide}
\usepackage{hyperref} \hypersetup{colorlinks,allcolors=[rgb]{0,0,0}}
\usepackage[a4paper,margin=1in]{geometry}
\usepackage{enumitem}

\def\titlerunning#1{\gdef\titrun{#1}}
\makeatletter
\def\author#1{\gdef\autrun{\def\and{\unskip, }#1}\gdef\@author{#1}}
\def\address#1{{\def\and{\\\hspace*{18pt}}\renewcommand{\thefootnote}{}%
\footnote {#1}}%
\markboth{\autrun}{\titrun}}
\makeatother
\def\email#1{e-mail: #1}
\def\subjclass#1{{\renewcommand{\thefootnote}{}%
\footnote{\emph{Mathematics Subject Classification (2010):} #1}}}
\def\keywords#1{\par\medskip
\noindent\textbf{Keywords.} #1}

\newtheorem{theorem}{Theorem}[section]
\newtheorem{corollary}[theorem]{Corollary}
\newtheorem{lemma}[theorem]{Lemma}

\newtheorem{thmx}{Theorem}

\newtheorem{proposition}[theorem]{Proposition}

\newtheorem*{lemmastar}{\textbf{Lemma}}

\numberwithin{theorem}{section}

\theoremstyle{definition}
\newtheorem{definition}[theorem]{Definition}
\newtheorem{remark}[theorem]{Remark}
\newtheorem{example}[theorem]{Example}

\numberwithin{equation}{section}

\def\hat{\widehat}
\let\li\overline

\newcommand{\Aut}{\textup{Aut}}
\newcommand{\A}{\mathbb A}
\newcommand{\BC}{\mathrm{BC}}
\newcommand{\BGSp}{\tu{B}_{\GSp}}
\newcommand{\BSp}{\tu{B}_{\Sp}}
\newcommand{\Bst}{{\mathbb B}_\st}
\newcommand{\C}{\mathbb C}
\newcommand{\End}{\tu{End}}

\newcommand{\Frob}{\textup{Frob}}
\newcommand{\GL}{{\textup {GL}}}
\newcommand{\GO}{\textup{GO}}
\newcommand{\GSO}{\tu{GSO}}
\newcommand{\GSpcmptFv}{\tu{GSp}_{2n, F_y}^{\tu{cmpt}}}
\newcommand{\GSpin}{\textup{GSpin}}
\newcommand{\GSp}{\textup{GSp}}
\newcommand{\Gal}{{\textup{Gal}}}
\newcommand{\Ga}{{\mathbb G}_\tu{a}}
\newcommand{\Gl}{{\textup {GL}}}

\newcommand{\Gm}{\textbf{G}_{\textup{m}}}
\newcommand{\Gsp}{\textup{GSp}}
\newcommand{\G}{\mathbb G}
\newcommand{\HT}{\tu{HT}}
\newcommand{\Hodge}{\textup{Hodge}}
\newcommand{\Hom}{\textup{Hom}}
\newcommand{\Hunr}{\cH^{\tu{unr}}}
\newcommand{\IH}{\tu{IH}}

\newcommand{\Irr}{\mathrm{Irr}}
\newcommand{\Isom}{\tu{Isom}}
\newcommand{\Ker}{\textup{ker}}
\newcommand{\Lef}{\textup{Lef}}

\newcommand{\Lie}{\textup{Lie}}
\newcommand{\PGL}{\tu{PGL}}
\newcommand{\PGSp}{\textup{PGSp}}
\newcommand{\PO}{\textup{PO}}
\newcommand{\PSp}{\tu{PSp}}

\newcommand{\Q}{\mathbb Q}
\newcommand{\Rep}{{\tu{Rep}}}
\newcommand{\Res}{\textup{Res}}
\newcommand{\R}{\mathbb R}
\newcommand{\SL}{\textup{SL}}
\newcommand{\SO}{{\textup{SO}}}
\newcommand{\STO}{\tu{STO}}
\newcommand{\ST}{\textup{ST}}
\newcommand{\SU}{{\textup{SU}}}
\newcommand{\Sbad}{{S_{\tu{bad}}}}
\newcommand{\Sh}{\textup{Sh}}
\newcommand{\Spin}{\textup{Spin}}
\newcommand{\Sp}{\textup{Sp}}
\newcommand{\Stab}{\textup{stab}}
\newcommand{\St}{\textup{St}}
\newcommand{\TGO}{\uT_{\GO}}
\newcommand{\TGSpin}{\uT_{\GSpin}}
\newcommand{\TGSp}{\uT_{\GSp}}

\newcommand{\TGsp}{\uT_{\GSp}}
\newcommand{\TSO}{\uT_{\SO}}

\newcommand{\TSp}{\uT_{\Sp}}
\newcommand{\Tr}{\textup{Tr}}

\newcommand{\Zhat}{\widehat{\mathbb{Z}}}
\newcommand{\Z}{\mathbb Z}
\newcommand{\ad}{\textup{ad}}
\newcommand{\armapsto}{\ar@{|->}}
\newcommand{\bad}{\textup{bad}}

\newcommand{\bs}{\backslash}
\newcommand{\bydef}{\overset{\textup{def}}{=}}
\newcommand{\cA}{{\mathcal A}}

\newcommand{\cE}{\mathcal E}

\newcommand{\cH}{\mathcal H}

\newcommand{\cL}{\mathcal L}
\newcommand{\cN}{\mathcal N}
\newcommand{\cO}{\mathcal O}
\newcommand{\cS}{\mathcal S}
\newcommand{\cV}{\mathcal V}

\newcommand{\coh}{\tu{shim}}
\newcommand{\cusp}{\textup{cusp}}
\newcommand{\dd}{\textup{d}}
\newcommand{\der}{\textup{der}}
\newcommand{\diag}{{\textup{diag}}}
\newcommand{\disc}{\textup{disc}}
\newcommand{\el}{\mathrm{ell}}
\newcommand{\eps}{\varepsilon}
\newcommand{\ep}{\textup{ep}}
\newcommand{\et}{{\textup{\'et}}}

\newcommand{\fkX}{\mathfrak{X}}
\newcommand{\fkh}{\mathfrak{h}}
\newcommand{\fkp}{\mathfrak{p}}
\newcommand{\funcGLlef}{f_{\Lef}^{\GL_{2n+1},\theta}}

\newcommand{\funcSP}{f^{\Sp_{2n}}_{\vst}}

\newcommand{\hra}{\hookrightarrow}
\newcommand{\iH}{\mathfrak H}

\newcommand{\ie}{\textit{i.e.}\ }
\newcommand{\ig}{\mathfrak g}

\newcommand{ \injects}{\rightarrowtail}
\newcommand{\inv}{^{-1}}
\newcommand{\isomto}{\overset \sim \to}
\newcommand{\kleinvierkant}[4]{{\tiny{\lhk \begin{smallmatrix} #1 & #2 \cr #3 & #4 \end{smallmatrix}\rhk}}}

\newcommand{\lbr}{\left\lbrace}
\newcommand{\lhk}{\left(}
\newcommand{\lmapsto}{\longmapsto}
\newcommand{\lql}{\overline{\Q}_\ell}
\newcommand{\lzl}{\overline{\Z}_\ell}

\newcommand{\ol}{\overline}
\newcommand{\one}{\textbf{\textup{1}}}

\newcommand{\ql}{\Q_\ell}
\newcommand{\qp}{{\mathbb{Q}_p}}
\newcommand{\rank}{\textup{rank}}
\newcommand{\ra}{\rightarrow}
\newcommand{\rbr}{\right\rbrace}
\newcommand{\rec}{\textup{rec}}

\newcommand{\rhk}{\right)}

\newcommand{\simil}{\textup{sim}}

\newcommand{\spin}{\textup{spin}} 
\newcommand{\ssimple}{ {\tu{ss}} }
\newcommand{\sslash}{{\mathbin{/\mkern-6mu/}}}
\newcommand{\std}{\textup{std}}
\newcommand{\st}{\tu{st}}
\newcommand{\surjects}{\twoheadrightarrow}
\newcommand{\tilrhoflat}{\widetilde {\rho_{\pi^\flat}}}
\newcommand{\tu}[1]{\textup{#1}}
\newcommand{\uHom}{\underline {\Hom}}
\newcommand{\uH}{\textup{H}}
\newcommand{\uIH}{\tu{IH}}
\newcommand{\uM}{\textup{M}}
\newcommand{\uN}{\textup{N}}
\newcommand{\uO}{\textup{O}}
\newcommand{\uT}{\textup{T}}
\newcommand{\uc}{\tu{c}}
\newcommand{\unr}{\mathrm{unr}}
\newcommand{\vierkant}[4]{{\lhk \begin{smallmatrix} #1 & #2 \cr #3 & #4 \end{smallmatrix} \rhk } }
\newcommand{\vinfty}{{v_ \infty}}
\newcommand{\vol}{\textup{vol}}
\newcommand{\vst}{{v_{\tu{St}}}}
\newcommand{\wh}[1]{\widehat{#1}}
\newcommand{\wt}{\widetilde}
\renewcommand{\Gm}{{\mathbb G}_{\tu{m}}}
\renewcommand{\Gm}{{\mathbb{G}_{\textup{m}}}}
\renewcommand{\Lie}{\tu{Lie}\,}

\usepackage[continuous, page]{pagenote}
\makepagenote

\begin{document}

\baselineskip=17pt

\titlerunning{Galois representations for general symplectic groups}

\title{Galois representations for general symplectic groups}

\author{Arno Kret \and Sug Woo Shin}

\date{\today}

\maketitle

\address{A. Kret: Korteweg-de Vries Institute
Science Park 105
1090 GE Amsterdam, Netherlands; \email{arnokret@gmail.com}
\and S. W. Shin: Department of Mathematics, UC Berkeley, Berkeley, CA 94720, USA /\!/ Korea Institute for Advanced Study, 85 Hoegiro,
Dongdaemun-gu, Seoul 130-722, Republic of Korea; \email{sug.woo.shin@berkeley.edu}}


\begin{abstract}
We prove the existence of $\GSpin$-valued Galois representations corresponding to cohomological cuspidal automorphic representations of general symplectic groups over totally real number fields under the local hypothesis that there is a Steinberg component. This confirms the Buzzard--Gee conjecture on the global Langlands correspondence in new cases. As an application we complete the argument by Gross and Savin to construct a rank seven motive whose Galois group is of type $G_2$ in the cohomology of Siegel modular varieties of genus three. Under some additional local hypotheses we also show automorphic multiplicity one as well as meromorphic continuation of the spin $L$-functions.
\end{abstract}

\tableofcontents

\section*{Introduction}
Let $G$ be a connected reductive group over a number field $F$. The conjectural global Langlands correspondence for $G$ predicts a correspondence between certain automorphic representations of $G(\A_F)$ and certain $\ell$-adic Galois representations valued in the $L$-group of $G$. Let us recall from \cite[\S3.2]{BuzzardGee} a rather precise conjecture on the existence of Galois representations for a connected reductive group $G$ over a number field $F$. Let $\pi$ be a cuspidal $L$-algebraic automorphic representation of $G(\A_F)$. (We omit their conjecture in the $C$-normalization, cf. \cite[Conj. 5.40]{BuzzardGee}, but see Theorem \ref{thm:RhoPi0} below.) Denote by $\hat G(\lql)$ the Langlands dual group of $G$ over $\lql$, and by $^L G(\lql)$ the $L$-group of $G$ formed by the semi-direct product of $\hat G(\lql)$ with $\Gal(\ol{F}/F)$. According to their conjecture, for each prime $\ell$ and each field isomorphism $\iota:\C\simeq \lql$, there should exist a continuous representation 
$$
\rho_{\pi,\iota} \colon \Gal(\ol{F}/F)\ra {}^LG(\lql),
$$
which is a section of the projection $^L G(\lql)\ra \Gal(\ol{F}/F)$, such that the following holds: at each place $v$ of $F$ where $\pi_v$ is unramified, the restriction $\rho_{\pi,\iota,v} \colon \Gal(\ol{F}_v/F_v)\ra {}^L G(\lql)$ corresponds to $\pi_v$ via the unramified Langlands correspondence. Moreover $\rho_{\pi,\iota}$ should satisfy other desiderata, cf. Conjecture 3.2.2 of \emph{loc. cit}. For instance at places $v$ of $F$ above $\ell$, the localizations $\rho_{\pi,\iota,v}$ are potentially semistable and have Hodge-Tate cocharacters determined by the infinite components of $\pi$. Note that if $G$ is a split group over $F$, we may as well take $\rho_{\pi,\iota}$ to have values in $\hat G(\lql)$. To simplify notation, we often fix $\iota$ and write $\rho_\pi$ and $\rho_{\pi,v}$ for $\rho_{\pi,\iota}$ and $\rho_{\pi,\iota,v}$, understanding that these representations do depend on the choice of $\iota$ in general.

Our main result confirms the conjecture for general symplectic groups over totally real fields in a number of cases (up to Frobenius semisimplification, meaning that we take the semisimple part in (ii) of Theorem \ref{thm:A} below). We find these groups interesting for two reasons. Firstly they naturally occur in the moduli spaces of polarized abelian varieties and their automorphic/Galois representations have been useful for arithmetic applications (such as the study of $L$-functions, modularity and the Sato-Tate conjecture). Secondly new phenomena (as the semisimple rank grows) make the above conjecture sufficiently nontrivial, stemming from the nature of the dual group of a general symplectic group: e.g. faithful representations have large dimensions and locally conjugate representations may not be globally conjugate.

Let $F$ be a totally real number field. Let $n \geq 2$. Let $\GSp_{2n}$ denote the split general symplectic group over $F$, equipped with similitude character $\simil:\GSp_{2n}\to \G_m$ over $F$. The dual group $\hat \GSp_{2n}$ is the general spin group $\GSpin_{2n+1}$, which we view over $\lql$ (or over $\C$ via $\iota$), admitting the spin representation
$$
\spin \colon \GSpin_{2n+1} \ra \GL_{2^n}.
$$
Consider the following hypotheses on $\pi$.
\begin{itemize}[leftmargin=+1in]
\item[\textbf{(St)}] There is a finite $F$-place $\vst$ such that $\pi_{\vst}$ is the Steinberg representation of $\GSp_{2n}(F_{\vst})$ twisted by a character.
\item[\textbf{(L-coh)}] $\pi_ \infty|\simil|^{-n(n+1)/4}$ is $\xi$-cohomological for an irreducible algebraic representation $\xi=\otimes_{y:F\hra \C} \xi_y$ of the group $(\Res_{F/\Q}\GSp_{2n}) \otimes_\Q \C\simeq \prod_{y:F\hra \C} \GSp_{2n,\C}$ (Definition \ref{def:cohomological} below).
\end{itemize}
A trace formula argument shows that there are plenty of (in particular infinitely many) $\pi$ satisfying (St) and (L-coh), cf. \cite{Clozel-limit}.
Let $S_{\mathrm{bad}}$ denote the finite set of rational primes $p$ such that either $p=2$, $p$ ramifies in $F$, or $\pi_v$ ramifies at a place $v$ of $F$ above $p$.

\begin{thmx}[Theorem \ref{thm:ThmAisTrue}]\label{thm:A}
Suppose that $\pi$ satisfies hypotheses (St), (L-coh). Let $\ell$ be a prime number and $\iota \colon \C \isomto \lql$ a field isomorphism. Then there exists a representation
$$
\rho_{\pi}=\rho_{\pi,\iota} \colon \Gal(\li F/F) \to \GSpin_{2n+1}(\lql),
$$
unique up to $\GSpin_{2n+1}(\lql)$-conjugation, attached to $\pi$ and $\iota$ such that the following hold.
\begin{enumerate}[label=(\textit{\roman*})]
\item\label{thm:Ai} The composition
$$
\Gal(\li F/F) \overset {\rho_{\pi}} \to \GSpin_{2n+1}(\lql) \to \SO_{2n+1}(\lql) \subset \GL_{2n+1}(\lql)
$$
is the Galois representation attached to a cuspidal automorphic $\Sp_{2n}(\A_F)$-{sub\-re\-pre\-sen\-ta\-tion} $\pi^\flat$ contained in $\pi$. Further, the composition
$$
\Gal(\li F/F) \overset {\rho_{\pi}} \to \GSpin_{2n+1}(\lql)/\Spin_{2n+1}(\lql) \simeq \GL_1(\lql)
$$
corresponds to the central character of $\pi$ via class field theory and $\iota$.
\item\label{thm:Aii} For every finite place $v$ which is not above $S_{\mathrm{bad}}\cup \{\ell\}$, the semisimple part of $\rho_{\pi}(\Frob_v)$ is conjugate to $\iota\phi_{\pi_v}(\Frob_v)$ in $\GSpin_{2n+1}(\lql)$, where $\phi_{\pi_v}$ is the unramified Langlands parameter of $\pi_{v}$.
\item\label{thm:Aiii} For every $v|\ell$, the representation $\rho_{\pi, v}$ is de Rham (in the sense that $r \circ \rho_{\pi,v}$ is de Rham for all representations $r$ of $\GSpin_{2n+1,\lql}$). Moreover
\begin{itemize}
\item[(a)] The Hodge-Tate cocharacter of $\rho_{\pi, v}$ is explicitly determined by $\xi$. More precisely, for all $y \colon F \to \C$ such that $\iota y$ induces $v$, we have
$$
\mu_{\tu{HT}}(\rho_{\pi, v}, \iota y) = \iota \mu_{\tu{Hodge}}(\xi_y) - \tfrac{n(n+1)}{4}\simil 
$$
(for $\mu_{\tu{HT}}$ and $\mu_{\tu{Hodge}}$ see Definitions \ref{def:Hodge-Tate} and \ref{def:Hodge} below).
\item[(b)] If $\pi_v$ has nonzero invariants under a hyperspecial (resp. Iwahori) subgroup of $\GSp_{2n}(F_v)$ then either $\rho_{\pi, v}$ or a quadratic character twist is crystalline (resp. semistable).
\item[(c)] If $\ell \notin S_{\mathrm{bad}}$ then $\rho_{\pi, v}$ is crystalline.
\end{itemize}
\item\label{thm:Aiv} For every $v| \infty$, $\rho_{\pi,v}$ is odd (see Definition \ref{def:odd} and Remark \ref{rem:odd} below).
\item\label{thm:Av}
The Zariski closure of the image of $\rho_\pi$ maps onto one of the following three subgroups of $\SO_{2n+1}$:
(a) $\SO_{2n+1}$, (b) the image of a principal $\SL_2$ in $\SO_{2n+1}$, or (c) (only possible when $n=3$) $G_2$ embedded in $\SO_7$.
\item\label{thm:Avi}
If $\rho' \colon \Gal(\li F/F) \to \GSpin_{2n+1}(\lql)$ is any other continuous morphism such that, for almost all $F$-places $v$ where $\rho'$ and $\rho_\pi$ are unramified, the semisimple parts $\rho'(\Frob_v)_{\textup{ss}}$ and $\rho_{\pi}(\Frob_v)_{\textup{ss}}$ are conjugate, then $\rho$ and $\rho'$ are conjugate.
\end{enumerate}
\end{thmx}

Our theorem is new when $n\ge 3$. When $n=2$, a better and fairly complete result without conditions (St) has been known by \cite{Taylor-GSp4, Laumon-GSp4, Weissauer-GSp4, Urban-GSp4, Sorensen-GSp4, Jorza-GSp4, GeeTaibi}. (Often $\pi$ is assumed globally generic in the references but this is no longer necessary thanks to \cite{GeeTaibi}.)

The above theorem in particular associates a weakly compatible system of $\lambda$-adic representations with $\pi$. See also Proposition \ref{prop:weakly-compatible} below for precise statements on the weakly compatible system consisting of $\spin\circ\rho_\pi$. It is worth noting that the uniqueness in (vi) would be false for general $\GSpin_{2n+1}(\lql)$-valued Galois representations in view of Larsen's example in \S\ref{sec:GSpin-valued} below.\footnote{In contrast, for cuspidal automorphic representations of the group $\Sp_{2n}(\A_F)$, the corresponding analogue of statement (vi) \emph{does} hold (cf. Proposition \ref{prop:ConjugacyStdRep}). So the possible failure of (vi) is a new phenomenon for the cuspidal spectrum of the similitude group $\GSp_{2n}(\A_F)$.} Our proof of (vi) relies heavily on the fact that $\rho_\pi$ contains a regular unipotent element in the image, coming from condition (St).

When $n=3$ and $F=\Q$, we employ the strategy of Gross and Savin \cite{GrossSavin} to construct a rank 7 motive over $\Q$ with Galois group of exceptional type $G_2$ in the cohomology of Siegel modular varieties of genus 3. The point is that $\rho_\pi$ as in the above theorem factors through $G_2(\lql)\hra \GSpin_7(\lql)$ if $\pi$ comes from an automorphic representation on (an inner form of) $G_2(\A)$ via theta correspondence. In particular we get yet another proof affirmatively answering a question of Serre in the case of $G_2$, cf. \cite{KLS2,DettweilerReiter,Yun-Motives,Patrikis} for other approaches to Serre's question (none of which uses Siegel modular varieties). Along the way, we also obtain some new instances of the Buzzard--Gee conjecture for a group of type $G_2$. Our result also provides a solid foundation for investigating the suggestion of Gross--Savin that a certain Hilbert modular subvariety of the Siegel modular variety should give rise to the cohomology class for the complement of the rank 7 motive of type $G_2$ in the rank 8 motive cut out by $\pi$, as predicted by the Tate conjecture. See Theorem \ref{thm:G2-A}, Corollary \ref{cor:G2}, and Remark \ref{remark:Tate} below for details. We mention that Magaard and Savin obtained similar results (using a different method) in a new version of \cite{MagaardSavin}.

As another consequence of our theorems, we deduce multiplicity one theorems for automorphic representations for inner forms of $\GSp_{2n, F}$ under similar hypotheses from the multiplicity formula by Bin Xu \cite{Xu}. As his formula suggests, multiplicity one is not always expected when all hypotheses are dropped.

\begin{thmx}[Theorem \ref{thm:AutomorphicMultiplicity}]\label{thm:IntroAutomMult}
Let $\pi$ be a cuspidal automorphic representation of $\GSp_{2n}(\A_F)$ satisfying (St) and (L-coh). The automorphic multiplicity of $\pi$ is equal to $1$.
\end{thmx}

By part (vi) of Theorem \ref{thm:A} asserting uniqueness, we have (a version of) strong multiplicity one for the $L$-packet of $\pi$. In Proposition \ref{prop:JL-for-GSp} we prove a weak Jacquet-Langlands transfer for $\pi$ in Theorem~\ref{thm:IntroAutomMult}. This allows us to propagate multiplicity one from $\pi$ as above to the corresponding automorphic representations on a certain inner form. See Theorem \ref{thm:AutomorphicMultiplicity2} below for details. Note that (weak and strong) multiplicity one theorem for globally generic cuspidal automorphic representations of $\GSp_4(\A_F)$ has been known by Jiang and Soudry \cite{Soudry-GSp4,JiangSoudry-GSp4}.

Our results yield a potential version of the spin functoriality, thus also a meromorphic continuation of the spin $L$-function, for cuspidal automorphic representations of $\GSp_{2n}(\A_F)$ satisfying (St), (L-coh), and the following strengthening of (spin-reg):
\begin{itemize}[leftmargin=+1in]
\item[\textbf{(spin-REG)}] The representation $\pi_v$ is spin-regular at \emph{every} infinite place $v$ of $F$.
\end{itemize}

The last condition means that the Langlands parameter of $\pi_{v}$ maps to a regular parameter for $\GL_{2^n}$ by the spin representation, cf. Definitions \ref{def:spin-reg-param} and \ref{def:SpinRegular} below.
Thanks to the potential automorphy theorem of Barnet-Lamb, Gee, Geraghty, and Taylor \cite[Thm. A]{BLGGT14-PA} it suffices to check the conditions for their theorem to apply to $\spin\circ \rho_\pi$ (for varying $\ell$ and $\iota$). This is little more than Theorem \ref{thm:A}; the details are explained in \S\ref{sec:meromorphic} below.

\begin{thmx}[Corollary \ref{cor:ThmDisTrue}]\label{thm:D}
Under hypotheses (St), (L-coh), and (spin-REG) on $\pi$, there exists a finite totally real extension $F'/F$ (which can be chosen to be disjoint from any prescribed finite extension of $F$ in $\ol{F}$) such that $\spin\circ \rho_\pi|_{\Gal(\li F/F')}$ is automorphic. More precisely,
there exists a cuspidal automorphic representation $\Pi$ of $\GL_{2^n}(\A_{F'})$ such that
\begin{itemize}
\item for each finite place $w$ of $F'$ not above $S_{\mathrm{bad}}$, the representation $\iota^{-1}\spin\circ\rho_\pi|_{W_{F'_w}}$ is unramified and its Frobenius semisimplification is the Langlands parameter for~$\Pi_w$, 
\item at each infinite place $w$ of $F'$ above a place $v$ of $F$, we have $\phi_{\Pi_w}|_{W_\C}\simeq \spin\circ \phi_{\pi_v}|_{W_\C}$.
\end{itemize}
In particular the partial spin $L$-function $L^S(s,\pi,\spin)$ admits a meromorphic continuation and is holomorphic and nonzero in a right half plane.
\end{thmx}

We can be precise about the right half plane in terms of $\pi$: For instance it is given by $\mathrm{Re}(s) \geq 1$ if $\pi$ has unitary central character. Due to the limitation of our method, we cannot control the poles. Before our work, the analytic properties of the spin $L$-functions have been studied mainly via Rankin-Selberg integrals; some partial results have been obtained for $\GSp_{2n}$ for $2\le n\le 5$ by Andrianov, Novodvorsky, Piatetski-Shapiro, Vo, Bump--Ginzburg, and more recently by Pollack--Shah and Pollack, cf. \cite{PS97,Vo97,BG00,PS,Pol}. See \cite[1.3]{Pol} for remarks on spin $L$-functions with further references.

Finally we comment on the hypotheses of our theorems. The statements (ii) and (iii.c) are not optimal in that we exclude a little more than the finite places $v$ where $\pi_v$ is unramified. This is due to the fact that the Langlands--Rapoport conjecture has been proved by \cite{KisinPoints} at a $p$-adic place only when $p>2$ and the defining data (including the level) are unramified at \emph{all} $p$-adic places\footnote{When $F=\Q$, our argument uses Siegel modular varieties, so we may allow $p=2$ by appealing to Kottwitz's work on PEL-type Shimura varieties instead of \cite{KisinPoints}.}. Condition (L-coh) is essential to our method but it is perhaps possible to prove Theorem \ref{thm:A} under a slightly weaker condition that $\pi$ appears in the coherent cohomology of our Shimura varieties.

The rather strong condition (spin-REG) in Theorem \ref{thm:D} is necessary due to the current limitation of potential automorphy theorems. Condition (St) in Theorem \ref{thm:A} is believed to be superfluous. Significant work and ideas would be needed to get around it. On the other hand, (St) is harmless to assume for local applications, which we intend to pursue in future work.

\smallskip
\noindent\textbf{Idea of proof.}
Our proof of Theorem \ref{thm:A} relies crucially on Arthur's book \cite{ArthurBook}. His results used to be conditional on the stabilization of the twisted trace formula, whose proof has been completed by Moeglin and Waldspurger; see \cite{MoeglinWaldspurgerStabilization1, MoeglinWaldspurgerStabilization2} for the last one. Thus Arthur's results are essentially unconditional.\footnote{Strictly speaking, Arthur postponed some technical details in harmonic analysis to future articles ([A25, A26, A27] in the bibliography of \cite{ArthurBook}), which have not appeared yet. The weighted fundamental lemma has been proved by Chaudouard--Laumon but the proof has appeared for split groups thus far.} Since Theorems \ref{thm:IntroAutomMult} and \ref{thm:D} depend in turn on Theorem \ref{thm:A}, all our results count on Arthur's book.

Let us sketch our proof with a little more detail. We prove Theorem \ref{thm:A} by combining two approaches: (M1) Shimura varieties for inner forms of $\GSp_{2n}$, and (M2) Lifting to $\GSpin_{2n+1}(\lql)$ the $\SO_{2n+1}(\lql)$-valued Galois representations as constructed by the works of Arthur and Harris-Taylor.

\smallskip
\noindent\textbf{First Method.} Consider the inner form $G$ over $F$ of $G^* := \GSp_{2n, F}$, such that the local group $G_v$ is
$$
\begin{cases}
\tu{non-split} & \tu{ if } v = \vst \tu{ and $[F=\Q]$ is even}, \cr
\tu{anisotropic modulo center} & \tu{ if } v| \infty,~ v\neq v_0, \cr
\tu{split} & \tu{otherwise}.
\end{cases}
$$
We consider Shimura varieties arising from the group $\Res_{F/\Q} G$ (and the choice of $X$ as in \S\ref{sect:ShimuraDatum} below). Note that (the $\Q$-points of) $\Res_{F/\Q} G$ has factor of similitudes in $F^\times$, as opposed to $\Q^\times$, and that our Shimura varieties are not of PEL type (when $F \neq \Q$) but of abelian type. This should already be familiar for $n = 1$ (though we assume $n\ge 2$ for our main theorem to be interesting), where we obtain the usual Shimura curves, cf. \cite{CarayolMauvaiseReduction}. In case $F = \Q$ our Shimura varieties are the classical Siegel modular varieties.

The idea is to consider (the semisimplification of) the compactly supported \'etale cohomology
$$
\uH^i_\tu{c}(\Sh_K, \cL_{\iota\xi}) = \varinjlim_{K \subset G(\A_F^ \infty)} \uH^i_{\tu{c}}(\Sh_K, \cL_{\iota\xi}),\quad i\ge 0,
$$
where $\cL_{\iota\xi}$ is the $\ell$-adic local system attached to some irreducible complex representation $\xi$ of $G$ via $\iota$. Then $\uH^i_\tu{c}(\Sh_K, \cL_{\iota\xi})$ has an action of $G(\A_F^ \infty) \times \Gal(\li F/F)$; and one hopes to prove that through this action the module $\uH^i_\tu{c}(\Sh_K, \cL_{\iota \xi})$ realizes the Langlands correspondence. In particular, one tries to attach to a cuspidal automorphic representation $\pi$ of $\GSp_{2n}(\A_F)$ the following virtual Galois representation (see also Equation \eqref{eq:GrothendieckGp})
$$
\pi \rightsquigarrow \pi' \mapsto H_\pi \bydef \sum_{i\ge0}(-1)^i [\Hom_{G(\A_F^ \infty)}(\pi^{\prime, \infty}_{\lql}, \uH^i_{\tu{c}}(\Sh_K, \cL_{\iota \xi})_{\tu{ss}})],
$$
where the first mapping is a weak transfer of $\pi$ from $G^*$ to $G$. (We need to twist $\pi$ by $|\cdot|^{-n(n+1)/4}$ to have $H_\pi\neq 0$ but this twist will be ignored in the introduction.) The subscript $(\cdot)_{\tu{ss}}$ denotes the semisimplification as a $G(\A_F^ \infty)$-module.

The construction of $H_\pi$ has two issues. Firstly there are the usual issues coming from endoscopy. Our assumption (St) circumvents this difficulty (and helps us at other places). Secondly, even without endoscopy, one does not expect $H_\pi$ to realize the representation $\rho_\pi$ itself; rather it realizes (up to dual, sign, twist, and multiplicity) the composition
$$
\Gal(\li F/F) \overset {\rho_\pi} \to \GSpin_{2n+1}(\lql) \overset {\spin} \to \GL_{2^n}(\lql).
$$
In particular if one wants to use $H_\pi$ to construct $\rho_\pi$, one has to show that $r \colon \Gal(\li F/F) \to \GL_{\lql}(H_\pi) \simeq \GL_{2^n}(\lql)$ has (up to conjugation) image in the group $\GSpin_{2n+1}(\lql) \subset \GL_{2^n}(\lql)$. With point counting methods it can be shown that this is true for the Frobenius elements, but since these elements are only defined up to $\GL_{2^n}(\lql)$-conjugation, we are unable to deduce directly that the entire representation $r$ has image in $\GSpin_{2n+1}(\lql)$. More information seems to be required.

\smallskip

\noindent\textbf{Second Method.} Consider a continuous representation $\rho \colon \Gal(\li F/F) \to \SO_{2n+1}(\lql)$, and the exact sequence
$$
1 \to \lql^\times \to \GSpin_{2n+1}(\lql) \to \SO_{2n+1}(\lql) \to 1.
$$
Using the theorem of Tate that $\uH^2(F, \Q/\Z)$ vanishes, it is not hard to show that $\rho$ has a continuous lift $\wt \rho \colon \Gal(\li F/F) \to \GSpin_{2n+1}(\lql)$. In particular one could try to attach to $\pi$ a (non-unique) automorphic $\Sp_{2n}(\A_F)$-{sub\-re\-pre\-sen\-ta\-tion} $\pi^\flat \subset \pi$, and then to $\pi^\flat$ the Galois representation $\rho_{\pi^\flat} \colon \Gal(\li F/F) \to \SO_{2n+1}(\lql)$, constructed by the combination of results of Arthur and others (Theorem \ref{thm:ExistGaloisRep}). One now hopes to construct $\rho_{\pi} \colon \Gal(\li F/F) \to \GSpin_{2n+1}(\lql)$ as a twist $\wt {\rho_{\pi^\flat}} \otimes \chi$ for some continuous character $\chi \colon \Gal(\li F/F) \to \lql^\times$. However it is a priori unclear where $\chi$ should come from. (The central character of $\pi$ only determines the square of $\chi$.)

\smallskip

\noindent \textbf{Construction.} \quad We find the character $\chi$ by comparing $\wt {\rho_{\pi^\flat}}$ with the representation $H_\pi$. Consider the diagram
$$
\xymatrix{
\Gal(\li F/F)\ar@/^1.5pc/[rrr]_{\rho_1} \ar@/^2.5pc/[rrr]_{\rho_2}\ar[rr]_{\tilrhoflat\quad \quad}\ar@/_1.5pc/[drr]_{\rho_{\pi^\flat}}&&\GSpin_{2n+1}(\lql)\ar[d]\ar[r]_{\quad\spin}&\GL_{2^n}(\lql)\ar[d]&\cr&&\SO_{2n+1}(\lql)\ar[r]_{\overline{\spin}\quad}&\PGL_{2^n}(\lql)\cr }
$$
where $\rho_1 = \spin \circ \wt{\rho_{\pi^\flat}}$ for some choice of lift $\wt{\rho_{\pi^\flat}}$ of $\rho_{\pi^\flat}$, and we construct $\rho_2$ from the representation $H_\pi$. We show in three steps that $\rho_1$ and $\rho_2$ are conjugate up to twisting by the sought after character $\chi$.

\smallskip

\noindent\textbf{Step 1}: Show connectedness of image.
\quad More precisely we show that the image of $\rho_{\pi^\flat}$ (thus also the image of $\overline{\spin} \circ\rho_{\pi^\flat}$) has connected Zariski closure. Because $\pi$ is a twist of the Steinberg representation at a finite place, we can assure that $\rho_{\pi^\flat}|_{\Gal(\li F/E)}$ has a regular unipotent element $N$ in its image. It then follows from results of Saxl--Seitz that the reductive subgroups of $\SO_{2n+1}(\lql)$ containing $N$ are connected (see Proposition \ref{prop:DetermineZariskiImage}).

\smallskip

\noindent\textbf{Step 2}: Construct $\rho_2$. \quad We compare the point counting formula for $(\Res_{F/\Q} G, \iH)$ with the Arthur-Selberg trace formula for $G/F$. Since the datum $(\Res_{F/\Q} G, \iH)$ is not of PEL type (unless $F=\Q$), the classical work of Kottwitz \cite{KottwitzPoints, KottwitzAnnArbor} does not apply. Instead we use the counting point formula as derived in \cite{KSZ} from Kisin's recent proof of the Langlands--Rapoport conjecture for Shimura data of abelian type, so in particular for $(\Res_{F/\Q} G, \iH)$. Consequently, we have $H_{\pi}|_{\Gal(\li F_v/F_v)} = a \cdot \spin \circ \iota\phi_{\pi_v}$ at the unramified places $v$, where $\phi_{\pi_v}$ is the unramified $L$-parameter of $\pi_v$ and $a \in \Z_{>0}$ is essentially the automorphic multiplicity of $\pi$. (In fact, we show that $a=1$ together with Theorem \ref{thm:IntroAutomMult} but only after the construction of $\rho_\pi$ is done. See \S\ref{sect:AutomorphicMultiplicity} below.)

\smallskip

\noindent\textbf{Step 3}: Produce $\chi$. \quad To do this we prove:

\begin{lemmastar}\label{lem:Intro}
Let $r_1, r_2 \colon \Gal(\li F/F) \to \GL_m(\lql)$ two continuous representations, which are unramified almost everywhere, $r_1$ is has Zariski connected image modulo center, and for almost all $F$-places
$$
\overline{r_1(\Frob_v)_{\ssimple}} \quad \tu{is conjugate to} \quad \overline{ r_2(\Frob_v)_{\ssimple}} \quad \tu{in} \quad \PGL_m(\lql).
$$
Then $r_1 \simeq r_2 \otimes \chi$ for a continuous character $\chi \colon \Gal(\li F/F) \to \lql^\times$.
\end{lemmastar}

By counting points on Shimura varieties we control $\rho_2$ at the unramified places. By a different argument $\rho_1$ is, up to scalars, also controlled at the unramified places. Hence the lemma applies, and allows us to find a character $\chi$ such that $\rho_2 \simeq \rho_1\otimes \chi$.

To prove Theorem \ref{thm:A} we define
$$
\rho_\pi \bydef \wt {\rho_{\pi^\flat}} \otimes \chi\colon \Gal(\li F/F) \to \GSpin_{2n+1}(\lql),
$$
and check that $\rho_\pi$ satisfies the desired properties stated in the theorem.

\smallskip

{\small
\noindent \textbf{Acknowledgments: } \quad
The authors wish to thank Fabrizio Andreatta, Alexis Bouthier, Laurent Clozel, Wushi Goldring, Christian Johansson, Kai-Wen Lan, Brandon Levin, David Loeffler, Judith Ludwig, Giovanni Di Matteo, Sophie Morel, Stefan Patrikis, Peter Scholze, Benoit Stroh, Richard Taylor, Olivier Ta\"\i bi, Jack Thorne, Jacques Tilouine, David van Overeem and David Vogan for helpful discussions and answering our questions. A.K. thanks S.W.S. and the Massachusetts Institute of Technology for the invitation to speak in the number theory seminar, which led to this work. Furthermore he thanks the Institute for Advanced Study, the Mathematical Sciences Research Institute, the Max Planck Institute for Mathematics, Korea Institute for Advanced Study, and the University of Amsterdam, where this work was carried out. During his time at the IAS (resp. MSRI) A.K. was supported by the NSF under Grant No. DMS 1128155 (resp. No. 0932078000). S.W.S. was partially supported by NSF grant DMS-1449558/1501882 and a Sloan Fellowship. Both authors are grateful to an anonymous referee for an extensive list of comments and simplifications. In particular he/she explained Proposition \ref{prop:WeaklyAcceptable} to us and suggested a significant simplification based on the fact that the $\SO_{2n+1}$-valued Galois representation $\rho_{\pi^\flat}$ has connected image (without extra regularity hypotheses on the component of $\pi$ at infinity). Following the suggestion we dispensed with the eigenvariety and patching arguments in an earlier version (\texttt{arXiv:1609.04223v1}) of our paper.}
We thank another anonymous reviewer for sending us corrections and suggesting improvements on the next version (\texttt{arXiv:1609.04223v2}), especially in \S\ref{sec:GSpin-valued}.

\section*{Notation}\label{sect:Notation}
We fix the following notation and convention.
\begin{itemize}
\item ``Almost all'' always means ``all but finitely many''.
\item $n \geq 2$ is an integer.
\item $F$ is a totally real number field, embedded into $\C$.
\item $\mathcal{O}_F$ is the ring of integers of $F$.
\item $\mathcal{O}_F[1/S]$ is the localization of $\cO_F$ with finite primes in $S$ inverted, where $S$ is a finite set of places of $F$. (Finite places are used interchangeably with finite primes.)
\item $\A_F$ is the ring of ad\`eles of $F$, i.e. $\A_F := (F \otimes_\Q \R) \times (F \otimes_\Z \Zhat)$.
\item If $\Sigma$ is a finite set of $F$-{pla\-ces}, then $\A_F^\Sigma \subset \A_F$ is the ring of ad\`eles with trivial components at the places in $\Sigma$, and $F_\Sigma := \prod_{v \in \Sigma} F_v$; $F_ \infty:=F\otimes_\Q \R$.
\item If $p$ is a prime number, then $F_p := F \otimes_{\Q} \qp$.
\item $\ell$ is a fixed prime number (different from $p$).
\item $\lql$ is a fixed algebraic closure of $\ql$, and $\iota \colon \C \isomto \lql$ is an isomorphism.
\item For each prime number $p$ we fix the positive root $\sqrt p \in \R_{>0} \subset \C$. From $\iota$ we then obtain a choice for $\sqrt p \in \lql$. Thus, if $q$ is a prime power, then $q^x$ is well-defined in $\C$ as well as in $\lql$ for all half integers $x \in \tfrac 12 \Z$ (e.g. $q_v^{n(n+1)/4}$ in Corollary \ref{cor:lgc-unramified})
\item If $\pi$ is a representation on a complex vector space then we set $\iota\pi:=\pi\otimes_{\C,\iota} \lql$. Similarly if $\phi$ is a local $L$-parameter of a reductive group $G$ so that $\phi$ maps into $^L G(\C)$ then $\iota\phi$ is the parameter with values in $^L G(\lql)$ obtained from $\phi$ via $\iota$.
\item $\Gamma := \Gal(\li F/F)$ is the absolute Galois group of $F$.
\item $\Gamma_v := \Gal(\li F_v/F_v)$ is (one of) the local Galois group(s) of $F$ at the place $v$.
\item $\cV_ \infty:= \Hom_\Q(F, \R)$ is the set of infinite places of $F$.
\item $c_v \in \Gamma$ is the complex conjugation (well-defined as a conjugacy class) induced by any embedding $\li{F}\hookrightarrow \C$ extending $v \in \cV_ \infty$.
\item If $G$ is a locally profinite group equipped with a Haar measure, then we write $\cH(G)$ for the \emph{Hecke algebra} of locally constant, complex valued functions with compact support. We write $\cH_{\lql}(G)$ for the same algebra, but now consisting of $\lql$-{va\-lued} functions.
\item We normalize parabolic induction by the half power of the modulus character as in \cite[1.8]{BZ77}, so as to preserve unitarity.
The Satake transform and parameters are normalized similarly, e.g. as in \cite[2.2]{BuzzardGee}.
\item We normalize class field theory so that geometric Frobenius elements correspond to local uniformizers. Our normalization of the local Langlands correspondence for $\GL_n$ is the same as in \cite{HarrisTaylor}.
\item If $H/\lql$ is a reductive group, a group-valued representation $\Gamma \to H(\lql)$ means, by definition, a \emph{continuous} group morphism for the Krull topology on $\Gamma$ and the $\ell$-adic topology on $H(\lql)$. Similarly every character is assumed to be continuous throughout the paper.
\end{itemize}

\subsection*{The (general) symplectic group}
Write $A_n$ for the $n \times n$-matrix with zeros everywhere, except on its antidiagonal, where we put ones. Write $J_n := \kleinvierkant \ {A_n}{-A_n}\ \in \GL_{2n}(\Z)$. We define $\GSp_{2n}$ as the algebraic group over $\Z$, such that for all rings $R$,
$$
\GSp_{2n}(R) = \{g \in \Gl_{2n}(R) \ |\ {}^\tu{t} g \cdot J_n \cdot g = x \cdot J_n \tu{ for some } x \in R^\times\}.
$$
The \emph{factor of similitude} $x \in R^\times$ induces a morphism $\simil \colon \Gsp_{2n} \to \Gm$. Write $\TGSp \subset \GSp_{2n}$ for the diagonal maximal torus. Then $X^*(\TGsp) = \bigoplus_{i=0}^n \Z e_i$ where
\begin{align*}
e_i \colon & \diag(a_1, \ldots, a_n, c a_n^{-1}, \ldots, c a_1^{-1}) \mapsto a_i \quad \quad (i > 0), \cr
e_0 \colon & \diag(a_1, \ldots, a_n, c a_n^{-1}, \ldots, c a_1^{-1}) \mapsto c.
\end{align*}
We let $\BGSp \subset \Gsp_{2n}$ be the upper triangular Borel subgroup. We have the following corresponding simple roots and coroots,
\begin{align*}
\alpha_1 &= e_1 - e_2, ~\ldots, ~\alpha_{n-1} = e_{n-1} - e_n, ~\alpha_n = 2e_n - e_0 \in X^*(\TGsp), \cr
\alpha_1^\vee &= e^*_1 - e^*_2, ~\ldots, ~\alpha_{n-1}^\vee = e^*_{n-1} - e^*_n, ~\alpha_n^\vee = e^*_n \in X_*(\TGsp).
\end{align*}
We define $\Sp_{2n} = \Ker(\simil)$. Write $\BSp = \BGSp \cap \Sp_{2n}$ and $\TSp = \TGsp \cap \Sp_{2n}$.

\subsection*{The (general) orthogonal group}
Let $m \in \Z_{\geq 1}$ and let $\GO_{m}$ be the algebraic group over $\Q$ such that for all $\Q$-algebras $R$,
$$
\GO_{m}(R) = \{g \in \GL_{m}(R)\ |\ {}^{\tu t} g \cdot A_{m} \cdot g = x \cdot A_{m} \tu{ for some $x \in R^\times$}\},
$$
where $A_m$ is the $m \times m$-antidiagonal unit matrix. We have the factor of similitude $\simil \colon \GO_{2n+1} \to \Gm$ and put ${\tu{O}}_{m} = \ker(\simil)$ and $\SO_{m} = \tu{O}_{m} \cap \SL_{m}$. Let $\TGO \subset \GO_{m}$, $\TSO \subset \SO_{m}$ be the diagonal tori. We write $\std \colon \GO_{m} \to \GL_{m}$ for the standard representation. If $m = 2n+1$ is odd, the root datum of $\SO_{2n+1}$ is dual to $\Sp_{2n}$. In particular, we identify $X_*(\TSO) = X^*(\TSp)$.

\subsection*{The (general) spin group} Consider the symmetric form
$$
\langle x, y \rangle = x_1 y_{2n+1} + x_2 y_{2n} + \cdots + x_{2n+1} y_1 = {}^{\tu{t}} x \cdot A_{2n+1} \cdot y
$$
on $\Q^{2n+1}$. The associated quadratic form is $Q(x) = x_1 x_{2n+1} + x_2 x_{2n} + \cdots + x_{2n+1} x_1$. Let $C$ be the Clifford algebra associated to $(\Q^{2n+1}, Q)$. It is equipped with an embedding $\Q^{2n+1} \subset C$ which is universal for maps $f \colon \Q^{2n+1} \to A$ into associative rings $A$ satisfying $f(x)^2 = Q(x)$ for all $x \in \Q^{2n+1}$.
Let $b_1, \ldots, b_{2n+1}$ be the standard basis of $\Q^{2n+1}$. The products $B_I = \prod_{i \in I} b_i$ for $I \subset \{1, 2, \ldots, 2n+1\}$ form a basis of $C$. The algebra $C$ has a $\Z/2\Z$-grading, $C = C^+ \oplus C^-$, induced from the grading on the tensor algebra. On the Clifford algebra $C$ we have an unique anti-involution $*$ that is determined by $(v_1 \cdots v_r)^* = (-1)^r v_r \cdots v_1$ for all $v_1, \ldots, v_r \in V$. We define for all $\Q$-algebras $R$,
$$
\GSpin_{2n+1}(R) = \{g \in (C^+ \otimes R)^\times \ |\ g \cdot R^n \cdot g^* = R^n \}.
$$
The spinor norm on $C$ induces a character $\mathcal N \colon \GSpin_{2n+1} \to \Gm$. The action of the group $\GSpin_{2n+1}$ stabilizes $\Q^{2n+1} \subset C$ and we obtain a surjection $q' \colon \GSpin_{2n+1} \to \GO_{2n+1}$. We write $q$ for the surjection $\GSpin_{2n+1} \surjects \SO_{2n+1}$ obtained from $q'$.

We write $\TGSpin \subset \GSpin_{2n+1}$ for the torus $(q')^{-1}(\TGSpin)$. We then have
\begin{align*}
X^*(\TGSpin) &= \Z e^*_0 \oplus \Z e^*_1 \oplus \cdots \oplus \Z e^*_n (= X_*(\TGSp)), \cr
X_*(\TGSpin) &= \Z e_0 \oplus \Z e_1 \oplus \cdots \oplus \Z e_n (= X^*(\TGsp)).
\end{align*}
The group $\GSpin_{2n+1}(\C)$ is dual to $\GSp_{2n}(\C)$. The character $\simil:\GSp_{2n}\ra \G_m$ induces a central embedding $\G_m\ra \GSpin_{2n+1}$, still denoted by $\simil$.

\subsection*{The spin representation}
Let $k$ be an algebraically closed field of characteristic zero. Write $W:=\oplus_{i=1}^n k\cdot b_i$ and $\bigwedge^\bullet W$ for the exterior algebra of $W$. Then $W$ is an $n$-dimensional isotropic subspace of $(\Q^{2n+1},Q)$. We have $\End(\bigwedge^\bullet W) \simeq C^+$, and hence $\bigwedge^\bullet W$ is a $2^n$-dimensional representation of $\GSpin_{2n+1, k}$, called the \emph{spin representation} (cf. \cite[(20.18)]{FultonHarris}). The composition of $\spin$ with $\GL_{2^n,k} \to \PGL_{2^n,k}$ induces a morphism $\overline {\spin} \colon \SO_{2n+1,k} \to \PGL_{2^n,k}$.

\begin{lemma}\label{lem:similitude-of-spin-rep}
When $n$ mod 4 is 0 or 3 (resp. 1 or 2), there exists a symmetric (resp. symplectic) form on the $2^n$-dimensional vector space underlying the spin representation such that the form is preserved under $\GSpin_{2n+1}(k)$ up to scalars. The resulting map $\GSpin_{2n+1}\ra \GO_{2^n}$ (resp. $\GSpin_{2n+1}\ra \GSp_{2^n}$) over $k$ followed by the similitude character of $\GO_{2^n}$ (resp. $\GSp_{2^n}$) coincides with the spinor norm $\mathcal N$.
\end{lemma}

\begin{proof}
We may identify the $2^n$-dimensional space with $\bigwedge^\bullet W$. Write $*$ for the main involution on $C^+_k$ as well as on $\bigwedge^\bullet W$. Given $s,t \in \bigwedge^\bullet W$, write $\beta(s,t) \in k$ for the projection of $s^*\wedge t \in \bigwedge^\bullet W$ onto $\bigwedge^n W = k$. It is elementary to check that $\beta(\cdot,\cdot)$ is symmetric if $n$ mod 4 is 0 or 3 and symplectic otherwise, cf. \cite[Exercise 20.38]{FultonHarris}. Now let $x \in \GSpin_{2n+1}(k)$, also viewed as an element of $C^+_k$. Note that $x^*x \in k^\times$ is the spinor norm of $x$. Then $\beta(xs,xt)=(xs)^* \wedge (xt) = (x^*x) s\wedge t = x^*x \beta(s,t)$, completing the proof.
\end{proof}

\section{Conventions and recollections}\label{sec:conventions}

Let $G/\lql$ be a reductive group, $T \subset G$ a maximal torus and $W$ its Weyl group in $G$.
Recall that by highest weight theory the trace characters of irreducible representations form a basis of the algebra $\cO(T/W)$.
We call a set $S$ of representations of $G/\lql$ a \emph{fundamental set} if the trace characters of all exterior powers of the representations in $S$ generate the algebra $\cO(T/W)$ of global sections of the variety $T/W$. Here are some examples of such $S$, where we define the \emph{standard representation} $\std$ of $\GSpin_m$ ($m$ even or odd) to be the composition $\GSpin_m \surjects \GSO_m \hookrightarrow \GL_m$. For the group $G_2$, write $\std$ and $\ad$ for the irreducible $7$-dimensional and $14$-dimensional (adjoint) representations, respectively.
\begin{center}
\begin{tabular}{|c||c|c|c|c|c|c|c|}\hline
$G$ &	 $\GL_n$ &		$\GSp_{2n}$ & 		$\GSO_{2n+1}$ &		$\GO_{2n}$	&	$\GSpin_{2n+1}$ &	$\GSpin_{2n}$	& $G_2$ \\\hline
$S$	& $\std$ & $\std, \tu{sim}$ & $\std, \tu{sim}$ & $\std, \tu{sim}$ & $\spin, \std, \cN$ & $\spin^+, \spin^-, \std, \cN$ & $\std$, $\ad$
\\ \hline
\end{tabular}
\end{center}
\begin{proof}[Justification of Table] The justification for the groups $G = \GO_{2n}$ and $\GSO_{2n+1}$ follows from appendix B. Notice that the remaining groups are all connected. The fundamental representations of a simply connected semisimple group $G$ are those irreducible representations whose highest weights (for some Borel) are dual to the basis of coroots of $G$; the trace characters of fundamental representations generate $\cO(T/W)$ as an algebra.
Exercises V.28--30 and 33 of \cite[p. 344]{KnappLieGroupsBeyond} show that for the Lie algebras ${\mathfrak s \mathfrak l}_n$, ${\mathfrak s \mathfrak o}_{2n+1}$, ${\mathfrak s \mathfrak o}_{2n}$ and ${\mathfrak g}_2$ the representations $\{\wedge^i \std \ (1 \leq i \leq n)\}$, $\{\wedge^i \std\ (1 \leq i \leq n-1), \spin\}$, $\{\wedge^i \std \ (1 \leq i \leq n-2), \spin^+, \spin^-\}$ and $\{\std, \ad\}$ (respectively) are the fundamental representations of these Lie algebras. The statement for the corresponding semisimple simply connected groups follow from this. It is then routine to derive the sets $S$ listed above for the groups $\GL_n$, $\GSpin_{2n+1}$, $\GSpin_{2n}$ and $G_2$. This leaves us with the group $\GSp_{2n}$. In this case the representations $\wedge^i \std$ of ${\mathfrak s \mathfrak p}_{2n}$ are reducible, but it follows from the maps $\varphi_k$ and theorem 17.5 in \cite[p. 260]{FultonHarris} that they generate the fundamental representations of ${\mathfrak s \mathfrak p}_{2n}$.
\end{proof}

\begin{lemma}\label{lem:GaugerSteinberg}
Let $g_1, g_2 \in G(\lql)$ be two semisimple elements. Then $g_1 \sim g_2 \in G(\lql)$ if and only if $\rho(g_1) \sim \rho(g_2) \in \GL(V)$ for all $(\rho, V) \in S$.
\end{lemma}
\begin{proof}
``$\Rightarrow$'' is obvious. We prove ``$\Leftarrow$''.
Since $T/W$ is an algebraic variety, two $\lql$-points $x_1, x_2$ of it are the same if and only if $f(x_1) = f(x_2)$ for the algebraic functions $f \in \cO(T/W)$. Since $\rho(g_1) \sim \rho(g_2)$ for all $\rho \in S$, we also have $\Tr \bigwedge^n \rho(g_1) = \Tr \bigwedge^n \rho(g_2)$ for all $n \geq 1$. Since the functions $f = \Tr \bigwedge^n \rho$ in $\cO(T/W)$ generate this algebra the images of $g_1$ and $g_2$ in $T/W$ are the same.
\end{proof}

\begin{remark}
Gauger \cite{Gauger} and Steinberg \cite[Thm. 3]{SteinbergConjugacy} proved Lemma \ref{lem:GaugerSteinberg} also for non-semisimple elements $g_1, g_2 \in G(\lql)$ and (under various assumptions) algebraically closed fields of characteristic $p \geq 0$.
\end{remark}

Let $r_1, r_2 \colon \Gamma \to G(\lql)$ be two semisimple representations that are unramified at almost all places. (As remarked before, all representations are continuous by convention.)

\begin{lemma}\label{lem:EquivalenceLocalConjugacy}
The following statements are equivalent:
\begin{enumerate}
\item for a Dirichlet density one set of finite $F$-places $v$ where $r_1, r_2$ are unramified we have $r_1(\Frob_v)_{\ssimple} \sim r_2(\Frob_v)_{\ssimple} \in G(\lql)$;
\item there exists a dense subset $\Sigma \subset \Gamma$ such that for all $\sigma \in \Sigma$ we have $r_1(\sigma)_{\ssimple} \sim r_2(\sigma)_{\ssimple} \in G(\lql)$;
\item for all $\sigma \in \Gamma$ we have $r_1(\sigma)_{\ssimple} \sim r_2(\sigma)_{\ssimple} \in G(\lql)$;
\item for all linear representations $\rho \colon G \to \GL_N$ the representations $\rho \circ r_1$ and $\rho \circ r_2$ are isomorphic.
\item for a fundamental set of irreducible linear representations $\rho \colon G \to \GL_N$ the representations $\rho \circ r_1$ and $\rho \circ r_2$ are isomorphic.
\end{enumerate}
\end{lemma}
\begin{proof}
(3) $\Rightarrow$ (1) is tautological, (1) $\Rightarrow$ (2) follows from the Chebotarev density theorem, and
(2) $\Rightarrow$ (3) follows from the continuity of the map $G(\lql) \to (T/W)(\lql)$ taking the semisimple part. The implication (3) $\Rightarrow$ (4) follows from the Brauer-Nesbitt theorem, (4) $\Rightarrow$ (5) is obvious, and (5) $\Rightarrow$ (3) is Lemma \ref{lem:GaugerSteinberg}.
\end{proof}

\begin{definition}\label{def:locconj}
If one of the conditions in Lemma \ref{lem:EquivalenceLocalConjugacy} holds, then $r_1$ and $r_2$ are said to be \emph{locally conjugate}, and we write $r_1 \approx r_2$.
\end{definition}

\begin{definition}\label{def:DefSuffRegular}
Let $T$ be a maximal torus in a reductive group $G$ over an algebraically closed field. A weight $\nu \in X^*(T)$ is \emph{regular} if $\langle \alpha^\vee, \nu \rangle \neq 0$ for all coroots $\alpha^\vee$ of $T$ in $G$.
\end{definition}

\begin{definition}\label{def:spin-reg-param}
Let $\phi \colon W_{\R} \to \GSpin_{2n+1}(\C)$ be a Langlands parameter. Denote by $T$ the diagonal maximal torus in $\GL_{2^n}$ and by $\hat T$ its dual torus. We have $W_\C=\C^\times \subset W_{\R}$. The composition
$$
\C^\times \subset W_{\R} \to \GSpin_{2n+1}(\C) \overset {\tu{spin}} \to \GL_{2^n}(\C)
$$
is conjugate to the cocharacter $z\mapsto \mu_1(z)\mu_2(\ol{z})$ given by some $\mu_1,\mu_2 \in X_*(\hat T)\otimes_\Z \C=X^*(T)\otimes_\Z \C$ such that $\mu_1-\mu_2 \in X^*(T)$. Then $\phi$ is \emph{spin-regular} if $\mu_1$ is regular (equivalently if $\mu_2$ is regular; note that $\mu_1$ and $\mu_2$ are swapped if $\spin\circ \phi$ is conjugated by the image of an element $j \in W_\R$ such that $j^2=-1$ and $jwj^{-1}=\ol{w}$ for $w \in W_\C$).\end{definition}

\begin{definition}\label{def:SpinRegular}
An automorphic representation $\pi$ of $\GSp_{2n}(\A_F)$ is \emph{spin-regular} at $v_ \infty$ if the Langlands parameter of the component $\pi_{\vinfty}$ is spin-regular.
\end{definition}

Let $H$ be a connected reductive group over $\lql$ for the following two definitions (which could be extended to disconnected reductive groups). Let $\fkh_{\der}$ denote the Lie algebra of its derived subgroup. Write $c$ for the nontrivial element of $\Gal(\C/\R)$.

\begin{definition}[cf. \cite{GrossOdd}]\label{def:odd}
A representation $\rho:\Gal(\C/\R) \to H(\lql)$ is \emph{odd} if the trace of $c$ on $\fkh_{\der}$ through the adjoint action of $\rho(c)$ is equal to $-\mathrm{\rank}(\fkh_{\der})$.
\end{definition}

We remark that the mapping $\Gl_1 \times \SO_{2n+1}\to \GO_{2n+1}$ is an isomorphism, in particular the latter is connected.

\begin{lemma}\label{lem:odd-GO}
Let $\rho^\flat\colon \Gal(\C/\R)\to \GO_{2n+1}(\lql)$ be a representation. Write $\rho^\sharp \colon \Gamma\to \GL_{2n+1}(\lql)$ for the composition of $\rho^\flat$ with the standard embedding. If $\Tr\rho^\sharp(c) \in \{\pm 1\}$ then $\rho^\flat$ is odd.
\end{lemma}
\begin{proof}
We may choose a model for the Lie algebra of $\SO_{2n+1}$ to consist of $X \in \GL_{2n+1}$ such that $X+A_{2n+1}^{\tu{t}} X A_{2n+1}^{-1}=0$. Such an $X=(x_{i,j})$ is characterized by the condition $x_{i,j}+x_{n+1-j,n+1-i}=0$ for every $1\le i,j\le 2n+1$. Write $t:=\rho^\flat(c)$. By conjugation and multiplying with $-1 \in \GL_{2n+1}$ if necessary, we can assume that $t$ is in the diagonal maximal torus in $\SO_{2n+1}$ (not only in $\GO_{2n+1}$, using the fact that the latter is the product of $\SO_{2n+1}$ with center) of the form $\diag(1_a,-1_b,1,-1_b,1_a)$, where $0\le a,b\le n$ and $a+b=n$. Since $\Tr \rho^\sharp(c_v) \in \{\pm 1\}$ we have $a=b$ if $n$ is even and $b-a=1$ if $n$ is odd. Now an explicit computation shows that the trace of the adjoint action of $t$ on $\Lie(\SO_{2n+1})$ has trace
$2 (a-b)^2 +2(a-b)-n$, which is equal to $-n=-\rank(\SO_{2n+1})$ in all cases.
\end{proof}

Let $K$ be a finite extension of $\Q_\ell$. Fix its algebraic closure $\ol{K}$ and write $\hat {\ol K}$ for its completion.

\begin{definition}[{cf. \cite[2.4]{BuzzardGee}}]\label{def:Hodge-Tate}
Let $\rho:\Gal(\ol{K}/K)\to H(\lql)$ be a representation. We say that $\rho$ is crystalline/semistable/de Rham/Hodge-Tate if for some (thus every) faithful algebraic representation $\xi: H\to \GL_N$ over $\lql$, the composition $\xi\circ \rho$ is crystalline/semistable/de Rham/Hodge-Tate. Now suppose that $\rho$ is Hodge-Tate. For each field embedding $i\colon \lql \to \hat{\ol K}$, a cocharacter $\mu_{\HT}(\rho,i)_{\hat{\ol K}} \colon \G_m\to H$ over $\hat{\ol K}$ is called a \emph{Hodge-Tate cocharacter} for $\rho$ and $i$ if for some (thus every) faithful algebraic representation $\xi \colon H\to \GL(V)$ on a finite dimensional $\lql$-vector space $V$, the cocharacter $\xi\circ \rho$ induces the Hodge-Tate decomposition of semilinear $\Gal(\ol{K}/K)$-modules on $\hat{\ol K}$-vector spaces:
$$V\otimes_{\lql,i} \hat{\ol K}= \bigoplus_{k \in \Z} V_k.$$
Namely $V_k$ is the
weight $k$ space for the $\G_m$-action through $\xi\circ\mu_{\HT}(\rho,i)$, while $V_k$ is also the ${\hat{\ol K}}$-linear span of the $K$-subspace in $V \otimes_{\lql,i} \hat{ \ol K}$ on which $\Gal(\ol{K}/K)$ acts through the $(-k)$-th power of the cyclotomic character. (So our convention is that the Hodge-Tate number of the cyclotomic character is $-1$.) Finally, we call a cocharacter $\mu_{\HT}(\rho,i) \colon \G_m \to H$ over $\lql$ a Hodge-Tate cocharacter, if it is conjugate to $\mu_{\HT}(\rho,i)_{\hat{\ol K}}$ in $H(\hat{\ol K})$.
\end{definition}

In any of the above conditions on $\rho$, if it holds for one $\xi$ then it holds for all $\xi$ (use \cite[Prop. I.3.1]{DMOS82}). Whenever $\rho$ is Hodge-Tate, a Hodge-Tate cocharacter exists by \cite[\S1.4]{SerreHT} and is shown to be unique (independent of $\xi$) by a standard Tannakian argument. Often we only care about the isomorphism class of $\rho$, in which case only the $H(\lql)$-conjugacy class of a Hodge-Tate cocharacter matters.

\begin{lemma}\label{lem:Hodge-Tate-cochar}
Let $f:H_1\ra H_2$ be a morphism of connected reductive groups over $\lql$. If $\rho:\Gal(\ol{K}/K)\to H_1(\lql)$ is a Hodge-Tate representation then $f\circ \rho$ is also Hodge-Tate with $\mu_{\HT}(f\circ \rho,i)=f\circ \mu_{\HT}(\rho,i)$ for all $i:\lql\to \hat{\ol K}$.
\end{lemma}

\begin{proof}
This is obvious by considering $\xi\circ f\circ \rho$ for any faithful algebraic representation $\xi:H_2\to \GL_N$.
\end{proof}

We return to the global setup. A representation $\rho:\Gamma\to H(\lql)$ is said to be \emph{totally odd} if $\rho|_{\Gal(\ol{F}_v/F_v)}$ is odd for every $v \in \cV_ \infty$. It is crystalline/semistable/de Rham/Hodge-Tate if $\rho|_{\Gal(\ol{F}_v/F_v)}$ is crystalline/semistable/de Rham/Hodge-Tate for every place $v$ above $\ell$.

\begin{definition}\label{def:cohomological}
Let $H$ be a real reductive group. Let $K_H$ be a maximal compact subgroup of $H(\R)$. Put $\widetilde K_H:=(K_H)^0 Z(H)(\R)$. Let $\xi$ be an irreducible algebraic representation of $H\otimes_\R \C$. An irreducible unitary representation $\pi$ of $H(\R)$ is said to be \emph{cohomological} for $\xi$ (or $\xi$-cohomological) if there exists $i\ge 0$ such that $\uH^i(\Lie H(\C),\widetilde K_H,\pi\otimes_\C \xi)\neq 0$. (The definition is independent of the choice of $K_H$. The group $\widetilde K_H$ is consistent with $K_ \infty$ in \S\ref{sect:ShimuraDatum} below.)
\end{definition}

\begin{example}\label{ex:example}
Let $H,\xi$ be as in Definition \ref{def:cohomological}. Assume that $H(\R)$ has discrete series representations.
Let $\Pi_{\xi}$ be the set of (irreducible) discrete series representations which have the same infinitesimal and central characters as $\xi^\vee$. Then $\Pi_\xi$ is a discrete series $L$-packet, whose $L$-parameter is going to be denoted by $\phi_\xi \colon W_\R\to {}^L H$. Then there are a Borel subgroup $\hat B\subset \hat H$ and a maximal torus $\hat T\subset \hat B$ such that $\phi_\xi(z)=(z/\ol{z})^{\lambda_\xi+\rho}$, where $\lambda_\xi$ is the $\hat B$-dominant highest weight of $\xi$, and $\rho$ is the half sum of $\hat B$-positive roots. Every member of $\Pi_\xi$ is $\xi$-cohomological. More precisely $\uH^i(\Lie H(\C),\widetilde K_H,\pi\otimes \xi)\neq 0$ exactly when $i=\frac{1}{2}\dim_{\R} H(\R)/K'_H$, in which case the cohomology is of dimension $[K_HZ(H)(\R): \widetilde K_H]$, cf. Remark \ref{rem:K_infty} below.
\end{example}

\begin{definition}\label{def:Hodge}
Consider a complex $L$-parameter $\phi\colon W_\C\ra \hat H$. For a suitable maximal torus $\hat T\subset \hat H$, one can describe $\phi$ as $z\mapsto \mu_1(z)\mu_2(\ol{z})$ for $\mu_1,\mu_2 \in X_*(\hat T)_\C$ with $\mu_1-\mu_2 \in X_*(\hat T)$. Write $\Omega_{\hat H}$ for the Weyl group of $\hat T$ in $\hat H$. We define $\mu_{\Hodge}(\phi)$ to be $\mu_1$ viewed as an element of $X_*(\hat T)_\C/\Omega_{\hat H}$. (When $\mu_1$ happens to be integral, i.e. in $X_*(\hat T)$, then we may also view $\mu_{\Hodge}(\phi)$ as a conjugacy class of cocharacters $\G_m\ra \hat H$ over $\C$.) Given $H,\xi$ as in the preceding example, define
$$\mu_{\Hodge}(\xi):=\mu_{\Hodge}(\phi_\xi|_{W_\C}).$$
So if $\hat B,\hat T$ are as before, then $\mu_{\Hodge}(\xi)=\lambda_\xi+\rho \in \frac{1}{2} X_*(\hat T)\subset X_*(\hat T)_\C$ up to the $\Omega_{\hat H}$-action.
\end{definition}

Let $f:H_1\ra H_2$ be a morphism of connected reductive groups over $\R$ whose image is normal in $H_2$ such that $f$ has abelian kernel and cokernel. (Later we will consider the dual of the mapping $\GL_1 \times \Sp_{2n} \to \GSp_{2n}$). Denote by $\hat f: \hat H_2\ra \hat H_1$ the dual morphism. We choose maximal tori $\hat T_{i}\subset \hat H_i$ for $i=1,2$ such that $f(\hat T_{2})\subset \hat T_{1}$. If $\phi_2: W_\R\ra \hat H_2$ is an $L$-parameter then obviously
\begin{equation}\label{eq:Functoriality-and-MuHodge}
\hat f (\mu_{\Hodge}(\phi_2|_{W_\C})) = \mu_{\Hodge}(\hat f \circ \phi_2|_{W_\C}).
\end{equation}

\begin{lemma}\label{lem:restriction-real-parameter}
With the above notation, let $\pi_2$ be a member of the $L$-packet for $\phi_2$. Then the pullback of $\pi_2$ via $f$ decomposes as a finite direct sum of irreducible representations of $H_1(\R)$, and all of them lie in the $L$-packet for $\hat f\circ \phi_2$.
\end{lemma}
\begin{proof}
This is property (iv) of the Langlands correspondence for real groups on page 125 of \cite{LanglandsRealClassification}.
\end{proof}

\section{Arthur parameters for symplectic groups} \label{sect:LiftingGlobal}
Assume $\pi^\flat$ is a cuspidal automorphic representation of $\Sp_{2n}(\A_F)$, such that
\begin{itemize}
\item $\pi^\flat$ is cohomological for an irreducible algebraic representation $\xi^\flat=\otimes_{v \in \cV_ \infty} \xi^\flat_v$ of $\Sp_{2n, F \otimes \C}$,
\item there is an auxiliary finite place $\vst$ such that the local representation $\pi_{\vst}^\flat$ is the Steinberg representation of $\Sp_{2n}(F_\vst)$.
\end{itemize}
In this section we apply the construction of Galois representations \cite{ArthurBook, ShinGalois} to $\pi^\flat$ to obtain a morphism $\rho_{\pi^\flat} \colon \Gamma \to \GO_{2n+1}(\lql)$ and then lift $\rho_{\pi^\flat}$ to a representation $\tilrhoflat \colon \Gamma \to \GSpin_{2n+1}(\lql)$.

Let us briefly recall the notion of (formal) Arthur parameters as introduced in \cite{ArthurBook}. We will concentrate on the discrete and generic case as this is all we need (after Corollary \ref{cor:TemperedParameter} below); refer to \textit{loc. cit.} for the general case. Here genericity means that no nontrivial representation of $\SU_2(\R)$ appears in the global parameter.

For any $N \in \Z_{\ge 1}$ let $\theta$ be the involution on (all the) general linear groups $\GL_{N, F}$, defined by $\theta(x)= J_N {}^{\tu{t}}x^{-1} J_N$ where $J_N$ is the $N\times N$-{ma\-trix} with $1$'s on its anti-diagonal, and $0$'s on all its other entries.
A generic discrete \emph{Arthur parameter} for the group $\Sp_{2n, F}$ is a finite unordered collection of pairs $\{(m_i,\tau_i)\}_{i=1}^r$, where
\begin{itemize}
\item $m_i, r \geq 1$ are positive integers, such that $2n + 1= \sum_{i=1}^r m_i$,
\item for each $i$, $\tau_i$ is a unitary cuspidal automorphic representation of $\GL_{m_i}(\A_F)$ such that $\tau_i^\theta \simeq \tau_i$,
\item the $\tau_i $ are mutually non-isomorphic, and
\item each $\tau_i$ is of orthogonal type, and the product of the central characters of the $\tau_i$ is trivial.
\end{itemize}
We write formally $\psi = \boxplus_{i=1}^r \tau_i $ for the Arthur parameter $\{(m_i,\tau_i)\}_{i=1}^r$. The parameter $\psi$ is said to be simple if $r=1$. The representation $\Pi_{\psi}$ is defined to be the isobaric sum $\boxplus_{i=1}^r \tau_i$; it is a self-dual automorphic representation of $\GL_{2n+1}(\A_F)$.

Exploiting the fact that $\Sp_{2n}$ is a twisted endoscopic group for $\GL_{2n+1}$, Arthur attaches \cite[Thm. 2.2.1]{ArthurBook} to $\pi^\flat$ a discrete Arthur parameter $\psi$. Let $\pi^\sharp$ denote the corresponding isobaric automorphic representation of $\GL_{2n+1}(\A_F)$ as in \cite[\S1.3]{ArthurBook}. (If $\psi$ is generic, which will be verified soon, then $\psi$ has the form as in the preceding paragraph.) For each $F$-{pla\-ce} $v$, the representation $\pi^\flat_v$ belongs to the local Arthur packet $\Pi(\psi_v)$ defined by $\psi$ localized at $v$. This packet $\Pi(\psi_v)$ satisfies the character relation (\cite[Thm. 2.2.1]{ArthurBook})
\begin{equation}\label{eq:EndoscopicCharacterRelation}
\Tr (A_\theta \circ \pi^\sharp_v( f_v)) = \sum_{\tau \in \Pi(\psi_v)} \Tr \tau(f_v^{\Sp_{2n}}),
\end{equation}
for all pairs of functions $f_v \in \cH(\GL_{2n+1}(F_v))$, $f_v^{\Sp_{2n, F}} \in \cH(\Sp_{2n}(F_v))$ such that $f_v^{\Sp_{2n, F}}$ is a Langlands-Shelstad-Kottwitz transfer of $f_v$. Here $A_\theta$ is an intertwining operator from $\pi^\sharp_v$ to its $\theta$-twist such that $A_\theta^2$ is the identity map. (The precise normalization is not recalled as it does not matter to us.)

\begin{lemma}\label{lem:ComponentIsSteinberg}
The component $\pi_{\vst}^\sharp$ is the Steinberg representation of $\GL_{2n+1}(F_{\vst})$.
\end{lemma}
\begin{proof}
This follows from \cite[Prop. 8.2]{MagaardSavin}.
\end{proof}

\begin{corollary}\label{cor:TemperedParameter}
The Arthur parameter $\psi$ of $\pi^\flat$ is simple (i.e. $\psi=\tau_1$ is cuspidal) and generic.
\end{corollary}
\begin{proof}
Lemma \ref{lem:ComponentIsSteinberg} implies in particular that $\psi_\vst$ is a generic parameter which is irreducible as a representation of the local Langlands group $W_{F_\vst}\times \SU_2(\R)$. Hence the global parameter $\psi$ is simple and generic.
\end{proof}

Denote by $\pi^\sharp$ the cuspidal automorphic representation $\tau_1=\Pi_\psi$. Let $A_\theta \colon \Pi_{\Psi} \isomto \Pi_{\Psi}^\theta$ the canonical intertwining operator such that $A_\theta^2$ is the identity and $A_\theta$ preserves the Whittaker model (see \cite[sect. 2.1]{ArthurBook} and \cite[sect. 5.3]{KottwitzShelstad} for this normalization).
Write $\eta$ for the $L$-{mor\-phism} ${}^L\Sp_{2n, F_v} \to {}^L\GL_{2n+1, F_v}$ extending the standard representation $\hat{\Sp}_{2n} = \SO_{2n+1}(\C)\to \GL_{2n+1}(\C)$ such that $\eta|_{W_{F_v}}$ is the identity map onto $W_{F_v}$.

\begin{lemma}\label{lem:ArthurUnramified}
Let $v$ be a finite $F$-{pla\-ce} where $\pi_v^\flat$ is unramified\footnote{We should mention that in the group $\Sp_{2n}(F_v)$, not all hyperspecial subgroups are conjugate (in contrast to $\Gsp_{2n}(F_v)$, where all hyperspecial groups are conjugate). When we say that $\pi_v^\flat$ is unramified, we mean that there exists a hyperspecial subgroup for which the representation has an invariant vector. }. Then $\pi_v^\sharp$ is unramified as well. Let $\phi_{\pi_v^\sharp} \colon W_{F_v} \times \SU_2(\R) \to \GL_{2n+1}(\C)$ be the Langlands parameter of $\pi_v^\sharp$. Let $\phi_{\pi_v^\flat} \colon W_{F_v} \times \SU_2(\R) \to \SO_{2n+1}(\C)$ be the Langlands parameter of $\pi_v^\flat$. Then $\eta \circ \phi_{\pi_v^\flat} = \phi_{\pi_v^\sharp}$.
\end{lemma}
\begin{proof}
The morphism $\eta^* \colon \Hunr(\GL_{2n+1}(F_v))\to \Hunr(\Sp_{2n}(F_v))$ is surjective because the restriction of finite dimensional characters of $\GL_{2n+1}$ to $\SO_{2n+1}$ generate the space spanned by finite dimensional characters of $\SO_{2n+1}$. The lemma now follows from \eqref{eq:EndoscopicCharacterRelation} and the twisted fundamental lemma (telling us that one can take $f_v^{\Sp_{2n}}=\eta^*(f_v)$ in \eqref{eq:EndoscopicCharacterRelation}).
\end{proof}

The existence of the Galois representation $\rho_{\pi^\sharp}$ attached to $\pi^\sharp$ follows from \cite[Thm. VII.1.9]{HarrisTaylor}, which builds on earlier work by Clozel and Kottwitz. (The local hypothesis in that theorem is satisfied by Lemma \ref{lem:ComponentIsSteinberg}. However this lemma is unnecessary for the existence of $\rho_{\pi^\sharp}$ if we appeal to the main result of \cite{ShinGalois}.) The theorem of \cite{HarrisTaylor} is stated over imaginary CM fields but can be easily adapted to the case over totally real fields, cf. \cite[Prop. 4.3.1]{CHT08} and its proof. (Also see \cite[Thm. 2.1.1]{BLGGT14-PA} for the general statement incorporating later developments such as the local-global compatibility at $v|\ell$, which we do not need.) We adopt the convention in terms of $L$-algebraic representations as in \cite[Conj. 5.16]{BuzzardGee} unlike the references just mentioned, in which $C$-algebraic representations are used (cf. \cite[\S5.3, \S8.1]{BuzzardGee}).

To state the Hodge-theoretic property at $\ell$ precisely, we introduce some notation based on \S\ref{sec:conventions}. At each $y \in \cV_ \infty$ we have a real $L$-parameter $\phi_{\xi^\flat_y}:W_{F_y}\ra {}^L\Sp_{2n}$ arising from $\xi^\flat_y$. The parameter is $L$-algebraic as well as $C$-algebraic. Via the embedding $y:F\hra \C$ we may identify the algebraic closure $\ol{F}_y$ with $\C$ so that $W_{\ol{F}_y}=W_\C$. As explained in Definition \ref{def:Hodge}, we have $\mu_{\Hodge}(\xi^\flat_y):=\mu_{\Hodge}(\phi_{\xi^\flat_y}|_{W_{\ol{F}_y}})$, a conjugacy class of cocharacters $\G_m\ra \SO_{2n+1}(\C)$.

\begin{theorem}\label{thm:ExistGaloisRep}
There exists an irreducible Galois representation
$$
\rho_{\pi^\flat}=\rho_{\pi^\flat,\iota} \colon \Gamma \to \SO_{2n+1}(\lql),
$$
unique up to $\SO_{2n+1}(\lql)$-conjugation, attached to $\pi^\flat$ (and $\iota$) such that the following hold.
\begin{itemize}
\item[(\textit{i})]
Let $v$ be a finite place of $F$ not dividing $\ell$. If $\pi^\flat_v$ is unramified then
$$
(\rho_{\pi^\flat}|_{W_{F_v}})_{\ssimple} \simeq \iota\phi_{\pi^\flat_v},
$$
where $\phi_{\pi^\flat_v}$ is the unramified $L$-parameter of $\pi^\flat_v$, and $(\cdot)_{\ssimple}$ is the semisimplification. For general $\pi^\flat_v$, the parameter $\iota\phi_{\pi^\flat_v}$ is isomorphic to the Frobenius-semisimplification of the parameter associated to $\rho_{\pi^\flat}|_{\Gamma_{F_v}}$. \footnote{This is equivalent to saying that $\iota(\std\circ\phi_{\pi^\flat_v})$ is isomorphic to the Frobenius-semisimplification of the Langlands parameter associated with $(\std\circ\rho_{\pi^\flat})|_{\Gamma_{F_v}}$ in view of Proposition \ref{prop:ConjugacyStdRep}. }
\item[(\textit{ii})] Let $v$ be a finite $F$-place such that $v\nmid \ell$ where $\pi_v$ is unramified. Then $\rho_{\pi^\flat,v}$ is unramified at $v$, and for all eigenvalues $\alpha$ of $\std(\rho_{\pi^\flat}(\Frob_v))_{\ssimple}$ and all embeddings $\li \Q \hookrightarrow \C$ we have $|\alpha| = 1$.
\item[(\textit{iii})] For every $v | \ell$, the representation $ \rho_{\pi^\flat, v}$ is potentially semistable. For each $y:F\hra \C$ such that $\iota y$ induces $v$, we have
$$
\mu_{\HT}(\rho_{\pi^\flat,v},\iota y) = \iota\mu_{\Hodge}(\xi^\flat_y).
$$
\item[(\textit{iv})] For every $v | \ell$, the Frobenius semisimplification of the Weil-Deligne representation attached to the de Rham representation $\rho_{\pi^\flat, v}$ is isomorphic to the Weil-Deligne representation attached to $\pi^\flat_v$ under the local Langlands correspondence.
\item[(\textit{v})] If $\pi_v$ is unramified at $v | \ell$, then $\rho_{\pi^\flat,v}$ is crystalline. If $\pi_v$ has a non-zero Iwahori fixed vector at $v | \ell$, then $\rho_{\pi^\flat,v}$ is semistable.
\item[(\textit{vi})] The representation $\rho_{\pi^\flat}$ is totally odd.
\item[(\textit{vii})] If $\pi^\flat_v$ is essentially square-integrable at $v\nmid \infty$ then $\std\circ\rho_{\pi^\flat}$ is irreducible.
\end{itemize}
\end{theorem}

\begin{remark}
We need the ramified case of (i) only when $\pi^\flat_v$ is the Steinberg representation.
\end{remark}

\begin{proof}
The representation $\pi^\sharp$ is regular algebraic \cite[Def. 3.12]{ClozelAnnArbor}, and so $\pi_ \infty^\sharp$ is cohomological for an irreducible algebraic representation.
By Theorem 1.2 of \cite{ShinGalois} (in view of Remark~7.6 in~\textit{loc. cit.}) there exists a Galois representation $\rho_{\pi^\sharp} \colon \Gamma \to \GL_{2n+1}(\lql)$ that satisfies properties (i)-(iii),(v) with $\pi^\sharp$ in place of $\pi^\flat$; property (iv) is established by Caraiani \cite[Thm. 1.1]{CaraianiLGCl=p}. (The reader may also refer to the statement of \cite[Thm. 2.1.1]{BLGGT14-PA}.) Strictly speaking, the normalization there is different from here, so one has to twist the Galois representation there by the $\frac{n-1}{2}$-th power of the cyclotomic character. In particular $\rho_{\pi^\sharp}$ is self-dual.

By Lemma \ref{lem:ComponentIsSteinberg}, $\pi_{\vst}^\sharp$ is the Steinberg representation and by Taylor--Yoshida \cite{TaylorYoshida} the representation $\rho_{\pi^\sharp} \colon \Gamma \to \GL_{2n+1}(\lql)$ is irreducible. The determinant of this representation is trivial, since in the Arthur parameter $\{(m_i,\tau_i)\}_{i=1}^r$, the product of the central characters of the $\tau_i$ is trivial. Together with the self-duality of $\rho_{\pi^\sharp}$, we see that $\rho_{\pi^\sharp}$ factors through a representation $\rho_{\pi^\flat} \colon \Gamma \to \SO_{2n+1}(\lql)$ via the standard embedding $\SO_{2n+1}\hra \GL_{2n+1}$ (after a conjugation by an element of $\GL_{2n+1}(\lql)$). We know the uniqueness of $\rho_{\pi^\flat}$ from Proposition \ref{prop:ConjugacyStdRep} (and Chebotarev density). Properties (i)-(v) for $\rho_{\pi^\flat}$ follow from those for $\rho_{\pi^\sharp}$. Part (vi) is deduced from Lemma \ref{lem:odd-GO} and the main theorem of \cite{TaylorOdd}. Lastly (vii) is \cite[Cor. B]{TaylorYoshida}.
\end{proof}

For the rest of this section, let $\pi$ be a cuspidal $\xi$-cohomological automorphic representation of $\GSp_{2n}(\A_F)$ for an irreducible algebraic representation $\xi$ of $\GSp_{2n,F\otimes \C}$.

\begin{lemma}[Labesse-Schwermer, Clozel]\label{lem:LabesseSchwermer}
There exists a cuspidal automorphic $\Sp_{2n}(\A_F)$-sub\-re\-pre\-sen\-ta\-tion $\pi^\flat$ contained in $\pi$.
\end{lemma}
\begin{proof}
This follows from the main theorem of Labesse-Schwermer \cite{LabesseSchwermer2018}.
\end{proof}

\begin{lemma}\label{lem:ForceTempered}
Suppose that $\pi$ is a twist of the Steinberg representation at a finite place. Then $\pi_v$ is essentially tempered at all places $v$.
\end{lemma}
\begin{proof}
Let $\pi^\flat$ be as in the previous lemma. We know that $\pi^\flat$ is the Steinberg representation at a finite place and $\xi^\flat$-cohomological, where $\xi^\flat$ is the restriction of $\xi$ to $\Sp_{2n,F\otimes \C}$ (which is still irreducible). Let $\pi^\sharp$ be the self-dual cuspidal automorphic representation of $\GL_{2n+1}(\A_F)$ as above. Note also that $\pi^\sharp$ is $C$-algebraic (and regular); this is checked using the explicit description of the archimedean $L$-parameters. Hence
\begin{itemize}
\item $\pi^\sharp_v$ is tempered at all $v| \infty$ by Clozel's purity lemma \cite[Lem. 4.9]{ClozelAnnArbor},
\item $\pi^\sharp_v$ is tempered at all $v\nmid \infty$ by \cite[Cor. 1.3]{ShinGalois} and (quadratic) automorphic base change.
\end{itemize}
(Since $\pi^\sharp$ is self-dual, if a local component is tempered up to a character twist then it is already tempered.) Hence $\pi^\flat_v$ is a tempered representation of $\Sp_{2n}(F_v)$, cf. \cite[Thm. 1.5.1]{ArthurBook}. This implies that $\pi_ \infty$ itself is essentially tempered. (Indeed, after twisting by a character, one can assume that $\pi_ \infty$ restricts to a unitary tempered representation on $\Sp_{2n}(F_v) \times Z(F_v)$, which is of finite index in $\GSp_{2n}(F_v)$. Then temperedness is tested by whether the matrix coefficient (twisted by a character so as to be unitary on $Z(F_v)$) belongs to $L^{2+\eps}(\GSp_{2n}(F_v)/Z(F_v))$. This is straightforward to deduce from the same property of matrix coefficient for its restriction to $\Sp_{2n}(F_v) \times Z(F_v)$.)
\end{proof}

\begin{corollary}\label{cor:Pi-is-in-Disc-series-Packet}
$\pi_ \infty$ belongs to the discrete series $L$-packet $\Pi_\xi$.
\end{corollary}
\begin{proof}
By \cite[Thm. III.5.1]{BorelWallach}, $\Pi_\xi$ coincides with the set of essentially tempered $\xi$-cohomological representations. Since $\pi_ \infty$ is essentially tempered and $\xi$-cohomological, the corollary follows.
\end{proof}

\section{Zariski connectedness of image}\label{sect:StrongIrreducibility}

Let $\rho_{\pi^\flat}$ be the Galois representation from Section~2. The goal of this section is to prove that $\rho_{\pi^\flat}$ has connected image in the sense defined below. The proof of this fact relies on the existence of the regular unipotent element in the image thanks to assumption (St).

\begin{definition}\label{def:connected-irreducible}
Let $G_0$ be a connected reductive group over $\lql$. A representation $r:\Gamma\ra G_0(\lql)$ is said to have \emph{connected image} if the image of $r$ has connected Zariski closure in $G_0$. (The definition applies to any topological group in place of $\Gamma$.) We say that $r \colon \Gamma \to G_0(\lql)$ is $G_0$-\emph{irreducible} if its image is not contained in any proper parabolic subgroup of $G_0$, and we say that $r$ is \emph{strongly $G_0$-irreducible} if $r|_{\Gamma'}$ is $G_0$-irreducible for every open finite index subgroup $\Gamma'$ of $\Gamma$. In case $G_0 = \GL_n$, we often leave out the reference to the group $G_0$.
\end{definition}

When $G_0 = \GL_n$, a representation $r \colon \Gamma \to \GL_n(\lql)$ is strongly irreducible if it is irreducible and has connected image in $\PGL_n$ (Lemma \ref{lem:SugWooNew1}). Typical examples of irreducible representations that are not strongly irreducible are Artin representations and the 2-dimensional $\ell$-adic representation arising from a CM elliptic curve over $\Q$.

The lemma below is extracted from arguments of \cite[p. 675]{SchneiderTeitelbaum}. (One can always enlarge $L$ to satisfy the first two conditions in the lemma.)

\begin{lemma}\label{lem:Tannakian}
Let $L/L'/\ql$ be two finite extensions in $\lql$. Write $\Gamma_{L'} := \Gal(\lql/L')$ and $\Gamma_ L:=\Gal(\lql/L)$. Let $\rho \colon \Gamma_{L'} \to \SO_{2n+1}(\lql)$ be a semi-stable representation. Let $H \subset \SO_{2n+1, \lql}$ be the Zariski closure of the image of $\rho$. Assume that
\begin{itemize}
\item $H$ is defined over $L$,
\item the image of $\rho$ is contained in $\SO_{2n+1}(L)$, and
\item the Weil-Deligne representation (the functor $\tu{WD}$ is defined as in \cite[p. 12]{BergerBarcelona})
\begin{equation}\label{eq:AssumpUnipotent}
\tu{WD} \lhk \lhk \Bst \otimes_{\Q_\ell} (\std \circ \rho)\rhk^{\Gamma_L} \rhk
\end{equation}
has a nilpotent operator $N_{\tu{reg}}$ that is regular in $\tu{Lie}( \SO_{2n+1}(\lql))$.
\end{itemize}
Then the group $H(\lql)$ contains a regular unipotent element of $\SO_{2n+1}(\lql)$.
\end{lemma}
\begin{proof} The underlying space of the $(\phi, N)$-module $\lhk \Bst \otimes_{\Q_\ell} (\std \circ \rho)\rhk^{\Gamma_L}$ is naturally a vector space over the maximal subfield $L_0 \subset L$ that is unramified over $\ql$.
We consider the functor
$$
\Psi \colon \Rep_L(H) \to \Rep_L(\Ga), \quad r \mapsto \tu{WD}\lhk\lhk \Bst \otimes_{\Q_\ell} (r \circ \rho)\rhk^{\Gamma_L} \rhk|_{\Ga}
$$
where, if $V$ is an $L$-vector space equipped with the structure of a Weil-Deligne representation for $\Gamma_{L'}$, we write $V|_{\Ga}$ for the unique representation of the additive group $\rho \colon \Ga \to \GL(V)$ such that $\Lie(\rho)(1)$ is equal to the nilpotent operator of the Weil-Deligne representation $V$.

Let $\omega_{H}$ and $\omega_{\Ga}$ be the standard fibre functors of the categories $\Rep_L(H)$ and $\Rep_L(\Ga)$. The functor $\Psi$ is an exact faithful $L$-linear $\otimes$-functor. The composition $\omega_{\Ga} \circ \Psi$ is therefore a fibre functor of $\Rep_L(G)$ \cite[Sect. II.3]{DMOS82}. By Theorem 3.2 of [\textit{loc. cit.}] we have the equivalence of categories,
$$
\tu{Fibre functors of $\Rep_L(H)$} \isomto \tu{Category of $H$-torsors over $L$}, \quad \eta \mapsto \uHom^\otimes(\eta, \omega_H).
$$
We obtain a morphism $\underline \Aut^\otimes(\omega_{\Ga})_{\lql} \to \underline \Aut^\otimes(\omega_H)_{\lql}$, $\ie$ an algebraic morphism $\Psi^* \colon \lql \to H(\lql)$, well-defined up to conjugation. Write $N = \Psi^*(1) \in H(\lql)$. For every linear representation $r$ of $H$ in a finite dimensional $L$-vector space the element $\Lie (\Psi^*)(1)$ is conjugate to the nilpotent operator of the Weil-Deligne representation attached to $\lhk \Bst \otimes_{\Q_\ell} (r \circ \rho)\rhk^{\Gamma_L}$. Taking $r = \std$, we see that $\Lie(\Psi^*)(1)$ and $N_{\tu{reg}}$ are $\GL_{2n+1}(\lql)$-conjugate by Assumption \eqref{eq:AssumpUnipotent}. Thus $N = \Psi^*(1) \in H(\lql)$ is a regular unipotent element of $\SO_{2n+1}(\lql)$.
\end{proof}

\begin{proposition}\label{prop:3cases}
Let $H$ be a semisimple subgroup of\, $\SO_{2n+1, \lql}$ containing a regular unipotent element $N \in\SO_{2n+1, \lql}$. Then either $H$ is the full group $\SO_{2n+1,\lql}$, $H$ is the principal $\PGL_{2,\lql}$ in $\SO_{2n+1,\lql}$, or $n = 3$ and $H$ is the simple exceptional group $G_2$ over $\lql$ (\ie $H$ is the automorphism group of the octonion algebra $\mathbb O \otimes \lql$).
\end{proposition}
\begin{proof}
 Let $H^0 \subset H$ denote the identity component. Since there are no non-trivial morphisms from ${\mathbb G}_\textup{a}$ to the finite group $\pi_0(H)$, we must have $N \in H^0$. Therefore $H^0$ is either $\PGL_{2,\lql}$, $G_2$, or $\SO_{2n+1,\lql}$ as in the theorem, by \cite[Thm. 1.4]{TestermanZalesski} classifying connected semisimple subgroups of $\SO_{2n+1, \lql}$ containing regular unipotent elements. From \cite{SaxlSeitz} it follows that in each of the 3 cases the subgroup $H^0 \subset \SO_{2n+1,\lql}$ is a maximal closed subgroup (Theorem B, (i.a), (iv.a) and (iv.e) of [\textit{loc. cit.}]).\footnote{In fact it is not completely clear whether \cite[Thm.~B]{SaxlSeitz} classifies maximal closed subgroups or maximal connected closed subgroups; see footnote 9 in the proof of \cite[Prop.~5.2]{GSO}. Nevertheless, we can still show that $H=H^0$ in the case at hand as in that proof, using the fact that $\PGL_{2,\lql}$ and $G_2$ have no outer automorphism, which implies that the conjugation action by elements of $H$ on $H^0$ gives inner automorphisms of $H^0$.} In particular $H^0 = H$.
\end{proof}

 The following lemma will be used in the proof of Proposition \ref{prop:DetermineZariskiImage} below.

\begin{lemma}[{Liebeck--Testerman \cite[Lem. 2.1]{LiebeckTesterman}}]\label{lem:LiebeckTesterman}
Let $G$ a semisimple connected algebraic group over an algebraically closed field. If $X$ is a connected $G$-irreducible subgroup of $G$, then $X$ is semisimple, and the centralizer of $X$ in $G$ is finite.
\end{lemma}

\begin{proposition}\label{prop:DetermineZariskiImage}
Assume $n\ge 3$. The Zariski closure of the subgroup $\rho_{\pi^\flat}(\Gamma) \subset \SO_{2n+1}(\lql)$ is either $\textup{PGL}_2$, $G_2$, or $\SO_{2n+1}$. (The embedding of $\textup{PGL}_2$ is induced by the symmetric $2n$-th power representation of $\GL_2$. The group $G_2$ occurs only when $n=3$ and embeds in $\SO_7$ via an irreducible self-dual 7-dimensional representation.)
\end{proposition}
\begin{proof}
Write $H\subset \SO_{2n+1}$ for the Zariski closure of $\rho_{\pi^\flat}(\Gamma)$. We claim that $H$ contains a regular unipotent element of $\GL_{2n+1}$ (thus also of $\SO_{2n+1}$).

To prove this, we distinguish cases. Let us first assume $\vst \nmid \ell$. Let $N_{\vst}$ be the unipotent operator of the Weil--Deligne representation attached to $\rho_{\pi^\flat, \vst}$ so that it corresponds to the image of $\left(1, \vierkant 1101 \right)$ under $\iota\phi_{\pi^\flat_{\vst}}$ by Theorem \ref{thm:ExistGaloisRep}. Then $N_{\vst}$ is regular unipotent in $\GL_{2n+1}$ since $\pi^\flat_{\vst}$ and thus $\pi^\sharp_{\vst}$ is Steinberg by Lemma \ref{lem:ComponentIsSteinberg}. The attached Weil-Deligne representation has the property that a positive power of $N_{\vst}$ lies in the image of $\rho_{\pi^\flat, \vst}$. This completes the proof if $\vst \nmid \ell$. If $\vst |\ell$, we take a finite extension $E/F$ such that $\rho: = \rho_{\pi^\flat, \vst}|_{\Gamma_E}$ is semistable at places above $\vst$. By part (iv) of Theorem \ref{thm:ExistGaloisRep}, the last assumption of Lemma \ref{lem:Tannakian} is satisfied. (To satisfy the first two, take $L$ to be large enough.) Applying Lemma \ref{lem:Tannakian}, we see that $\rho_{\pi^\flat, \vst}(\Gamma_E)$, thus also $H$, contains a regular unipotent element of $\GL_{2n+1}$.

We have shown that $H\subset \SO_{2n+1}$ contains a regular unipotent element. On the other hand, $H$ is an irreducible subgroup of $\SO_{2n+1, \lql}$ by Theorem \ref{thm:ExistGaloisRep}. By Lemma~\ref{lem:LiebeckTesterman}, $H^0$ is a semisimple group and hence $H$ as well. By Proposition \ref{prop:3cases}, $H$ is one of the groups $\SO_{2n+1}$, $G_2$ or $\textup{PGL}_2$ over $\lql$. Thus $H=H^0$ and the proposition follows.
\end{proof}

\section{Weak acceptability and connected image}

A classical theorem for Galois representations states that if $r_1, r_2 \colon \Gamma \to \GL_m(\lql)$ are two semisimple representations which are locally conjugate, then they are conjugate. (Recall that every representation is assumed to be continuous by the convention of this paper.) This is a consequence of the Brauer-Nesbitt theorem combined with the Chebotarev density theorem. In this section we investigate analogous statements when $\GL_m(\lql)$ is replaced by a more general (not necessarily connected) reductive group. We show that in many cases the implication still holds if one assumes that one of the two representations has Zariski connected image.

\begin{definition}
A (possibly disconnected) reductive group $G$ over $\lql$ is said to be \emph{weakly acceptable} if for every
profinite group $\Delta$ and any two locally conjugate semisimple continuous representations $r_1, r_2 \colon \Delta \to G(\lql)$ there exists an open subgroup $\Delta' \subset \Delta$ such that $r_1$ and $r_2$ restricted to $\Delta'$ are conjugate.
 \end{definition}

\begin{lemma}\label{lem:ConjugacyUpToRootsOfUnity}
Let $r_1, r_2 \colon \Delta \to \GL_m(\lql)$ be two representations such that for some finite subgroup $\mu$ of $\lql^\times$ we have for each $\sigma \in \Delta$ that $\Tr r_1(\sigma) = \zeta \Tr r_2(\sigma)$ for some $\zeta = \zeta_\sigma \in \mu$. Then $\Tr r_1(\sigma) = \Tr r_2(\sigma)$ for all $\sigma$ in some open subgroup $\Delta'$ of $\Delta$.
\end{lemma}
\begin{proof}
Choose some open ideal $I \subset \lzl$ such that $(1-\zeta)m \notin I$ for all $\zeta \in \mu_t(\lql)$ with $\zeta \neq 1$. Take $U \subset \Delta$ the set of $\sigma$ such that for $i=1,2$ we have for all $\sigma \in U$ that $\Tr r_i(\sigma) \equiv m \mod I$. Then $U$ is open by continuity of the representations $r_i$. Let $\sigma \in U$. The traces $\Tr r_1(\sigma)$ and $\Tr r_2(\sigma)$ agree up to an element $\zeta_\sigma \in \mu_t(\lql)$. Reducing modulo $I$ we get $m \equiv \zeta_\sigma m$. This is only possible if $\zeta_\sigma = 1$. For the unit element $e \in \Delta$ we have $\Tr r_i(e) = m$ ($i=1,2$), consequently $U$ is an open neighborhood of $e \in \Delta$. We may now take $\Delta'$ an open subgroup of $\Delta$ contained in $U$.
\end{proof}

\begin{lemma}\label{lem:A}
Assume $[\mu \injects H \surjects G]$ is a central extension of reductive groups over $\lql$, where we assume additionally that $\mu$ is finite. Then the group $G$ is weakly acceptable if and only if $H$ is weakly acceptable.
\end{lemma}
\begin{proof}
``$\Leftarrow$''. Assume that $H$ is weakly acceptable. Let $\Delta$ be a profinite group and $r_1,r_2 \colon \Delta \to G(\lql)$ locally conjugate. Choose $U \subset H(\lql)$ an open subgroup with $U \cap \mu = \{1\}$. Write $\li U \subset G(\lql)$ for the image of $U$ in $G(\lql)$. The map $H \surjects G$ induces a bijection $U \isomto \li U$. Consider $\Delta' := r_1^{-1}(\li U)\cap r_2^{-1}(\li U) \subset \Delta$. Using the isomorphism $U \isomto \li U$ we obtain lifted representations $\wt r_1, \wt r_2 \colon \Delta' \to H(\lql)$. By the local $G$-conjugacy of $r_1, r_2$, there exists for each $\sigma \in \Delta'$ some element $h = h_\sigma \in H(\lql)$ and a $z = z_{\sigma} \in \mu(\lql)$ such that
\begin{equation}\label{eq:Pairzh}
\wt r_2(\sigma)_{\tu{ss}} = z h \wt r_1(\sigma)_{\tu{ss}} h\inv \in H(\lql)
\end{equation}
We claim that in an open subgroup $\Delta''$ of $\Delta'$ the elements $\wt r_2(\sigma)$ and $\wt r_1(\sigma)$ are in fact $H$-conjugate for every $\sigma \in \Delta''$ (as opposed to conjugate up to $\mu \subset H$).

Choose a faithful representation $\varphi \colon H \to \GL_N$, and write $\varphi$ as a direct sum of irreducible representations $\varphi = \oplus_{i=1}^m \varphi_i$. By Schur's lemma we have $\varphi_i(\mu) \subset \mu_t$ where $\mu_t$ are the $t$-th roots of unity for some $t \in \Z_{\geq 1}$. In particular we have the identity
$$
\Tr \varphi_i (\wt r_2(\sigma)) = \zeta \Tr \varphi_i (\wt r_1(\sigma)) \in \lql,
$$
where $\zeta := \varphi_i(z)$ with $z$ as in \eqref{eq:Pairzh}. Consequently there exists by Lemma \ref{lem:ConjugacyUpToRootsOfUnity} some open subgroup $\Delta_i \subset \Delta$ such that on $\Delta_i$ we have
$\Tr \varphi_i \wt r_1(\sigma) = \Tr \varphi_i \wt r_2(\sigma)$.
Consider an open subgroup $\Delta'_i \subset \Delta_i$ such that $\Tr \varphi_i \wt r_1(\sigma) \neq 0$ for all $\sigma \in \Delta'_i$. Now let $\sigma \in \Delta'' := \cap_{i=1}^m \Delta_i'$ and let $(z, h)$ be as in Equation \eqref{eq:Pairzh}. By applying $\Tr \varphi_i$ to \eqref{eq:Pairzh} we find
$$
0 \neq \Tr \varphi_i \wt r_1(\sigma) = \Tr \varphi_i \wt r_2(\sigma) = \varphi_i(z) \Tr \wt r_1(\sigma) \in \lql,
$$
so $\varphi_i(z) = 1$ for all $i = 1,2,\ldots, m$, and hence $z = 1$. This proves the claim.

The group $H$ is weakly acceptable. Hence the representations $\wt r_1|_{\Delta'''}$ and $\wt r_2|_{\Delta'''}$ are conjugate by some $h \in H(\lql)$ on some open subgroup $\Delta''' \subset \Delta''$. Since $H(\lql) \supset U\isomto \li U \subset G(\lql)$ the representations $r_1|_{\Delta'''}$ and $r_2|_{\Delta'''}$ are conjugate by the image of $h$ in $G(\lql)$.

``$\Rightarrow$''. Let $r_1, r_2 \colon \Delta \to H(\lql)$ be two locally conjugate semisimple continuous representations. Their projections $\li r_1, \li r_2$ are locally conjugate in $G(\lql)$. Hence there exists an open $\Delta' \subset \Delta$ such that $\li r_2|_{\Delta'} = g (\li r_1|_{\Delta'}) g\inv$.
Lift $g$ to $\wt g \in H(\lql)$ and replace $r_2$ by $\wt g r_2 \wt g\inv$. Since $\mu$ is central we obtain the character $\chi(\sigma) := r_1(\sigma)r_2(\sigma)\inv$ of $\Delta'$. This character has finite order. Hence $r_1$ and $r_2$ agree on the open subgroup $\Ker(\chi) \subset \Delta'$.
\end{proof}

\begin{lemma}\label{lem:C}
Let $G/\lql$ be a reductive group with center $Z_G$ for which there exists a semisimple subgroup $G_1 \subset G$ such that $Z_G \times G_1 \to G$ is surjective.
\begin{enumerate}
\item The group $G$ is weakly acceptable if and only if its adjoint group $G_{\ad}$ is weakly acceptable.
\item If $G$ is connected and $G_{\ad}$ is a product of copies of the groups $\PGL_n, \SO_{2n+1}$ and $\tu{PSp}_{2n}$, then $G$ is weakly acceptable.
\item If $G$ is $\textup{GPin}_{2n}$ or $\GO_{2n}$, then $G$ is weakly acceptable.
\end{enumerate}
\end{lemma}
\begin{remark}
Note that for connected groups $G$, we may take $G_1 = G^\der$ and the condition of the lemma is always satisfied. For non-connected $G$ the condition is not empty: One may consider the semi-direct product $G$ of $\GL_2$ with $\{\pm 1\}$ where $-1$ acts by $g \mapsto g^{-1,t}$. In this case there is no semisimple subgroup $G_1 \subset G$ such that $Z_G \times G_1 \to G$ is surjective ($\tu{rk}(Z_G \times G_1) = 0+1 < 2= \tu{rk}(G)$). For applications in this paper, we only need connected $G$, so the reader may choose to restrict to this case.
However we state a more general version for potential applications elsewhere.
\end{remark}
\begin{proof}
(1). The statement follows by applying the previous lemma to the natural surjections $Z_G \times G_1 \to G$ and $G_1 \to G_{1,\ad}$. Both maps have finite kernels as $Z_{G_1}$ is finite and contains $Z_G\cap G_1$. Therefore $G$ is weakly acceptable if and only if $G_1$ is weakly acceptable if and only if $G_{1,\ad}$ is weakly acceptable. Finally $G_1\subset G$ induces an isomorphism $G_{1,\ad}\isomto G_{\ad}$.

(2). If $G$ is connected we can take $G_1 = G^{\der}$ and use acceptability of the groups $\GL_n$, $\SO_{2n+1}$ and $\Sp_{2n}$; the latter two cases follow from Proposition \ref{prop:ConjugacyStdRep}.

(3).
Since $\tu{O}_{2n}$ is acceptable by Proposition \ref{prop:ConjugacyStdRep}, the group $\PO_{2n}$ is weakly acceptable by Lemma \ref{lem:A}. For $G = \textup{GPin}_{2n}$ (resp. $\GO_{2n}$) we take $G_1 := \textup{Pin}_{2n}$ (resp. $\textup{O}_{2n}$), and we have the surjection $Z_G \times G_1 \surjects G$. From the surjection $G_1 \surjects \PO_{2n}$ we find that $G_1$ is weakly acceptable, and then $G$ is also weakly acceptable by (\textit{i}).
\end{proof}

\begin{proposition}\label{prop:WeaklyAcceptable}
Let $G$ be a reductive group with center $Z$, cocenter $D$ and $\mu \subset Z$ the kernel of the natural morphism $Z \to D$. Assume that $G$ is weakly acceptable. Let $\Delta$ be a profinite group and $r_1, r_2 \colon \Delta \to G(\lql)$ be two locally conjugate continuous semisimple representations where we assume that $r_1$ has Zariski connected image\footnote{Recall that a morphism of finite type affine algebraic groups over a field is closed, in particular there is no difference in taking the Zariski closure before or after reducing modulo $Z$.} modulo $Z$. Then:
\begin{enumerate}
\item We have $r_2 =\chi \cdot g r_1 g\inv $ where $g \in G(\lql)$ and $\chi \colon \Delta \to \mu(\lql)$ is some character.
\item If the image of $r_1$ in $G(\lql)$ is Zariski connected modulo a subgroup $Z_0 \subset Z$ with $Z_0 \cap \mu = \{1\}$, then we have $\chi = 1$ in (1).
\item Assume the existence of $\sigma_0 \in \Delta$ with the following property: the images of $r_1(\sigma_0)$ and $g r_1(\sigma_0) g\inv$ in $G/\mu(\lql)$ are not $G^0(\lql)$-conjugate for any $g \in G_{\tu{ad}}(\lql) \backslash G^0_{\tu{ad}}(\lql)$ Then $\chi r_1$ and $r_2$ are $G^0(\lql)$-conjugate.
\end{enumerate}
\end{proposition}
\begin{example}\label{ex:PGL-example} Consider the case $G = \PGL_m$. By (1) any two locally conjugate projective Galois representations $r_1, r_2$ are conjugate on an open subgroup of $\Delta$ (of index at most $m$). If furthermore one of the two representations has connected image, then $r_1$ and $r_2$ are $\PGL_m(\lql)$-conjugate by (2). Similarly, two locally conjugate $\GSpin_{2n+1}$-valued Galois representations are conjugate on an open subgroup of index $2$, since $\mu$ is a group of order 2 when $G=\GSpin_{2n+1}$.
\end{example}
\begin{proof}
(1).
Since $G$ is weakly acceptable, there exists some $g \in G(\lql)$ and an open subgroup $\Delta' \subset \Delta$ such that $r_2 = g r_1 g\inv$ on $\Delta'$. Let $\varphi \colon G/Z \to \GL_N$ be a faithful representation and choose $x \in \GL_N(\lql)$ such that $\varphi r_1(\sigma) = x \varphi r_2(\sigma) x\inv \in \GL_N(\lql)$ for all $\sigma \in \Delta$. For $\sigma \in \Delta'$ we obtain
$$
\varphi r_1(\sigma) = x \varphi r_2(\sigma) x\inv= x \varphi (g r_1(\sigma) g\inv) x\inv \in \GL_N(\lql).
$$
Taking the Zariski closure we have $\varphi(i) = x \varphi (gig\inv) x\inv \in \GL_N$ for all $i$ in the Zariski closure $I$ of the image of $r_1(\Delta')$. Since $I/Z$ is connected, we obtain $\varphi r_1(\sigma) = x \varphi(g r_1(\sigma) g\inv) x\inv \in \GL_N(\lql)$ for all $\sigma \in \Delta$. Earlier we found $\varphi r_1(\sigma)=x \varphi r_2(\sigma) x\inv \in \GL_N(\lql)$ for all $\sigma \in \Delta$. Thus $\varphi r_2(\sigma) = \varphi (g r_1(\sigma)g\inv) \in \GL_N(\lql)$ for all $\sigma \in \Delta$. We obtain that $r_2(\sigma)=\chi(\sigma) \cdot g r_1(\sigma) g\inv \in G(\lql)$ for some character $\chi \colon \Delta \to Z(\lql)$. By mapping down to the cocenter $D$ of $G$ and using local conjugacy, we see that $\chi$ has image in $\mu$.

(2). Repeat the same proof, but now take $\varphi$ a faithful representation of $G/Z_0$ (as opposed to $G/Z$). Then we find a character $\chi$ as above, whose image is in $\mu$ and $Z_0$, and therefore trivial.

(3). By (1) we have $r_2 = g \chi r_1 g\inv$ for some $g \in G_{\ad}(\lql)$. By local conjugacy at $\sigma_0$, we have that $r_2(\sigma_0)= \chi(\sigma_0) g r_1(\sigma_0) g\inv$ and $r_1(\sigma_0)$ are $G_{\ad}(\lql)$-conjugate. By the assumption, this implies $g \in G^0_{\ad}(\lql)$. Hence $r_2$ and $\chi r_1$ are $G^0_{\ad}(\lql)$-conjugate.
\end{proof}

\begin{lemma}\label{lem:SugWooNew1}
Let $\Delta$ be a profinite group. Let $r \colon \Delta\ra \GL(V)$ be a representation on a finite dimensional $\lql$-vector space $V$ and let $\chi \colon \Delta\ra \lql^\times$ be a nontrivial character.
\begin{enumerate}
\item If $r\simeq r\otimes \chi$ then $r$ is not strongly irreducible.
\item If $r$ is irreducible and has connected image in $\PGL(V)$, then $r$ is strongly irreducible.
\end{enumerate}
\end{lemma}

\begin{proof}
(1) We may assume $r$ is irreducible. Since $\det(r)=\det(r\otimes\chi)=\det(r) \chi^{\dim V}$, the character $\chi$ has finite order. Take $\Delta':=\ker \chi$, which is an open normal subgroup of $\Delta$. It is enough to show that we have the following inequality
$$
\dim \Hom_{\Delta'} (r,r)=\dim \Hom_{\Delta}(r,\mathrm{Ind}_{\Delta'}^\Delta r)>1.
$$
We have the tautological $\Delta$-equivariant embedding $r\hra \mathrm{Ind}_{\Delta'}^\Delta r$ given by $v\mapsto (\gamma\stackrel{f_v}{\mapsto} r(\gamma)v)$. Define $f^\chi_v(\gamma):=\chi(\gamma) f_v(\gamma)$. Then $v\mapsto f^\chi_v$ gives another embedding $r\hra \mathrm{Ind}_{\Delta'}^\Delta r$ that is not a scalar multiple of $v\mapsto f_v$.

(2) This follows from the fact that the Zariski closure of the image of $r$ in $\PGL(V)$ does not change upon restriction to an open normal subgroup $\Delta'\subset \Delta$ due to connectedness. Since (strong) irreducibility only depends on the Zariski closure of image in $\PGL(V)$, the lemma follows.
\end{proof}

The preceding lemma leads to the following proposition, which will be employed when studying the spinor norm of $\GSpin_{2n+1}$-valued Galois representations and when proving an automorphic multiplicity one result. Consider $r:\Delta\ra \GL(V)$ and $\chi:\Delta\ra \lql^\times$ as in the setup of Lemma \ref{lem:SugWooNew1}.
Let $r\simeq \oplus_{i=1}^t m_i r_i$ be a decomposition into mutually non-isomorphic irreducible $\Delta$-representations with multiplicities $m_i \in \Z_{>0}$.

\begin{proposition}\label{prop:SugWooNew2}
Assume that
\begin{itemize}
\item $\{\dim r_i\}_{i=1}^t$ are mutually distinct.
\item $\mathrm{im}(r)$ has connected Zariski closure in $\PGL(V)$.
\end{itemize}
Then each of the following is true.
\begin{enumerate}
\item If $r\simeq r^\vee \otimes \chi$ and if $\chi'\subset r\otimes r$ is a one-dimensional $\Delta$-submodule then $\chi'=\chi$.
\item If $r\simeq r\otimes \chi$ then $\chi=1$.
\end{enumerate}
\end{proposition}

\begin{proof}
(1) The first condition says $\oplus_i m_i r_i \simeq \oplus_i m_i r_i^\vee \otimes \chi$. Since $\dim(r_i)$ are distinct, $r_i\simeq r_i^\vee \otimes \chi$ for all $i$. Again by the dimension assumption, $\chi'\subset r_i \otimes r_i$ for some $i$. (If $\chi'\subset r_i\otimes r_j$ for $i\neq j$ then there is a nonzero $\Delta$-equivariant map from $r_i^\vee$ to $r_j \otimes \chi'^{-1}$, which must be an isomorphism by irreducibility; however this violates the dimension assumption.) Moreover if $r$ has connected image in $\textup{PGL}$ then so does $r_i$ (because $\mathrm{im}(r_i)$ is the image of $\mathrm{im}(r)$ by the projection map). Hence we are reduced to the case when $r$ is irreducible.

Now suppose $r$ is irreducible. Then $\chi'\subset r\otimes r$ implies $r\simeq r^\vee \otimes \chi'^{-1}$. We are assuming $r\simeq r^\vee \otimes \chi$, so we have
$r\simeq r\otimes(\chi'\chi^{-1})$. If $\chi'\neq \chi$ then part (1) of Lemma \ref{lem:SugWooNew1} says $r|_{\Delta'}$ is reducible for some open normal subgroup $\Delta'\subset \Delta$, contradicting part (2) of the same lemma. Therefore $\chi'=\chi$.

(2) Arguing similarly as in the proof of (1), we reduce to the case of irreducible $r$ and deduce that $\chi=1$ by using Lemma \ref{lem:SugWooNew1}.
\end{proof}

\section{$\GSpin$-valued Galois Representations}\label{sec:GSpin-valued}

In this section we study the notion of local conjugacy for the group $\GSpin_{2n+1}(\lql)$. In general it is not expected that local conjugacy implies (global) conjugacy of Galois representations: In the paper \cite[proof of Prop.~3.10]{LarsenTable} Larsen constructs a certain finite group $\Delta$ (in his text it is called $\Gamma$), which is a double cover of the (non-simple) Mathieu group $M_{10}$ in the $10$-th alternating group $A_{10}$. More precisely, he realizes $M_{10} \subset \SO_9(\lql)$ by looking at the standard representation of $A_{10} \subset \GL_{10}(\lql)$. Then $\Delta$ is the inverse image of $M_{10}$ in $\Spin_9(\lql)$. Let us just assume that $\Delta$ can be realized as a Galois group $s \colon \Gamma \surjects \Delta$. The group $\Delta$ comes with a map $\phi_1 \colon \Delta \to \Spin_9(\lql)$, and $\eta$ is the composition $\Delta \surjects M_{10} \surjects M_{10}/A_6 \simeq \Z/2\Z \injects \Spin_9(\lql)$. He defines $\phi_2(x) := \eta(x) \phi_1(x)$. We may define $r_1 := \phi_1 \circ s$ and $r_2 := \phi_2 \circ s$. Then the argument of Larsen shows that $\phi_1(\sigma)$ and $\phi_2(\sigma)$ are $\Spin_9(\lql)$ conjugate for every $\sigma \in \Gamma$, while $\phi_1$ and $\phi_2$ are not $\Spin_9(\lql)$-conjugate. The maps $\phi_i$ cannot be $\GSpin_9(\lql)$-conjugate: If we would have $\phi_2 = g \phi_1 g^{-1}$ for some $g \in \GSpin_9(\lql)$ then we can find a $z \in \lql^\times$, such that $h = gz \in \Spin_9(\lql)$ and $h \phi_1 h^{-1} = g \phi_1 g^{-1} = \phi_2$ which contradicts Larsen's conclusion. Thus, assuming the inverse Galois problem holds over $F$ for $\Delta$, $(r_1, r_2)$ is a pair of locally conjugate, but non-conjugate Galois representations.

In \cite{LarsenTable} Larsen explains that counterexamples may be constructed for all $\Spin_m(\lql)$ with $m \geq 8$ (for $m=8$, see \cite[Prop.~2.5]{larsen2})\footnote{In a private communication, Chenevier informed us that the group $\Spin_7(\lql)$ is also not acceptable.}. Proposition \ref{prop:WeaklyAcceptable} shows that any two locally conjugate $\GSpin_{2n+1}(\lql)$-valued Galois representations are conjugate up to a quadratic character twist (see also Example \ref{ex:PGL-example}).

\begin{lemma}\label{lem:for-mult-one}
Let $\rho \colon \Gamma \ra \GSpin_{2n+1}(\lql)$ be a semisimple representation that contains a regular unipotent element in the Zariski closure of its image. If $\chi \colon \Gamma\ra\lql^\times$ is a character such that $\spin(\rho)\simeq \spin(\rho)\otimes \chi$ then $\chi=1$.
\end{lemma}
\begin{proof}
This follows from part (2) of Proposition \ref{prop:SugWooNew2} applied to $r=\spin(\rho)$. For this we need to check the assumptions there; the only nontrivial part is the dimension hypothesis, which we verify as follows. Write $\ol{\textup{im}(\rho)}$ for the Zariski closure of $\textup{im}(\rho)$ in $\GSpin_{2n+1}$. The image of $\ol{\textup{im}(\rho)}$ in $\SO_{2n+1}$ is either $\PGL_2$, $\SO_{2n+1}$, or $G_2$ (last case only when $n=3$) by Proposition \ref{prop:3cases}. Hence we have either
\begin{enumerate}[label=(\roman*)]
\item $\SL_2\subset \ol{\textup{im}(\rho)} \subset \GL_2$,
\item $\PGL_2\subset \ol{\textup{im}(\rho)} \subset \PGL_2\times \G_m$,
\item $\Spin_{2n+1}\subset \ol{\textup{im}(\rho)} \subset \GSpin_{2n+1}$, or
\item $G_2\subset \ol{\textup{im}(\rho)} \subset G_2\times \G_m$ (the last case is possible only when $n=3$).
\end{enumerate}
In each case we explain the dimension hypothesis of Proposition \ref{prop:SugWooNew2} in detail when $\ol{\textup{im}(\rho)}$ is equal to $\GL_2$, $\PGL_2\times \G_m$, $\GSpin_{2n+1}$, or $G_2\times \G_m$, respectively, as the argument is essentially the same in general.
 Case (iii) is obvious. Cases (i) and (ii) follow from the fact that $\GL_2$ or $\PGL_2\times \G_m$-representations are determined by the dimension if the central character has weight 1 (namely the central $\G_m$ acts through the identity map). In case (iv), the underlying spin representation has dimension 8, and it decomposes as the direct sum of a one-dimensional representation with an irreducible 7-dimensional representation of $G_2$, whereas the $\G_m$ factor acts by weight 1. So the assumption still holds.
\end{proof}

\begin{proposition}\label{prop:SpinConjugacyA}
Let $r_1, r_2 \colon \Gamma \to \GSpin_{2n+1}(\lql)$ be two semisimple representations that are unramified at almost all places and locally conjugate. Assume the image of $r_1$ contains a regular unipotent element in the Zariski closure of its image. Then $r_1$ and $r_2$ are conjugate.
\end{proposition}
\begin{proof}
As in the proof of Lemma \ref{lem:for-mult-one}, the image of $r_1$ in $\SO_{2n+1}$ has connected Zariski closure. Recall from Lemma \ref{lem:C} that $\GSpin_{2n+1}$ is weakly acceptable.
By Proposition~\ref{prop:WeaklyAcceptable} (1) we may assume that $r_2 = r_1 \chi$, with $\chi$ a character taking values in the center of $\Spin_{2n+1}(\lql)$. By Lemma~\ref{lem:for-mult-one} we have $\chi = 1$.
\end{proof}

\section{The trace formula with fixed central character}\label{sect:TraceFormula}

In this section we recall the general setup for the trace formula with fixed central character.\footnote{More details are to be available in \cite{KSZ} and Dalal's thesis. Compare with Section 1 of \cite{Rog83}, Sections 2 and 3 of \cite{ArtSTF1}, or Section 3.1 of \cite{ArthurBook}. } For later use, we prove some instances of the Langlands functoriality for general reductive groups under the assumptions analogous to (St) at a finite place and cohomological at infinite places. (One might try to prove such a result by directly applying the $L^2$-Lefschetz formula by \cite{ArthurLefschetz} and \cite{GKM97} but the formula only tells us $0=0$ when $F\neq \Q$. This is why we need a version with fixed central character.)

Let $G$ be a connected reductive group over a number field $F$ with center $Z$. Write $A_Z$ for the maximal $\Q$-split torus in $\Res_{F/\Q}Z$ and set $A_{Z, \infty}:=A_Z(\R)^0$. Write $G(\A_F)^1$ for the subgroup of $G(\A_F)$ as on p.11 of \cite{Arthur1981} so that $G(\A_F)=G(\A_F)^1\times A_{Z, \infty}$. Consider a closed subgroup $\fkX\subset Z(\A_F)$ which contains $A_{Z, \infty}$ such that $Z(F)\fkX$ is closed in $Z(\A_F)$ (then $Z(F)\fkX$ is always cocompact in $Z(\A_F)$) and a continuous character $\chi:(\fkX\cap Z(F))\bs\fkX\ra\C^\times$. Such a pair $(\fkX,\chi)$ is called a central character datum.

In what follows we need to choose Haar measures consistently for various groups, but we will suppress these choices as this is quite standard. For instance the same Haar measures on $G(\A_F)$ and $\fkX$ have to be chosen for each term in the identity of Lemma \ref{lem:simple-TF} below.

Let $v$ be a place of $F$, and $\fkX_v$ a closed subgroup of $Z(F_v)$. Let $\chi_v:\fkX_v\ra \C^\times$ be a smooth character. Write $\cH(G(F_v),\chi_v^{-1})$ for the space of smooth compactly supported functions on $G(F_v)$ which transform under $\fkX_v$ via $\chi^{-1}_v$; if $v$ is archimedean, we also require functions in $\cH(G(F_v),\chi_v^{-1})$ to be $K_v$-finite for a maximal compact subgroup $K_v$ of $G(F_v)$. (We fix such a $K_v$ in this section, and use the same $K_v$ to compute relative Lie algebra cohomology. In the main case of interest, the choice of $K_v$ is made in \S\ref{sect:ShimuraDatum}.) Given a semisimple element $\gamma_v \in G(F_v)$ and an admissible representation $\pi_v$ of $G(F_v)$ with central character $\chi_v$ on $\fkX_v$, the orbital integral and trace character for $f_v \in \cH(G(F_v),\chi_v^{-1})$ are defined as follows. Below $I_{\gamma_v}$ denotes the connected centralizer of $\gamma_v$ in $G$.
\begin{eqnarray}
O_{\gamma_v}(f_v)&:=& \int_{I_{\gamma_v}(F_v)\bs G(F_v)} f_v(x^{-1}\gamma_v x)dx,\nonumber\\
\Tr(f_v|\pi_v)=\Tr \pi_v(f_v) &:=&\Tr\left( \int_{G(F_v)/\fkX_v} f_v(g)\pi_v(g) dg\right). \nonumber
\end{eqnarray}
Note that the trace is well defined since the operator $ \int_{G(F_v)/\fkX_v} f_v(g)\pi_v(g)dg $ is of finite rank if $v$ is finite and is of trace class if $v$ is infinite.

For our purpose, we henceforth assume the following.
\begin{itemize}
\item $\fkX=\fkX^ \infty\fkX_ \infty$ for an open compact subgroup $\fkX^ \infty\subset Z(\A^ \infty_F)$ and $\fkX_ \infty=Z(F_ \infty)$,
\item $\chi=\prod_v \chi_v$ with $\chi_v=1$ at every finite place $v$,
\end{itemize}
One defines the adelic Hecke algebra $\cH(G(\A_F),\chi^{-1})$ as well as orbital integrals and trace characters by taking a product over the local case considered above. Write $\Gamma_{\el,\fkX}(G)$ for the set of $\fkX$-orbits of elliptic conjugacy classes in $G(F)$. Let $L^2_{\disc,\chi}(G(F)\bs G(\A_F))$ denote the space of functions on $G(F)\bs G(\A_F)$ transforming under $\fkX$ by $\chi$ and square-integrable on $G(F)\backslash G(\A_F)^1/\fkX\cap G(\A_F)^1$. Write $\cA_\chi(G)$ for the set of isomorphism classes of cuspidal automorphic representations of $G(\A_F)$ whose central characters restricted to $\fkX$ are $\chi$. (In particular such representations are $G(\A_F)$-submodules of $L^2_{\disc,\chi}(G(F)\bs G(\A_F))$.)

For $f \in \cH(G(\A_F),\chi^{-1})$ we define invariant distributions $T^G_{\el,\chi}$ and $T^G_{\disc,\chi}$ by
\begin{eqnarray}
T^G_{\el,\chi}(f)&:=&\sum_{\gamma \in \Gamma_{\el,\fkX}(G)} \iota(\gamma)^{-1} \vol(I_\gamma(F)\bs I_\gamma(\A_F)/\fkX) O_{\gamma}(f),\nonumber\\
T^G_{\disc,\chi}(f)&:=&\Tr (f \,|\, L^2_{\disc,\chi}(G(F)\bs G(\A_F))).\nonumber
\end{eqnarray}
Similarly $T^G_{\cusp,\chi}$ is defined by taking trace on the space of square-integrable cuspforms. We omit the $G$ from the notation when clear from the context. In general we do not expect that $T_{\el,\chi}(f)=T_{\disc,\chi}(f)$ (unless $G/Z$ is anisotropic over $F$); the equality should hold only after adding more terms on both sides. However we do have $T_{\el,\chi}(f)=T_{\disc,\chi}(f)$ if $f$ satisfies some local hypotheses; this is often referred to as the simple trace formula.

We also need to consider the central character datum $(\fkX_0,\chi_0)$ with $\fkX_0:=A_{Z, \infty}$ and $\chi_0:=\chi|_{A_{Z, \infty}}$.
The quotient $\fkX\cap Z(F)\bs \fkX/\fkX_0$ is compact as it is closed in $Z(F)\backslash Z(\A_F)/A_{Z, \infty}$, which is compact.
We have a natural surjection $\cH(G(\A_F),\chi_0^{-1})\ra \cH(G(\A_F),\chi^{-1})$ given by
$$
f_0\mapsto \left(g\mapsto \ol{f}_0^\chi(g):= \int_{z \in \fkX\cap Z(F)\bs \fkX/\fkX_0} f_0(zg)\chi(z)dz\right).$$
Translating the function $f_0$ by $z$, define a function $f_0^z(g)=f_0(zg)$. Then
$$\frac{1}{\mathrm{volume}(\fkX\cap Z(F)\bs \fkX/\fkX_0)} \int_{z \in \fkX\cap Z(F)\bs \fkX/\fkX_0} T^G_{\star,\chi_0}(f^z_0)\chi(z)dz=T^G_{\star,\chi}(\ol{f}^\chi_0),\qquad
\star \in \{\el,\disc\}.$$
From here on we assume that $F$ is totally real and that $G(F_ \infty)$ admits discrete series representations (for Lefschetz functions to be nontrivial).

Let $\xi$ be an irreducible algebraic representation of $(\Res_{F/\Q} G)\otimes_\Q \C\simeq G\otimes_F F_ \infty$, also thought of as a continuous representation of $G(F_ \infty)$ on a complex vector space. Denote by $\chi_\xi:Z(F_ \infty)\ra \C^\times$ the restriction of $\xi^\vee$ to $Z(F_ \infty)$. Write $f_\xi=f^{G}_\xi \in \cH(G(F_ \infty),\chi_\xi^{-1})$ for a Lefschetz function (a.k.a. Euler-Poincar\'e function) associated with $\xi$ such that $\Tr \pi_ \infty(f_\xi)$ computes the Euler-Poincar\'e characteristic for the relative Lie algebra cohomology of $\pi_ \infty\otimes \xi$ for every irreducible admissible representation $\pi_ \infty$ of $G(F_ \infty)$ with central character $\chi_\xi$. See Appendix \ref{sec:Lefschetz} for details. Analogously we have the notion of Lefschetz functions at finite places as recalled in Appendix \ref{sec:Lefschetz}.

The following simple trace formula is standard for $\fkX=A_{G, \infty}$ and some other choices of $\fkX$ (such as $\fkX=Z(G(\A_F))$), but we want the result to be more flexible. Our proof reduces to the case that $\fkX=A_{G, \infty}$.

\begin{lemma}\label{lem:simple-TF} Consider the central character datum $(Z(F_ \infty),\chi)$.
Assume that $f_{\vst} \in \cH(G(F_\vst))$ is a (truncated) Lefschetz function and that $f_ \infty \in \cH(G(F_ \infty),\chi^{-1})$ is a cuspidal function. Then $$T^G_{\el,\chi}(f)=T^G_{\disc,\chi}(f)=T^G_{\cusp,\chi}(f).$$
\end{lemma}

\begin{proof}
Consider $f_0=\prod_v f_{0,v} \in \cH(G(\A_F),\chi_0^{-1})$, where $f_{0,v}:=f_v$ at all finite places $v$, and $f_{0, \infty}$ is a cuspidal function such that $\ol{f}^{\chi}_{0, \infty}$ (as defined above) and $f_ \infty$ have the same trace against every tempered representation of $G(F_ \infty)$. The existence of such an $f_{0, \infty}$ is guaranteed by the trace Paley--Wiener theorem.
By assumption $f_\vst=f_{0,\vst}$ is strongly cuspidal (Lemma \ref{lem:LefschetzIsStabilizing}). Hence the simple trace formula \cite[Cor. 7.3, 7.4]{Art88b} implies that
$$T^G_{\el,\chi_0}(f_0)=T^G_{\disc,\chi_0}(f_0)=T^G_{\cusp,\chi_0}(f_0).$$
We deduce the lemma by averaging over $Z(F_ \infty)/Z(F_ \infty)\cap G(F_ \infty)^1$ against $\chi$, noting that the average of $f_{0, \infty}$, namely the function $g\mapsto \int_{Z(F_ \infty)/Z(F_ \infty)\cap G(F_ \infty)^1} f_{0, \infty}(zg)\chi(z)dz$, recovers $f_ \infty$ up to a nonzero constant.
\end{proof}

Now we go back to the general central character data and discuss the stabilization for the trace formula with fixed central character under simplifying hypotheses. Assume that $G^*$ is quasi-split over $F$. Write $\Sigma_{\el,\fkX}(G^*)$ for the set of $\fkX$-orbits on the set of $F$-elliptic stable conjugacy classes in $G^*(F)$.
Define
$$\ST_{\mathrm{ell},\chi}^{G^*}(f) :=\tau(G^*) \sum_{\gamma \in \Sigma_{\el,\fkX}(G^*)} \tilde{\iota}(\gamma)^{-1} \SO^{G^*}_{\gamma,\chi}(f),\qquad f \in \cH(G^*(\A_F),\chi^{-1}),$$
where $\tilde\iota(\gamma)$ is the number of $\Gamma$-fixed points in the group of connected components in the centralizer of $\gamma$ in $G^*$, and $\SO^{G^*}_{\gamma,\chi}(f)$ denotes the stable orbital integral of $f$ at $\gamma$. If $G^*$ has simply connected derived subgroup (such as $\Sp_{2n}$ or $\GSp_{2n}$), we always have $\tilde\iota(\gamma)=1$.

Returning to a general reductive group $G$, let $G^*$ denote its quasi-split inner form over $F$ (with a fixed inner twist $G^*\simeq G$ over $\ol{F}$). Since $Z$ is canonically identified with the center of $G^*$, we may view $(\fkX,\chi)$ as a central character datum for $G^*$. Let $f^* \in \cH(G^*(\A_F),\chi^{-1})$ denote a Langlands-Shelstad transfer of $f$ to $G^*$.\footnote{The transfer for $\chi^{-1}$-equivariant functions is defined in terms of the usual transfer factors. Its existence is implied by the existence of the usual Langlands-Shelstad transfer (for compactly supported functions). For details, see \cite[3.3]{Xu-lifting} for instance.} Such a transfer exists in this fixed-central-character setup: One can lift $f$ via the surjection $\cH(G(\A_F))\ra \cH(G(\A_F),\chi^{-1})$ given by $\phi\mapsto (g\mapsto \int_{\fkX} \phi(gz)\chi(z)dz)$, apply the transfer from $\cH(G(\A_F))$ to $\cH(G^*(\A_F))$ (the transfer to quasi-split inner forms is due to Waldspurger), and then take the image under the similar surjection down to $\cH(G^*(\A_F),\chi^{-1})$.

Let $\vst$ be a finite place of $F$.

\begin{lemma}\label{lem:stabilization}
Assume that $f_{\vst}$ is a Lefschetz function.
Then
$T^G_{\el,\chi}(f)=\ST^{G^*}_{\el,\chi}(f^*)$.
\end{lemma}

\begin{proof}
Since $f_{\vst}$ is stabilizing (Lemma \ref{lem:LefschetzIsStabilizing}), this follows from \cite[Thm. 4.3.4]{LabesseStabilization} (specialized to $L=G$, $H=G^*$, $\theta=1$; note that the ``$(G^*,H)$-essentiel'' condition there is vacuous). 
\end{proof}

Let $S_0\subset S_ \infty\cup \{\vst\}$ be a finite subset, and $S$ a finite set of places containing~$S_0\cup S_ \infty$.
We assume that $G^*$ is unramified away from $S$ and fix a reductive model for $G^*$ over $\cO_F[1/S]$. (See \cite[Prop. 8.1]{ShinTemplier} for the existence of such a model.)

Let $G^*\simeq G$ over $\ol{F}$ be an inner twist which is trivialized at each place $v\notin S_0$ (i.e. up to an inner automorphism the inner twist descends to an isomorphism over $F_v$). In particular $G$ is unramified outside $S$, and fix a reductive model for $G$ over $\cO_F[1/S]$ as well. By abuse of notation we still write $G$ and $G^*$ for reductive models. The notion of unramified representations is taken relative to the hyperspecial subgroups given by these integral models. The inner twist determines an isomorphism $G^*_{F_v}\simeq G_{F_v}$ for $v\notin S_0$, well defined up to inner automorphisms of $G^*_{F_v}$. When $v\notin S$ we further have an isomorphism $G^*_{\cO_{F_v}}\simeq G_{\cO_{F_v}}$. We fix these isomorphisms and use them to transport representations between $G^*$ and $G$.

We prove weak transfers of a certain class of automorphic representations between $G^*$ and $G$.

\begin{proposition}\label{prop:JL-for-GSp}
Assume that the adjoint group of $G$ is nontrivial and simple over $F_{\vst}$. Consider $\pi$ and $\pi^\natural$ as follows:
\begin{itemize}
\item $\pi$ is a cuspidal automorphic representation of $G^*(\A_F)$ such that
\begin{itemize}
\item $\pi$ is unramified at every place apart from $S$,
\item $\pi_{\vst}$ is an unramified character twist of the Steinberg representation,
\item $\pi_ \infty$ is $\xi$-cohomological.
\end{itemize}
\item $\pi^\natural$ is a cuspidal automorphic representation of $G(\A_F)$ such that
\begin{itemize}
\item $\pi^\natural$ is unramified at every place apart from $S$,
\item $\pi^\natural_{\vst}$ is an unramified character twist of the Steinberg representation,
\item $\pi^\natural_{ \infty}$ is $\xi$-cohomological.
\end{itemize}
\end{itemize}
\begin{enumerate}
\item Suppose that $\xi$ has regular highest weight. Then for each $\pi$ (resp. $\pi^\natural$) as above, there exists $\pi^\natural$ (resp. $\pi$) as above such that $\pi^\natural_{v}\simeq \pi_v$ at every finite place $v \notin S\cup \{\vst\}$.
\item For each $\pi$ as above, suppose that
\begin{itemize}
\item $\pi$ is not one dimensional,
\item For every $\tau \in \cA_{\chi}(G^*)$ such that $\tau^S\simeq \pi^S$, if $\tau_ \infty$ is $\xi$-cohomological $\tau_{\vst}$ is an unramified twist of the Steinberg representation then $\tau_ \infty$ is a discrete series representation.
\end{itemize}
Then there exists $\pi^\natural$ as above such that $\pi^\natural_{v}\simeq \pi_v$ at every finite place $v \notin S\cup \{\vst\}$.
Moreover the converse is true with $G$ in place of $G^*$ and the roles of $\pi$ and $\pi^\natural$ switched.
\end{enumerate}
\end{proposition}

\begin{remark}
The adjoint group of $G$ is assumed to be simple over $F_{\vst}$ in order to apply Proposition \ref{prop:TiwstedLefschetzFunction}. Without the assumption
we may have a mix of trivial and Steinberg representations (up to unramified twists) for $\pi_{\vst}$ corresponding to simple factors of $G$ at $\vst$.
\end{remark}

\begin{remark}\label{rem:JL}
Corollary \ref{cor:Pi-is-in-Disc-series-Packet} tells us that the condition in (2) for the transfer from $G^*$ to $G$ is satisfied by $G^*=\GSp_{2n}$. Later we will see in Corollary \ref{cor:Pi-is-in-Disc-series-Packet-2} that the same is also true for a certain inner form of $\GSp_{2n}$.
\end{remark}

\begin{proof}
We will only explain how to go from $\pi$ to $\pi^\natural$ as the opposite direction is proved by the exact same argument. Let $f =\prod_v f_v$ be such that $f_v \in \cH^{\unr}(G(F_v))$ for finite places $v\notin S$, $f_{\vst}$ is a Lefschetz function at $\vst$, and $f_ \infty$ is a Lefschetz function for $\xi$. Then we can choose the transfer $f^*=\prod_v f^*_v$ such that $f^*_v=f_v$ for all $v\notin S_ \infty\cup \{\vst\}$, $f^*_{\vst}$ is a Lefschetz function, and $f^*_ \infty$ is a Lefschetz function for $\xi$. We know from Lemmas \ref{lem:LefschetzMatching-finite} and \ref{lem:LefschetzMatching} that $f_\vst$ (resp. $f_ \infty$) and $f^*_\vst$ (resp. $f^*_ \infty$) are associated up to a nonzero constant. Hence $cf$ and $f^*$ are associated for some $c \in \C^\times$. The preceding two lemmas imply that
$$ T^{G^*}_{\cusp,\chi}(f^*)=\ST^{G^*}_{\el,\chi}(f^*)=c\cdot T^{G}_{\cusp,\chi}(f).$$
By linear independence of characters, we have
$$
\sum_{\tau \in \cA_{\chi}(G^*)\atop \tau^S\simeq \pi^S} m(\tau)\Tr (f^*_S|\tau_S) =c \cdot \sum_{\tau^\natural \in \cA_{\chi}(G)\atop \tau^{\natural,S}\simeq \pi^{S}}m(\tau^\natural) \Tr (f_S|\tau^\natural_S) .
$$
Let us prove (1). Choose $f_v=f^*_v$ at finite places $v$ in $S$ but outside $\{\vst\}\cup S_0\cup S_ \infty$ such that $\Tr(f^*_v|\pi_v)>0$. (This is vacuous if there is no such $v$.) At infinite places, as soon as $\Tr(f^*_v|\tau_v)\neq 0$ at $v| \infty$, the regularity condition on $\xi$ implies that $\tau_v$ is a discrete series representations and that $\Tr(f^*_v|\tau_v)=(-1)^{q(G)}$ by Vogan--Zuckerman's classification of unitary cohomological representations. At $v=\vst$, whenever $\Tr(f^*_\vst|\tau_\vst)\neq 0$ (which is true for $\tau=\pi$), the unitary representation $\tau_\vst$ is either an unramified twist of the trivial or Steinberg representation by Lemma \ref{prop:ClassicalLefschetzFunction}. If $G_{\vst}$ is anisotropic modulo center then the Steinberg representation \emph{is} the trivial representation. In case $G_{\vst}$ is isotropic modulo center, if $\tau_\vst$ were one-dimensional then the global representation $\tau$ would be one-dimensional by a well known strong approximation argument (e.g. \cite[Lem. 6.2]{KST2} for details), implying that $\tau_ \infty$ cannot be tempered.

All in all, for all $\tau$ as above such that $\Tr (f^*_S|\tau_S)\neq 0$, we see that $\tau_{\vst}$ is an unramified twist of the Steinberg representation and that $\Tr (f^*_S|\tau_S)$ has the same sign. Moreover $\tau=\pi$ contributes nontrivially to the left sum by our assumption. Therefore the right hand side is nonzero, i.e. there exists $\pi^\natural \in \cA_{\chi}(G)$ such that $\pi^{\natural,S}\simeq \pi^{\natural,S}$ and $m(\pi^\natural) \Tr (f_S|\pi^\natural_S)\neq 0$. The nonvanishing of trace confirms the conditions on $\pi^\natural$ at $\vst$ and $ \infty$ by the same argument as above.

Now we prove (2). Make the same choice of $f_v=f^*_v$ at finite places $v$ in $S$. Since one-dimensional representations of $G(F_{v'})$ are excluded by assumption, the condition $\Tr(f^*_\vst|\tau_\vst)\neq 0$ implies that $\tau_\vst$ is an unramified twist of the Steinberg representation. By the assumption $\tau_ \infty$ is then a discrete series representation. Hence the nonzero contributions from $\tau$ on the left hand side all has the same sign. Thereby we deduce the existence of $\pi^\natural$ as in (1).
\end{proof}

As before $G^*$ is a quasi-split group over a totally real field $F$. Let $E/F$ be a finite cyclic extension of totally real fields such that $[E:F]$ is a prime. Write $\theta$ for a generator of the group $\Gal(E/F)$. Set $\tilde G^*:=\Res_{E/F}(G^*_E)$, which is equipped with a natural action of $\theta$ defined over $F$. Let $S$ be a finite set of places of $F$ such that the group $G^*$ and the extension $E/F$ are unramified outside $S$, so that $\tilde G^*$ is also unramified apart from $S$. Fix a reductive model of $G^*$ over $\cO_F[1/S]$ as before. By base change followed by $\Res_{E/F}$, this gives rise to a reductive model of $\tilde G^*$ over $\cO_F[1/S]$. The models determine hyperspecial subgroups of $G^*$ and $\tilde G^*$ away from $S$, thus also unramified Hecke algebras and unramified representations of $G^*$ and $\tilde G^*$.

For each place $v$ of $F$ put $E_v:= E\otimes_F F_v$. Canonically $E_v\simeq \prod_{w|v} E_w$ with $w$ running over places of $E$ above $v$. Let $\BC_{E/F}:\Irr^{\unr}(G^*(F_v))\rightarrow \Irr^{\unr}(G^*(E_v))$ denote the base change map for unramified representations. Writing $\BC^*_{E/F}:\cH^{\unr}(G^*(E_v))\ra \cH^{\unr}(G^*(F_v))$ for the base change morphism of unramified Hecke algebras, we have from Satake theory that
\begin{equation}\label{eq:BC-identity}
\Tr (BC^*_{E/F}(f_v)\,|\,\pi_v)=\Tr (f_v\,|\, BC_{E/F}(\pi_v)).
\end{equation}
Write $\xi=\otimes_{v| \infty} \xi_v$ as a representation of $G(F_ \infty)\simeq\prod_{v| \infty} G(F_v)$. For each infinite place $v$ of $F$, define $\xi_{E,v}:=\otimes_{w|v} \xi_v$ as a representation of $G(E_v)\simeq \prod_{w|v} G(E_w)$, where $w$ runs over places of $E$ dividing $v$ (here the two isomorphisms are canonical). Set $\xi:=\otimes_{v| \infty} \xi_{E,v}$.

\begin{proposition}\label{prop:BC-for-GSp}
Let $\pi$ be a cuspidal automorphic representation of $G^*(\A_F)$ such that
\begin{itemize}
\item $\pi$ is unramified at all finite places apart from a finite set $S$,
\item $\pi_{\vst}$ is an unramified character twist of the Steinberg representation,
\item $\pi_ \infty$ is $\xi$-cohomological.
\end{itemize}
Suppose that either $\xi$ has regular highest weight or that the condition for $\pi$ in Proposition \ref{prop:JL-for-GSp}.(2) is satisfied. (This is always true for $G^*=\GSp_{2n}$, cf. Remark \ref{rem:JL}.)
Then there exists a cuspidal automorphic representation $\pi_E$ of $G^*(\A_E)$ such that
\begin{itemize}
\item $\pi_E$ is unramified at all finite places apart from $S$,
\item $\pi_{E,\vst}$ is an unramified character twist of the Steinberg representation,
\item $\pi_{E, \infty}$ is $\xi_E$-cohomological,
\end{itemize}
and moreover $\pi_{E,v}\simeq \BC_{E/F}(\pi_v)$ at every finite place $v \notin S$.
\end{proposition}

\begin{proof}
We will be brief as our proposition and its proof are very similar to those in Labesse's book \cite[\S4.6]{LabesseStabilization}, and also as the proof just mimics the argument for Proposition \ref{prop:JL-for-GSp} in the twisted case. In particular we leave the reader to find further details about the twisted trace formula for base change in \textit{loc. cit}. Strictly speaking one has to incorporate the central character datum $(\fkX,\chi)$ above in Labesse's argument, but this is done exactly as in the untwisted case.

We begin by setting up some notation. Take $\tilde K$ to be a sufficiently small open compact subgroup of $G^*(\A_E^ \infty)$ such that $N_{E/F}\colon Z(\A_E) \ra Z(\A_F)$ maps $Z(E)\cap \tilde K$ into $Z(F)\cap K$. Let $\tilde\fkX:=Z(E)\cap \tilde K$, redefine $\fkX$ to be $N_{E/F}(\tilde\fkX)$ and restrict $\chi$ to this subgroup, $\tilde\chi:=\chi\circ N_{E/F}$.) Let $\tilde \xi$ denote the representation $\otimes_{E\hra \C}\xi$ of $G^*_{2n}(E\otimes_\Q \C)=\prod_{E\hra \C}G^*(F\otimes_\Q \C)$ (both indexed by $F$-algebra embedding $E\hra \C$).

We choose the test functions $$\tilde f=\prod_w \tilde f_w \in \cH(G^*(\A_E),\tilde\chi^{-1})\quad\mbox{and}\quad f=\prod_v f_v \in \cH(G^*(\A_F),\chi^{-1})$$ as follows. If $v=\vst$ then $\tilde f_v=\prod_{w|v} \tilde f_w$ and $f_v$ are set to be Lefschetz functions as in Appendix \ref{sec:Lefschetz}. So $\tilde f_v$ and $f_v$ are associated up to a nonzero constant by Lemma \ref{lem:SteinbergAssociated}. If $v$ is a finite place of $F$ contained in $S$ then choose $f_v$ to be the characteristic function on a sufficiently small open compact subgroup $K_v$ of $G^*(F_v)$ such that $\pi_v^{K_v}\neq 0$. By \cite[Prop. 3.3.2]{LabesseStabilization} $f_v$ is a base change transfer of some function $\tilde f_v=\prod_{w|v} \tilde f_w \in \cH(G^*(E_v),\tilde\chi^{-1})$. Given a finite place $v\notin S$ and each place $w$ of $E$ above it, let $\tilde f_w$ be an arbitrary function in $\cH^{\unr}(G^*(E_w,\tilde\chi^{-1}))$. The image of $\tilde f_v$ in $\cH^{\unr}(G^*(F_v))$ under the base change map is denoted by $f_v$. At infinite places let $\tilde f_ \infty=\prod_{w| \infty} \tilde f_w$ be the twisted Lefschetz function determined by $\tilde\xi$ and $f_ \infty=\prod_{v| \infty} f_v$ the usual Lefschetz function for $\xi$. Again $\tilde f_ \infty$ and $f_ \infty$ are associated up to a nonzero constant by Lemma \ref{lem:TwistedLefschetz}. By construction $\tilde f $ and $c f$ are associated for some $c \in \C^\times$.

We write $T^{\tilde G^*\theta}_{\cusp,\tilde \chi}$ and $T^{\tilde G^*\theta}_{\el,\tilde \chi}$ for the cuspidal and elliptic expansions in the base-change twisted trace formula, which are defined analogously as their untwisted analogues. (For instance $T^{\tilde G^*\theta}_{\el,\tilde \chi}$ is defined as in $T^L_e$ on p.98 of \cite{LabesseStabilization} with $L=\tilde G^*\theta$ but making the obvious adjustment to account for the central character as earlier in this section.)
Just like the trace formula for $G$ and $f$, the twisted trace formula for $\tilde G$ and $\tilde f$ as well as its stabilization simplifies greatly exactly as in Lemmas \ref{lem:simple-TF} and \ref{lem:stabilization} in light of Lemmas \ref{lem:SteinbergAssociated} and \ref{lem:TwistedLefschetz}. So we have
$$ T^{\tilde G^*\theta}_{\cusp,\tilde \chi}(\tilde f)=T^{\tilde G^*\theta}_{\el,\tilde \chi}(\tilde f)=c\cdot \ST^{G^*}_{\el,\chi}(f)=c \cdot T^{G^*}_{\cusp,\chi}(f).$$
By linear independence of characters and the character identity \eqref{eq:BC-identity}, we have
$$\sum_{\tilde\tau \in \cA_{\tilde \chi}(G^*(\A_E))\atop \tilde\tau^\theta\simeq \tilde\tau,~\tilde\tau^S\simeq \BC_{E/F}(\pi^S)} \tilde m(\tilde\tau)\Tr^{\theta} (\tilde f_S|\tilde\tau_S) =c \sum_{\tau \in \cA_{\chi}(G^*(\A_F))\atop \tau^{S}\simeq \pi^{S}}m(\tau) \Tr (f_S|\tau_S) ,$$
where $\Tr^\theta$ denotes the $\theta$-twisted trace (for a suitable intertwining operator for the $\theta$-twist), and $ \tilde m(\tilde\tau)$ denotes the relative multiplicity of $\tilde\tau$ as defined on p.106 of \cite{LabesseStabilization}. The right hand side is nonzero as in the proof of Proposition \ref{prop:JL-for-GSp}. Therefore there exists $\pi_E:=\tilde\tau \in \cA_{\chi}(G^*(\A_E))$ contributing nontrivially to the left hand side. By construction of $\tilde f$ such a $\pi_E$ satisfies all the desired properties.
\end{proof}

\section{Cohomology of certain Shimura varieties of abelian type}\label{sect:ShimuraDatum}

In this section we construct a Shimura datum and then state the outcome of the Langlands-Kottwitz method on the formula computing the trace of the Frobenius and Hecke operators in the case of good reduction.

We first construct our Shimura datum $(\Res_{F/\Q} G, X)$. The group $G/F$ is a certain inner form of the quasi-split group $G^* := \Gsp_{2n, F}$. We recall the classification of such inner forms, and then define our $G$ in terms of this classification. The inner twists of $\Gsp_{2n, F}$ are parametrized by the cohomology $\uH^1(F, \PSp_{2n})$. Kottwitz defines in \cite[Thm. 1.2]{KottwitzEllipticSingular} for each $F$-place $v$ a morphism of pointed sets
$$
\alpha_v \colon \uH^1(F_v, \PSp_{2n}) \to \pi_0(Z(\Spin_{2n+1}(\C))^{\Gamma_v})^D \simeq \Z/2\Z,
$$
where $(\cdot)^D$ denotes the Pontryagin dual. If $v$ is finite, then $\alpha_v$ is an isomorphism. If $v$ is infinite, then \cite[(1.2.2)]{KottwitzEllipticSingular} tells us that 
$$
\ker(\alpha_v) = \tu{im}[\uH^1(F_v, \Sp_{2n}) \to \uH^1(F_v, \PSp_{2n})],
$$
$$
\tu{im}(\alpha_v) = \ker[\pi_0(Z(\Spin_{2n+1}(\C))^{\Gamma_v})^D \to \pi_0(Z(\Spin_{2n+1}(\C))^D)].
$$
Thus $\alpha_v$ is surjective, with trivial kernel as $\uH^1(\R, \Sp_{2n})$ vanishes by \cite[Chap. 2]{PlatonovRapinchuk}. However $\alpha_v$ is not a bijection (when $n\ge 2$). In fact $\alpha_v^{-1}(1)$ classifies unitary groups associated to Hermitian forms over the Hamiltonian quaternion algebra over $F_v$ with signature $(a,b)$ with $a+b=n$ modulo the identification as inner twists between signatures $(a,b)$ and $(b,a)$. (See \cite[3.1.1]{TaibiTransfer} for an explicit computation of $\uH^1(F_v, \PSp_{2n})$.) So $\alpha_v^{-1}(1)$ has cardinality $\lfloor \frac n2 \rfloor+1$. There is a unique nontrivial inner twist of $\Gsp_{2n, F_v}$ (up to isomorphism), to be denoted by $\tu{GSp}_{2n, F_v}^{\tu{cmpt}}$, such that $\tu{GSp}_{2n, F_v}^{\tu{cmpt}}$ is compact modulo center. It comes from a definite Hermitian form.

By \cite[Prop.~2.6]{KottwitzEllipticSingular} we have an exact sequence
$$
\Ker^1(F, \PSp_{2n}) \injects \uH^1(F, \PSp_{2n}) \to \bigoplus_v \uH^1(F_v, \PSp_{2n}) \overset \alpha \to \Z/2\Z,
$$
where $\alpha$ sends $(c_v) \in \uH^1(F, G(\A_F \otimes_F \li \Q))$ to $\sum_v \alpha_v(c_v)$. By \cite[Lem.~4.3.1]{KottwitzCuspidalTempered} we have $\Ker^1(F, \PSp_{2n}) = 1$.
We conclude
\begin{equation}\label{eq:H1ofPSp}
\uH^1(F, \PSp_{2n}) = \lbr (x_v) \in \bigoplus_v \uH^1(F_v, \PSp_{2n})\,|\, \sum_v \alpha(x_v) = 0 \in \Z/2\Z \rbr.
\end{equation}
In particular there exists an inner twist $G$ of $\GSp_{2n, F}$ such that:
\begin{itemize}
\item For the infinite places $y \in \cV_ \infty \bs \{\vinfty\}$ the group $G_{F_y}$ is isomorphic to $\GSpcmptFv$. For $y = \vinfty$ we have $G_{F_y} = \GSp_{2n, F_y}$. (Recall: $\vinfty$ is the infinite $F$-place corresponding to the embedding of $F$ into $\C$ that we fixed in the Notation section.)
\item If $[F:\Q]$ is odd, we take $G_{\A_F^ \infty} \simeq \GSp_{2n, \A_F^ \infty}$.
\item If $[F:\Q]$ is even we fix a finite $F$-place $\vst$, and take $G_{\A_F^{ \infty,\vst}} \simeq \GSp_{2n, \A_F^{ \infty,\vst}}$. The form $G_{F_{\vst}}$ then has to be the unique nontrivial inner form of $\GSp_{2n, F_{\vst}}$.
\end{itemize}
More concretely $G$ can be defined itself as a similitude group but we do not need it in this paper.

Let $\SS$ be the Deligne torus $\Res_{\C/\R} \Gm$. Over the real numbers the group $(\Res_{F/\Q} G)_\R$ decomposes into the product $\prod_{y \in \cV_ \infty} G \otimes_F F_y$. Let $I_n$ be the $n\times n$-identity matrix and $A_n$ be the $n\times n$-matrix with all entries $0$, except those on the anti-diagonal, where we put $1$. Let $h_0 \colon \SS \to (\Res_{F/\Q} G)_\R$ be the morphism given by
\begin{equation}\label{eq:Morphismh}
\SS(R) \to G(F \otimes_{\Q} R), \quad a+bi \mapsto \lhk {\vierkant {a I_n} {b A_n} {-b A_n} {a I_n}}_{\vinfty}, 1, \ldots, 1 \rhk \in \prod_{y \in \cV_ \infty} (G \otimes_F F_y)(R),
\end{equation}
for all $\R$-al\-ge\-bras $R$ (the non-trivial component corresponds to the non-compact place $\vinfty \in \cV_ \infty$). We let $X$ be the $(\Res_{F/\Q} G)(\R)$-con\-ju\-gacy class of $h_0$. This set $X$ can more familiarly be described as \emph{Siegel double half space} $\iH_n$, \ie the $n(n+1)/2$-di\-men\-sio\-nal space consisting of complex symmetric $n\times n$-ma\-tri\-ces with definite (positive or negative) imaginary part. Let us explain how the bijection $X \simeq \iH_n$ is obtained. The group $\Gsp_{2n}(\R)$ acts transitively on $\iH_n$ via fractional linear transformations: $\kleinvierkant A B C D Z$ $:=$ $(A Z + B)(C Z + D)^{-1}$, $\kleinvierkant A B C D \in \GSp_{2n}(\R)$, $Z \in \iH_n$ \cite{SiegelSymplecticGeometry} (in these formulas the matrices $A, B, C, D, Z$ are all of size $n \times n$). The place $\vinfty$ induces a surjection $(\Res_{F/\Q} G)(\R) \surjects \Gsp_{2n}(\R)$ and via this surjection we let $(\Res_{F/\Q} G)(\R)$ act on $\iH_n$. The stabilizer $K_ \infty:=\Stab_{(\Res_{F/\Q}G)(\R)}(i I_n)$ of the point $iI_n \in \iH_n$ has the form $K_ \infty := K_{v_ \infty}\times \prod_{v \in \cV_ \infty \bs \{\vinfty\}} G(F_v)$, where $K_{v_ \infty}$ is the stabilizer of $iI_n$ in $\GSp_{2n}(F_{v_ \infty})$. Similarly, the stabilizer of $h_0 \in X$ is equal to $K_ \infty$. Thus there is an isomorphism $X \simeq \iH_n$ under which $h_0$ corresponds to $i I_n$. It is a routine verification that Deligne's axioms (2.1.1.1), (2.1.1.2), and (2.1.1.3) for Shimura data \cite{DeligneCorvallis} are satisfied for $(\Res_{F/\Q} G, \iH_n)$. Since moreover the Dynkin diagram of the group $G_\ad$ is of type $C$, it follows from Deligne \cite[Prop. 2.3.10]{DeligneCorvallis} that $(\Res_{F/\Q} G, \iH_n)$ is of abelian type.

We write $\mu=(\mu_y)_{y \in \cV_ \infty} \in X_*(T)^{\cV_ \infty}$ for the cocharacter of $\Res_{F/\Q} G$ such that $\mu_y=1$ if $y\neq v_ \infty$ and $\mu_{v_ \infty}(x)= \kleinvierkant {xI_n}{0}{0}{I_n}$. Then $\mu$ is conjugate to the holomorphic part of $h_0 \otimes \C$. An element $\sigma \in \Gal(\li \Q/\Q)$ sends $\mu = (\mu_y)$ to $\sigma(\mu) = (\mu_{\sigma \circ y})$. Thus the conjugacy class of $\mu$ is fixed by $\sigma$ if and only if $\sigma \in \Gal(\li \Q/F)$. Therefore the reflex field of $(\Res_{F/\Q} G, \iH_n)$ is $F$ (embedded in $\C$ via $v_ \infty$). For $K \subset G(\A^ \infty_F)$ a sufficiently small compact open subgroup, write $\Sh_K / F$ for the corresponding Shimura variety. In case $F = \Q$ the datum $(\Res_{F/\Q} G, \iH_n)$ is the classical Siegel datum (of PEL type) and $\Sh_K$ are the usual non-compact Siegel modular varieties.
If $[F : \Q] > 1$ so that $G$ is anisotropic modulo center, then it follows from the Baily-Borel compactification \cite[Thm. 1]{BailyBorel} that $\Sh_K$ is projective. Whenever $[F:\Q] > 1$ the datum $(\Res_{F/\Q} G, \iH_n)$ is not of PEL type. If moreover $n=1$ then the $\Sh_K$ have dimension 1 and they are usually referred to as Shimura curves, which have been extensively studied in the literature.

Let $\xi=\otimes_{y| \infty}\xi_y$ be an irreducible algebraic representation of $(\Res_{F/\Q} G)\times_\Q \C=\prod_{y| \infty} G(\ol{F}_y)$ with each $\ol{F}_y$ canonically identified with $\C$. The central character $\omega_{\xi_y}$ of $\xi_y$ has the form $z\mapsto z^{w_y}$ for some integer $w_y \in \Z$.

\begin{lemma}\label{lem:alg-central-char}
If there exists a $\xi$-cohomological discrete automorphic representation $\pi$ of $G(\A_F)$ then $w_y$ has the same value for every infinite place $y$ of $F$.
\end{lemma}

\begin{proof}
Under the assumption, the central character $\omega_\pi \colon F^\times \backslash \A^\times_F \to \C^\times$ is an (L-)algebraic Hecke character. Hence $\omega_{\pi}=\omega_0 |\cdot|^w$ for a finite Hecke character $\omega_0$ and an integer $w$ by Weil \cite{Wei56}. It follows that $w=w_y$ for every infinite place $y$.
\end{proof}
In light of the lemma we henceforth make the hypothesis as follows, which implies that $\xi$ restricted to the center $Z(F)=F^\times$ of $(\Res_{F/\Q} G)(\Q)=\GSp_{2n}(F)$ is the $w$-th power of the norm character $\uN_{F/\Q}$.
\begin{itemize}
\item[\textbf{(cent)}] $w_y$ is independent of the infinite place $y$ of $F$ (and denoted by $w$).
\end{itemize}

Following \cite[2.1]{Car86} (especially paragraph 2.1.4) we construct an $\ell$-adic sheaf on $\Sh_K$ for each sufficiently small open compact subgroup $K$ of $G(\A_F^ \infty)$ from the $\ell$-adic representation $\iota\xi=\xi\otimes_{\C,\iota}\lql$. For simplicity we write $\cL_{\iota\xi}$ for the $\ell$-adic sheaf (omitting $K$). It is worth pointing out that the construction relies on the fact that $\xi$ is trivial on $Z(F)\cap K$ for small enough $K$. (For a fixed open compact subgroup $K_0$, we see that $Z(F)\cap K_0\subset \mathcal{O}_F^\times$ is mapped to $\{\pm 1\}$ under $N_{F/\Q}$. So $\xi$ is trivial on $Z(F)\cap K$ for every $K$ contained in some subgroup of $K_0$ of index at most $2$.) Similarly, we write $\cL_\xi$ for the (complex) local system on $\Sh_K(\C)$
attached to $\xi$.

Without loss of generality we assume throughout that $K$ decomposes into a product $K = \prod_{v\nmid \infty} K_v$ where $K_v \subset G(F_v)$ is a compact open subgroup. We call $(G, K)$ \emph{unramified} at a finite $F$-place $v$ if the group $K_v$ is hyperspecial in $G(F_v)$. If so, fix a smooth reductive model $\underline G$ of $G_{F_v}$ over $\mathcal{O}_{F_v}$ such that $\underline G(\mathcal{O}_{F_v}) = K_v$. We say that $(G,K)$ is unramified at a rational prime $p$ if it is unramified at all $F$-places $v$ above $p$. For the moment, we do not assume these conditions on $(G,K)$.

We fix an algebraic closure $\li \Q\subset \C$ of $F$ with respect to $v_ \infty: F\hra \C$. Our primary interest lies in the $\ell$-adic \'etale cohomology with compact support
\begin{equation}\label{eq:GrothendieckGp}
\uH_c(\Sh_K, \cL_{\iota\xi}) := \sum_{i=0}^{n(n+1)} (-1)^i [\uH^i_{\uc}(\Sh_{K, \li \Q}, \cL_{\iota\xi})]
\end{equation}
as a virtual $\cH_{\lql}(G(\A^ \infty_F)\sslash K) \times \Gamma$-mo\-dule. Let us make the word ``virtual", and the notation $[\cdot]$ more precise. As in the book of Harris-Taylor, if $\Omega$ is a topological group that is locally profinite \cite[p. 23]{HarrisTaylor}, we let $\tu{Groth}(\Omega)$ be the abelian group of formal, possibly infinite, sums$\sum_{\Pi \in \tu{Irr}(\Omega)} n_\Pi\Pi$, where $\tu{Irr}(\Omega)$ is the set of isomorphism classes of admissible representations of $\Omega$ in $\lql$-vector spaces. Given an admissible $\Omega$-representation $\Pi$, it defines an element $[\Pi] \in \tu{Groth}(\Omega)$ (bottom of page 23 of \cite{HarrisTaylor}). In particular we may take $\Omega = \Gamma$, $G(\A_F^ \infty)$ and $\Gamma \times G(\A_F^ \infty)$, or, by proceeding similarly, take $\Omega$ equal to a Hecke algebra times a group, such as $\cH_{\lql}(G(\A^ \infty_F)\sslash K) \times \Gamma$. Taking a direct limit over sufficiently small open compact subgroups $K$ of $G(\A_F^ \infty)$, we similarly obtain a $G(\A_F^ \infty)\times \Gamma$-module with admissible $G(\A_F^ \infty)$-action and a virtual $G(\A_F^ \infty)\times \Gamma$-module
$$
\uH^i_c(\Sh, \cL_{\iota\xi}):=\varinjlim_{K} \uH^i_{\et}(\Sh_{K,\li \Q}, \cL_{\iota\xi}),\quad \uH_c(\Sh, \cL_{\iota\xi}) := \sum_{i=0}^{n(n+1)} (-1)^i [\uH^i_c(\Sh, \cL_{\iota\xi})].
$$

Let us introduce some more cohomology spaces to be used mainly for the $F=\Q$ case in the next section. For $K$ as above and for each $i \in \Z_{\ge0}$, we write $\uH^i_{0}(\Sh_K,\cL_{\xi})$, $\uH^i_{(2)}(\Sh_K,\cL_{ \xi})$, and $\uIH^i(\Sh_K,\cL_{\iota \xi})$ for the cuspidal, $L^2$, and $\ell$-adic intersection cohomology of $\Sh_K$, respectively. (The first two spaces are as in \cite[\S6]{FaltingsCohomology}, considering $\Sh_K$ as a complex manifold with respect to $F\hra \C$ corresponding to $v_ \infty$. The last one is the cohomology of the intermediate extension of $\cL_{\iota \xi}$ to the Bailey-Borel compactification \cite[Sect. 2.2]{Morel-GSp}.) By taking direct limits over $K$, we obtain admissible $G(\A_F^ \infty)$-representations $\uH^i_{0}(\Sh,\cL_{ \xi})$, $\uH^i_{(2)}(\Sh,\cL_{ \xi})$, and $\uIH^i(\Sh,\cL_{\iota \xi})$. The last one is equipped with commuting actions of $G(\A_F^ \infty)$ and $\Gamma$.

We are going to apply the Langlands-Kottwitz method to relate the action of Frobenius elements of $\Gamma$ at primes of good reduction to the Hecke action on the compact support cohomology. In the PEL case of type A or C, Kottwitz worked it out in \cite{KottwitzAnnArbor,KottwitzPoints} including the stabilization. As we are dealing with Shimura varieties of abelian type, we import the result from \cite{KSZ}. In fact the stabilization step is very simple under hypothesis (St).

Suppose that $(G,K)$ is unramified at a prime $p$. Let $\fkp$ be a finite $F$-place above $p$ with residue field $k(\fkp)$. We fix an isomorphism $\iota_p:\C\simeq \li \Q_p$ such that $\iota_p v_ \infty:F\hra \li \Q_p$ induces the place $\fkp$ of $F$. Let us introduce some more notation.
\begin{itemize}
\item $r:=[k(\fkp):\mathbb F_p] \in \Z_{\ge 1}$. 
\item For $j \in \Z_{\ge 1}$ denote by $\Q_{p^j}$ the unramified extension of $\Q_p$ of degree $j$, and $\Z_{p^j}$ its integer ring.
\item Write $F_p := F \otimes_\Q \qp$, $F_{p, j} := F \otimes_{\Q} \Q_{p^j}$, and $\mathcal{O}_{F_{p, j}} := \mathcal{O}_F \otimes_\Z \Z_{p^j}$.
\item Let $\sigma$ be the automorphism of $F_{p, j}$ induced by the (arithmetic) Frobenius automorphism on $\Q_{p^j}$ and the trivial automorphism on $F$.
\item Let $\cH(G(\A^ \infty_F))$ (resp. $\cH(G(F_p))$, $\cH(G(F_{p, j}))$) be the convolution algebra of compactly supported smooth complex valued functions on $G(\A_F^ \infty)$ (resp. $G(F_p)$, resp. $G(F_{p, j})$), where the convolution integral is defined by the Haar measure giving the group $K$ (resp. $K_p$ resp. $\underline G(\mathcal{O}_{F_{p, j}})$) measure $1$.
\item We write $\Hunr(G(F_p))$ (resp. $\Hunr(G(F_{p, j}))$) for the spherical Hecke algebra, \ie the algebra of $K_p$ (resp. $\underline G(\mathcal{O}_{F_{p, j}})$)-bi-invariant functions in $\cH(G(F_p))$ (resp. $\cH(G(F_{p, j}))$).
\item Define $\phi_j \in \Hunr(G(F_{p, rj}))$ to be the characteristic function of the double coset
$\underline G(\mathcal{O}_{F_{p, rj}}) \cdot \mu(p^{-1}) \cdot \underline G(\mathcal{O}_{F_{p, rj}})$ in $G(F_{p, rj})$.
\item $N_ \infty:=|\Pi_\xi^{G(F_ \infty)}|\cdot |\pi_0(G(F_ \infty)/Z(F_ \infty))|=2^{n-1}\cdot 2=2^n$.
\item $\ep(\tau_ \infty \otimes \xi):=\sum_{i=0}^ \infty (-1)^i \dim \uH^i(\Lie G(F_ \infty), K_ \infty; \tau_ \infty \otimes \xi)$ for an irreducible admissible representation $\tau_ \infty$ of $G(F_ \infty)$.
\item $(\fkX,\chi)$ is the central character datum given as $\fkX:=(Z(\A_F^ \infty)\cap K)\times Z(F_ \infty)$ and $\chi:=\chi_{\xi}$ (extended from $Z(F_ \infty)$ to $\fkX$ trivially on $Z(\A_F^ \infty)\cap K$).
\end{itemize}

\begin{remark}\label{rem:K_infty}
We point out that $K_ \infty$ is a subgroup of index $|\pi_0(G(F_ \infty)/Z(F_ \infty))|$ in some $K'_ \infty$ which is a product of $Z(F_ \infty)$ and a maximal compact subgroup of $G(F_ \infty)$ and that $\ep(\tau_ \infty \otimes \xi)$ is defined in terms of $K_ \infty$, not $K'_ \infty$. When $\tau_ \infty$ belongs to the discrete series $L$-packet $\Pi^G_\xi$, the representation $\tau_ \infty$ decomposes into a direct sum of irreducible $(\Lie G(F_ \infty),K_ \infty)$-modules, the number of which is equal to $|\pi_0(G(F_ \infty)/Z(F_ \infty))|$, cf. paragraph below Lemma 3.2 in \cite{KottwitzInventiones}\footnote{The discussion there is correct in our setting, but it can be false for non-discrete series representations (which are allowed in that paper). For instance when $\xi$ and $\pi_ \infty$ are the trivial representation in the notation of \cite{KottwitzInventiones}, it is obvious that $\tau_ \infty$ cannot decompose further even if $|\pi_0(G(F_ \infty)/Z(F_ \infty))|>1$.}. So $\ep(\tau_ \infty\otimes \xi)=(-1)^{-n(n+1)/2}|\pi_0(G(F_ \infty)/Z(F_ \infty))|=(-1)^{-n(n+1)/2}2$ for $\tau_ \infty \in \Pi^G_\xi$.
\end{remark}

Let $f^{ \infty, p} \in \cH(G(\A^{ \infty p}_F)\sslash K^p)$. We would like to present a stabilized trace formula computing the action of $f^{ \infty,p}$ and $\Frob_{\fkp}$ on $\uH_c(\Sh_K, \cL_{\iota \xi})$.

Let $\cE_{\mathrm{ell}}(G)$ denote the set of representatives for isomorphism classes of $(G, K)$-un\-ra\-mi\-fied elliptic endoscopic triples of $G$. For each $(H,s,\eta_0) \in \cE_{\mathrm{ell}}(G)$ we make a fixed choice of an $L$-morphism $\eta \colon {}^LH \to {}^LG$ extending $\eta_0$ (which exists since the derived group of $\GSp_{2n}$ is simply connected). We recall the definition of $\iota(G,H) \in \Q$ and $h^H=h^{H, \infty,p}h^H_p h^H_ \infty \in \cH(H(\A_F))$ only for the principal endoscopic triple $(H,s,\eta)=(G^*,1,\mathrm{id})$, which is all we need. (For other endoscopic triples, the reader is referred to (7.3), second display on p.180 and second display on p.186 of \cite{KottwitzAnnArbor}; this is adapted to a little more general setup as ours in \cite{KSZ}.) We have $\iota(G,G^*)=1$ and $h^{G^*}=h^{G^*, \infty,p}h^{G^*}_p h^{G^*}_ \infty \in \cH(G(\A_F),\chi^{-1})$ given as follows.
\begin{itemize}
\item $h^{G^*, \infty,p} \in \cH(G^*(\A^{ \infty,p}_F))$ is an endoscopic transfer of the function $f^{ \infty,p}$ to the inner form $G^*$ ($h^{G^*, \infty,p}$ can be made invariant under $Z(\A_F^{ \infty,p})\cap K^p$ by averaging over it),
\item $h^{G^*}_p \in \cH^{\unr}(G^*(F_p))$ is the base change transfer of $\phi_j \in \Hunr(G(F_{p, rj}))$ down to $\Hunr(G^*(F_p))=\Hunr(G(F_p))$,
\item $h^{G^*}_ \infty \in \cH(G^*(F_ \infty),\chi^{-1})):=|\Pi^{G^*}_\xi|^{-1}\sum_{\tau_ \infty \in \Pi^{G^*}_\xi} f_{\tau_ \infty}$, i.e. the average of pseudo-coefficients for the discrete series $L$-packet $\Pi^{G^*}_\xi$. As shown by \cite[Lem. 3.2]{KottwitzInventiones}, $h^{G^*}_ \infty$ is $N_ \infty^{-1}$ times the Euler-Poincar\'e function for
$\xi$ (defined in \S\ref{sect:TraceFormula} but using $K_ \infty$ of this section):
\begin{equation}\label{eq:h_infty}
\Tr \pi_ \infty(h^{G^*}_ \infty)=N_ \infty^{-1} \ep(\pi_ \infty\otimes\xi),\quad \pi_ \infty \in \Irr(G^*(F_ \infty)).
\end{equation}
\end{itemize}
The main result of \cite{KSZ} is the following (which works for every Shimura variety of abelian type when $p\neq 2$), where the starting point is Kisin's proof of the Langlands-Rapoport conjecture for all Shimura varieties of abelian type \cite{KisinPoints}. When $F=\Q$ it was already shown by \cite{KottwitzAnnArbor,KottwitzPoints}.

\begin{theorem}\label{thm:KottwitzPointConjecture}
Let $f^ \infty \in \cH(G(\A^ \infty_F)\sslash K)$. Suppose that $p$ and $\fkp$ are as above such that
\begin{itemize}
\item $(G,K)$ is unramified at $p$,
\item $f^ \infty=f^{ \infty,p}f_p$ with $f^{ \infty, p} \in \cH(G(\A^{ \infty p}_F)\sslash K^p)$ and $f_p = \one_{K_p}$.
\end{itemize}
Then there exists a positive integer $j_0$ (depending on $\xi$, $f^ \infty$, $\fkp$, $p$) such that for all $j \geq j_0$ we have

\begin{equation}\label{eq:KottwitzConjectureA}
\iota^{-1}\Tr(\iota f^ \infty \Frob_{\fkp}^j,\, \uH_c(\Sh_K, \cL_{\iota \xi})) = \sum_{H \in \cE_{\mathrm{ell}}(G)} \iota(G, H) \ST_{\mathrm{ell},\chi}^H(h^H).
\end{equation}
\end{theorem}

\section{Galois representations in the cohomology}\label{sect:ConstructionOfaGaloisRepresentation}

Let $\pi$ be a cuspidal $\xi$-cohomological automorphic representation of $\GSp_{2n}(\A_F)$ satisfying condition (St), and fix the place $\vst$ in that condition. Define an inner form $G$ of $\GSp_{2n}$ as in \S\ref{sect:ShimuraDatum}; when $[F:\Q]$ is even, we take $\vst$ in that definition to be the fixed place $\vst$. Let $\pi^\natural$ be a transfer of $\pi$ to $G(\A_F)$ via Proposition \ref{prop:JL-for-GSp} (which applies thanks to Remark \ref{rem:JL}) so that at the unramified places $\pi^{\natural, \infty}$ and $\pi^{ \infty}$ are isomorphic, $\pi^\natural_{\vst}$ is an unramified twist of the Steinberg representation, and $\pi^\natural_ \infty$ is $\xi$-cohomological. The aim of this section is to compute a certain $\pi^{\natural, \infty}$-isotypical component of the cohomology of the Shimura variety $\Sh$ attached to $(G, X)$.

Let $A(\pi^\natural)$ be the set of (isomorphism classes of) cuspidal automorphic representations $\tau$ of $G(\A_F)$ such that
$\tau_{\vst}\simeq \pi^\natural_{\vst}\otimes \delta$ for an unramified character $\delta$ of the group $G(F_{\vst})$, $\tau^{ \infty,\vst}\simeq \pi^{\natural, \infty, \vst}$ and $\tau_ \infty$ is $\xi$-cohomological. Let $K \subset G(\A_F^ \infty)$ be a sufficiently small decomposed compact open subgroup such that $\pi^{\natural, \infty}$ has a non-zero $K$-invariant vector. Let $\Sbad$ be the set of prime numbers $p$ for which either $p=2$, the group $\Res_{F/\Q} G$ is ramified or $K_p = \prod_{v|p} K_v$ is not hyperspecial.

Let $\uH_\uc^i(\Sh, \cL_{\iota \xi})_{\tu{ss}}$ denote the semisimplification of $\uH_\uc^i(\Sh, \cL_{\iota \xi})$ as a $G(\A_F^ \infty)\times \Gamma$-module. Likewise $\uH_\uc^i(\Sh_{K,\li \Q}, \cL_{\iota \xi})_{\tu{ss}}$ means the semisimplification as an $\cH(G(\A^ \infty_F)\sslash K)\times \Gamma$-module. For each $\tau \in A(\pi^\natural)$ (or any irreducible admissible representation of $G(\A_F^ \infty)$), we define the $\tau^ \infty$-isotypic part of an admissible $G(\A_F^ \infty)$-module $R$ on a $\C$-vector space as
\begin{equation}\label{eq:tau-isotypic-part}
R[\tau^ \infty]:=\Hom_{G(\A^ \infty_F)}(\tau^{ \infty}, R),
\end{equation}
which is finite-dimensional. If the underlying space of $R$ is over $\lql$, we define $R[\tau^ \infty]$ using $\iota\tau^ \infty$ in \eqref{eq:tau-isotypic-part}.
If $R$ has an action of $\Gamma$ commuting with $G(\A_F^ \infty)$, then $R[\tau^ \infty]$ inherits the $\Gamma$-action.
As a primary example, $\uH_\uc^i(\Sh, \cL_{\iota \xi})_{\tu{ss}}[\tau^ \infty]$ is a finite-dimensional representation of $\Gamma$,
which is isomorphic to $\Hom_{\cH(G(\A^ \infty_F)\sslash K)}(\iota(\tau^{ \infty})^K, \uH_\uc^i(\Sh_{K,\li \Q}, \cL_{\iota \xi})_{\tu{ss}})$ for each $K$ such that $(\tau^ \infty)^K\neq 0$ since the isomorphism class of $(\tau^{ \infty})^K$ and that of $\tau$ determine each other. In particular the representation is finite-dimensional and unramified at almost all places.

Consider two representations $\tau, \tau' \in A(\pi^\natural)$ equivalent and write $\tau \sim \tau'$ if $\tau^ \infty \simeq \tau^{\prime, \infty}$. Define a virtual Galois representation
\begin{equation}\label{eq:Rho2Plus}
\rho_2^\coh = \rho_2^\coh(\pi^\natural):=(-1)^{n(n+1)/2} \sum_{\tau \in A(\pi^\natural)/{\sim}} \sum_{i=0}^{n(n+1)} (-1)^i [\uH_\uc^i(\Sh, \cL_{\iota \xi})_{\tu{ss}}[\tau^ \infty]],
\end{equation}
in the Grothendieck group of the category of continuous representations of $\Gamma$ on finite dimensional $\lql$-vector spaces
which are unramified at almost all places.
We compute in this section $\rho_2^\coh$ at almost all $F$-places not dividing a prime in $\Sbad$.

Define a rational number
\begin{equation}\label{eq:Defa0}
a(\pi^\natural) \bydef (-1)^{n(n+1)/2} N^{-1}_ \infty \sum_{\tau \in A(\pi^\natural)} m(\tau) \cdot \tu{ep}(\tau_ \infty \otimes \xi),
\end{equation}
where $m(\tau)$ is the multiplicity of $\tau$ in the discrete automorphic spectrum of $G$.

Recall the integer $w \in \Z$ determined by $\xi$ as in condition (cent) of the last section. In the following, the subscript ss always means the semisimplification as a $G(\A_F^ \infty)$-module, or as a $G(\A_F^ \infty)\times \Gamma$-module if there is a $\Gamma$-action. As in \cite[\S6]{FaltingsCohomology}, we have natural $G(\A_F^ \infty)$-equivariant maps of $\C$-vector spaces
\begin{equation}\label{eq:Hc-to-H2-to-H}
\uH^i_c(\Sh,\cL_{\xi}) \ra \uH^i_{(2)}(\Sh,\cL_{\xi}) \ra \uH^i(\Sh,\cL_{\xi}), 
\end{equation}
where $\uH^i_c(\Sh,\cL_{\xi})$ and $\uH^i(\Sh,\cL_{\xi})$ denote the topological cohomologies, which are isomorphic to the $\ell$-adic cohomologies $\uH^i_c(\Sh,\cL_{\iota\xi})$ and $\uH^i(\Sh,\cL_{\iota\xi})$ via $\iota$.

\begin{lemma}\label{lem:H2=Hc}
For each $\tau \in A(\pi^\natural)$ the following hold.
\begin{enumerate}
\item For every $i \in \Z_{\ge 0}$, the maps of \eqref{eq:Hc-to-H2-to-H} induce isomorphisms
 $$\uH^i_c(\Sh,\cL_{\xi})[\tau^ \infty] \stackrel{\sim}{\ra} \uH^i_{(2)}(\Sh,\cL_{\xi})[\tau^ \infty] \stackrel{\sim}{\ra} \uH^i(\Sh,\cL_{\xi})[\tau^ \infty].$$
 Moreover $\dim \uH^i_c(\Sh,\cL_{\xi})_{\tu{ss}}[\tau^ \infty]=\dim \uH^i_c(\Sh,\cL_{\xi})[\tau^ \infty]$. (Namely all subquotients of $\uH^i_c(\Sh,\cL_{\xi})$ isomorphic to $\tau^ \infty$ appear as subspaces.)
\item For every finite place $x\neq \vst$ of $F$, if $\tau$ is unramified at $x$ then $\tau_x|\simil|^{w/2}$ is tempered and unitary.
\end{enumerate}
\end{lemma}
\begin{proof}
(1) If $F\neq \Q$ the isomorphisms are clear by compactness of $\Sh_K$. In that case, $\uH^i(\Sh,\cL_{\xi})$ is semisimple as a $G(\A^ \infty_F)$-module since the same is true for the $L^2$-automorphic spectrum (which is entirely cuspidal since $G$ is anisotropic modulo center over $F$), via Matsushima's formula.

 Suppose that $F=\Q$ from now. 
By Poincar\'e duality, it suffices to check that \eqref{eq:Hc-to-H2-to-H} induces an isomorphism $\uH^i_{(2)}(\Sh,\cL_{\xi})[\tau^ \infty] \stackrel{\sim}{\ra} \uH^i(\Sh,\cL_{\xi})[\tau^ \infty]$. 
To set things up,
let $S\supset S_{\textup{bad}}$ be a finite set of $F$-places including all infinite places. Then $K^{S}$ is a product of hyperspecial subgroups, and $(\tau^S)^{K^S}\neq 0$. The category of $\cH(G(\A^S_F)\sslash K^S)$-modules is semisimple and irreducible objects are one-dimensional. For a $\cH(G(\A^S_F)\sslash K^S)$-module $R$, define the $\tau^S$-part of $R$ to be the maximal subspace $R_{\tau^S}$ such that $R_{\tau^S}$ is isomorphic to a self-direct sum of $\tau^S$. Then it suffices to prove that \eqref{eq:Hc-to-H2-to-H} induces
\begin{equation}\label{eq:tau-S-part}
\uH^i_{(2)}(\Sh_K,\cL_{\xi})_{\tau^S}\stackrel{\sim}{\ra} \uH^i(\Sh_K,\cL_{\xi})_{\tau^S},\quad i\ge 0.
\end{equation}
Indeed, the desired isomorphism of (1) will follow from the preceding formula by taking the functor $\Hom_{G(F_{S\backslash \cV_ \infty})}(\tau_{S\backslash \cV_ \infty},\cdot)$. The latter assertion of (1) is also implied by \eqref{eq:tau-S-part} as we now explain. 
Again by Poincar\'e duality, we can reduce to showing it for the ordinary cohomology rather than the compact support cohomology.
Then it suffices to verify that $\uH^i(\Sh_K,\cL_{\xi})_{\tau^S}$ is a semisimple $G(F_{S\backslash \cV_ \infty})$-module. This follows from \eqref{eq:tau-S-part}
since $\uH^i_{(2)}(\Sh_K,\cL_{\xi})$ is a semisimple $G(\A^ \infty_F)$-module by comparison with the $L^2$-discrete automorphic spectrum via Borel--Casselman's theorem.

It remains to obtain \eqref{eq:tau-S-part}. We apply Franke's spectral sequence in the notation of \cite[Cor., p.151]{Waldspurger-Franke} (cf.~\cite[Thm. 19]{Franke}). 
 The spectral sequence, which is $G(\A_F^ \infty)$-equivariant, stays valid when restricted to the $\tau^S$-part.
 We claim that the $\tau^S$-part is zero in any parabolic induction of an automorphic representation on a proper Levi subgroup of $\GSp_{2n}(\A)$. It suffices to check the analogous claim with $\tau^\flat$ (as in Lemma \ref{lem:LabesseSchwermer}) and $\Sp_{2n}$ in place of $\tau$ and $\GSp_{2n}$. The latter claim follows from Arthur's main result \cite[Thm. 1.5.2]{ArthurBook}, Corollary \ref{cor:TemperedParameter}, and the strong multiplicity one theorem for general linear groups.
 As a consequence of the claim, the direct summand of $E_1^{p,q}$ over $(w,P)$ in \cite[p.151]{Waldspurger-Franke} has vanishing $\tau^S$-part unless $P=G$ and $w=1$. 
 For $w=1$, we see from the first line on p.151 of \emph{loc.~cit.} that there is a nonzero contribution to $E_1^{p,q}$ for a unique $p$ (independent of $q$).
 Therefore the $\tau^S$-part of the spectral sequence degenerates at $E_1$ and yields an isomorphism, in the notation there (writing $i$ for $p+q$),
$$ H^i(\mathfrak{m}^1,K_{\R};\mathcal{A}_2(G,\mu_\lambda)\otimes_{\C} E^G_\lambda)_{\tau^S} \stackrel{\sim}{\ra} H^i(\mathfrak{m}^1,K_{\R};\mathcal{A}(G,\mu_\lambda)\otimes_{\C} E^G_\lambda)_{\tau^S}, \quad i\ge 0.$$
This map is the natural one induced by $\mathcal{A}_2(G,\mu_\lambda)\subset\mathcal{A}(G,\mu_\lambda)$, and translates into \eqref{eq:tau-S-part} in our notation via Borel's theorem \cite[p.143]{Waldspurger-Franke} on the right hand side. The proof of (1) is complete.

(2) Let $\omega:F^\times\backslash \A_F^\times\to \C^\times$ denote the common central character of $\tau$, $\pi^\natural$, and $\pi$. (The central characters are the same as they are equal at almost all places.) Since $F$ is totally real, $\omega=\omega_0 |\cdot|^a$ for a finite character $\omega_0$ and a suitable $a \in \C$. Since $\pi^\natural_ \infty$ is $\xi$-cohomological, we must have $a=-w$. Then $\pi^\natural_x |\simil|^{w/2}$ has unitary central character and is essentially tempered by Lemma \ref{lem:ForceTempered}. We are done since $\pi^\natural_x\simeq \tau_x$.
\end{proof}

\begin{proposition}\label{prop:PointCounting}
For almost all finite $F$-pla\-ces $v$ not dividing a prime number in $\Sbad$ and all sufficiently large integers $j$, we have
$$
\Tr\rho_2^\coh(\Frob_v^j) =q_v^{jn(n+1)/4}\Tr (i_{a(\pi^\natural)}(\spin^\vee(\iota\phi_{\pi_v^\natural})))(\Frob_v^j).
$$
 Moreover the virtual representation $\rho_2^\coh$ is a true representation.
\end{proposition}
\begin{proof}
We imitate arguments from \cite{KottwitzInventiones}. We consider a function $f$ on $G(\A_F)$ of the form $f = f_ \infty \otimes f_\vst \otimes f^{ \infty,\vst}$, where the components are chosen in the following way:
\begin{itemize}
\item We let $f_ \infty$ be $N_ \infty^{-1}$ times an Euler-Poincar\'e function for $\xi$ (and $K_ \infty$) on $G(F_ \infty)$ so that
\begin{equation}\label{eq:AuxFunctioninfty}
\Tr \tau_ \infty(f_ \infty)=N_ \infty^{-1} \ep(\tau_ \infty\otimes\xi)=N_ \infty^{-1}\sum_{i = 0}^ \infty (-1)^i \dim \uH^i(\ig, K_ \infty; \tau_ \infty \otimes \xi).\end{equation}
\item We pick a Hecke operator $f^{ \infty,\vst} =\prod_{v\notin \cV_ \infty\cup \{\vst\}} \in \cH(G(\A^{ \infty,\vst}_F)\sslash K^{\vst})$ 
such that for all automorphic representations $\tau$ of $G(\A_F)$ with $\tau^{ \infty,K} \neq 0$ and $\Tr \tau_ \infty(f_ \infty) \neq 0$ we have
\begin{equation}\label{eq:AuxFunction}
\Tr\tau^{ \infty,\vst}(f^{ \infty,\vst}) = \begin{cases}
1, & \tu{if $\tau^{ \infty,\vst} \simeq \pi^{\natural, \infty,\vst}$} \cr
0, & \tu{otherwise.}
\end{cases}
\end{equation}
This is possible since there are only finitely many such $\tau$ (one of which is $\pi^\natural$).
\item We let $f_{\vst}$ be a Lefschetz function from Equation \eqref{eq:fLef-truncated} of Appendix A. 
\end{itemize}
There exists a finite set of primes $\Sigma$ containing places over $2$ such that $f^ \infty$ decomposes $f^ \infty = f_\Sigma \otimes f^{ \infty\Sigma}$ with
\begin{itemize}
 \item $f_\Sigma \in \cH(G(F_\Sigma)\sslash K_\Sigma)$,
 \item $f^{ \infty, \Sigma} = \one_{K^\Sigma} \in \cH(G(\A^{ \infty\Sigma}_F))$,
 \item $K^\Sigma$ is a product of hyperspecial subgroups, and
 \item if $v_1,v_2$ are $F$-places above the same rational prime, then $v_1 \in \Sigma$ if and only if $v_2 \in \Sigma$.
\end{itemize}

In the rest of the proof we consider only finite places $v$ such that $v \notin \Sigma$, $v \nmid \ell$.
Fix such a place and call it $\fkp$ as in \S\ref{sect:ShimuraDatum}. We also fix an isomorphism $\iota_p:\C\simeq \li \Q_p$ such that $\iota_p v_ \infty:F\hra \li \Q_p$ induces $\fkp$.
Write $p:=\fkp|_\Q$. Then $K_p=\prod_{v|p} K_v$ is a hyperspecial subgroup of $G(F_p)$, and $f_p=\prod_{v|p} f_v=\prod_{v|p}\one_{K_v}$.
Thus $(G,K)$ is unramified at $p$ in the sense of \S\ref{sect:ShimuraDatum} so that the Langlands--Kottwitz method applies.

The stabilized Langlands-Kottwitz formula (Theorem \ref{thm:KottwitzPointConjecture}) simplifies as
\begin{equation}\label{eq:Stabilizing}
\iota^{-1}\Tr(\iota f^ \infty \Frob_\fkp^j,\, \uH_c(\Sh_K, \cL_{\iota\xi})) = \ST_{\el,\chi}^{G^*}(h^{G^*}),
\end{equation}
essentially for the same reason as in the proof of Lemma \ref{lem:stabilization}. Indeed $f_\vst$ is a stabilizing function (Lemma \ref{lem:LefschetzIsStabilizing}), thus unless $H=G^*$, the stable orbital integrals of $h^H_{\vst}$ vanish as they equal $\kappa$-orbital integrals of $f_\vst$ with $\kappa\neq1$, which are zero (Definition \ref{def:cuspidal-stabilizing}), up to a nonzero constant via the Langlands--Shelstad transfer. Note that $h_{\vst}^{G^*}$ can be chosen to be a Lefschetz function thanks to Lemma \ref{lem:LefschetzMatching-finite}.

To extract spectral information, we will compare \eqref{eq:Stabilizing} with the stabilized trace formula of \S\ref{sect:TraceFormula}.
We choose the following test functions $\tilde f_p$ and $\tilde f_ \infty$ on $G(F_p)$ and $G(F_ \infty)$.
\begin{itemize}
 \item $\tilde f_p:=h^{G^*}_p$ via the identification $G(F_p)=G^*(F_p)$,
 \item $\tilde f_ \infty$ is $N_ \infty^{-1}$ times an Euler-Poincar\'e function associated to $\xi$.
\end{itemize}
Then $h^{G^*}_p$ and $h^{G^*}_ \infty$ are transfers of $\tilde f_p$ and $\tilde f_ \infty$ from $G$ to $G^*$, respectively.
This is trivial at $p$ and Lemma \ref{lem:LefschetzMatching} at $ \infty$.
By construction, $h^{G^*, \infty,p}$ is a transfer of $f^{ \infty,p}$.
Therefore Lemmas \ref{lem:simple-TF} and \ref{lem:stabilization} tell us that
\begin{equation}\label{eq:Tcusp=STell}
T_{\cusp,\chi}^{G}(f^{ \infty,p} \tilde f_p \tilde f_ \infty) = \ST_{\el,\chi}^{G^*}(f^{G^*}).
\end{equation}
 By \eqref{eq:Stabilizing} and \eqref{eq:Tcusp=STell}, we have
\begin{equation}\label{eq:SpectralKottwitzFormula}
\iota^{-1} \Tr(\iota f^ \infty\Frob_\fkp^j,\, \uH_c(\Sh_{K}, \cL_{\iota\xi})) =
\sum_{\tau \in \cA_{\cusp,\chi}(G)} m(\tau)\Tr \tau^{ \infty,p}(f^{ \infty,p})\Tr \tau_p(\tilde f_p) \Tr \tau_ \infty(\tilde f_ \infty).
\end{equation}
The summand on the right hand side vanishes unless $\Tr\tau_ \infty(f_ \infty) \neq 0$, $\tau_p$ is unramified, and $\Tr\tau^{ \infty,p}(f^{ \infty,p}) \neq 0$.
The latter two imply nonvanishing of $\Tr\tau^{ \infty}(f^{ \infty})$, thus $\Tr\tau^{ \infty,\vst}(f^{ \infty,\vst})\neq 0$
and $\Tr \tau_{\vst}(f_{\vst})\neq 0$. By the choice of function $f^{ \infty, \vst}$ (see \eqref{eq:AuxFunction}) and $f_\vst$,
it follows that $\tau$ contributes nontrivially in \eqref{eq:SpectralKottwitzFormula} only if $\tau \in A(\pi^\natural)$. In this case, we have $\tau_p\simeq \pi^\natural_p$ in particular. Recall $\tilde f_p=h^{G^*}_p$. Applying \eqref{eq:AuxFunctioninfty} and \eqref{eq:AuxFunction}, we identify the right hand side of \eqref{eq:SpectralKottwitzFormula} with
\begin{equation}\label{eq:SpectralKottwitzFormula2}
N_ \infty^{-1}\sum_{\tau \in A(\pi^\natural)} m(\tau) \ep(\tau_ \infty\otimes\xi) \Tr \pi_p^\natural(h^{G^*}_p) =(-1)^{n(n+1)/2}a(\pi^\natural) \Tr \pi_p^\natural(h^{G^*}_p).
\end{equation}
By the choice of $f^ \infty$ the left hand side of \eqref{eq:SpectralKottwitzFormula} is equal to the trace of $\Frob_v^j$ on $(-1)^{n(n+1)/2} \rho_2^\coh$.
 To sum up,
\begin{equation}\label{eq:LKPCF4}
\Tr (\Frob_\fkp^j, \rho_2^\coh) = a(\pi^\natural) \iota \Tr \pi_p^\natural(h_p^{G^*}).
\end{equation}
To analyze the right hand side, consider the cocharacter $\mu=(\mu_y)_{y \in \cV_ \infty}$ from \S\ref{sect:ShimuraDatum}.
The irreducible representation of $\hat{\Res_{F/\Q}G}\simeq \prod_{y \in \cV_ \infty} \GSpin_{2n+1}(\C)$ of highest weight $\mu$
is $\spin$ on the $v_ \infty$-component and trivial on the other components (since $\mu_y=1$ if $y\neq v_ \infty$). This extends uniquely to a representation $r_\mu$ of the $L$-group of $(\Res_{F/\Q}G)_{F_\fkp}$ as in \cite[Lem. 2.1.2]{KottwitzTwistedOrbital} since the conjugacy class of $\mu$ is defined over $F$ (thus also $F_\fkp$). The Satake parameter of $\pi_p\simeq \prod_{v|p} \pi_v$ (with respect to $G(F_p)\simeq \prod_{v|p} G(F_v)$) belongs to the $\Frob_\fkp$-coset of $\hat{(\Res_{F/\Q}G)_{\Q_p}}$, identified with $\hat{\Res_{F/\Q}G}$ via $\iota_p$, in the $L$-group. Then the $v_ \infty$-component of the Satake parameter of $\pi_p$ comes from the Satake parameter of $\pi_\fkp$ since $F\hra \C\stackrel{\iota_p}{\simeq} \li \Q_p$ induces $\fkp$.
With this observation, we apply \cite[(2.2.1)]{KottwitzTwistedOrbital} (his $\pi_\varphi$, $G$, $n$, $F$, $\phi^G_p$ correspond to our $\pi^\natural_p$, $(\Res_{F/\Q}G)_{\Q_p}$, $rj$, $\Q_{p^{rj}}$, $h^{G^*}_p$) to obtain
\begin{equation}\label{eq:Kottwitz-f_j}
\Tr \pi_p^\natural(h_p^{G^*}) = q_\fkp^{jn(n+1)/4} \Tr (\spin^\vee(\phi_{\pi_\fkp^\natural}))(\Frob_\fkp^j). 
\end{equation}
(See p.656 and the second last paragraph on p.662 of \cite{KottwitzInventiones} for a similar computation.)
This finishes the proof of the first assertion.

It remains to show that $\rho_2^\coh$ is not just a virtual representation but a true representation. To this end we will show that $\rho_2^\coh$ is concentrated in the middle degree $n(n+1)/2$. Let $i \ge 0$. There are natural maps from $\uH_c^i(\Sh,\cL_{\iota\xi})$ to each of $\IH^i(\Sh,\cL_{\iota\xi})$ and $\uH^i_{(2)}(\Sh,\cL_{\iota\xi})$. Compatibility with Hecke correspondences implies that both maps are $G(\A_F^ \infty)$-equivariant; the first map is moreover equivariant for the action of $\Gamma$ (which commutes with the $G(\A_F^ \infty)$-action). By Zucker's conjecture, which is proved in the articles \cite{Looijenga,LooijengaRapoport,SaperSternI}, the $L^2$-cohomology $\uH^i_{(2)}(\Sh, \cL_{\xi})$ is naturally isomorphic to the intersection cohomology $\IH^i(\Sh, \cL_{\iota\xi})$ via $\iota$ so that we have a $G(\A_F^ \infty)$-equivariant commutative diagram\footnote{The $G(\A_F^ \infty)$-equivariance is clear for the horizontal map, which is induced by the map from the extension-by-zero of $\cL_{\iota\xi}$ to the Baily-Borel compactification to the intermediate extension thereof. For the vertical map it can be checked as follows (we thank Sophie Morel for the explanation). From the proof of Zucker's conjecture we know that $\IH^i(\Sh,\cL_{\iota\xi})$ and $\uH^i_{(2)}(\Sh,\cL_{\iota\xi})$ are represented by two complexes of sheaves on the Baily-Borel compactification which are quasi-isomorphic and become canonically isomorphic to $\cL_{\iota\xi}$ when restricted to the original Shimura variety. The Hecke correspondences extend to the two complexes (so as to define Hecke actions on the intersection and $L^2$-cohomologies) and coincide if restricted to the original Shimura variety. Now the point is that the Hecke correspondences extend uniquely for the intersection complex as follows from \cite[Lem. 5.1.3]{Morel-GSp}, so the same is true for the $L^2$-complex via Zucker's conjecture. Therefore the Hecke actions on the two cohomologies are identified via the isomorphism of Zucker's conjecture.}
$$
\xymatrix{ \uH_c^i(\Sh,\cL_{\iota\xi}) \ar[rr] \ar[rrd] && \uIH^i(\Sh, \cL_{\iota\xi}) \ar[d]^\sim \\ && \iota\uH^i_{(2)}(\Sh, \cL_{\xi})}.
$$
The diagram together with Lemma \ref{lem:H2=Hc} yields a $\Gamma$-equivariant isomorphism
$$
\uH_c^i(\Sh,\cL_{\iota\xi})[\tau^ \infty]\simeq \IH^i(\Sh, \cL_{\iota\xi})[\tau^ \infty],
$$
the semisimplification of which is isomorphic to $\uH_c^i(\Sh,\cL_{\iota\xi})_{\tu{ss}}[\tau^ \infty]$.

In view of condition (cent), the intersection complex defined by $\xi$ is pure of weight $-w$. Hence the action of $\Frob_\fkp$ on $\IH^i(\Sh, \cL_{\iota\xi})[\tau^ \infty]$, thus also on $\uH_c^i(\Sh,\cL_{\iota\xi})_{\tu{ss}}[\tau^ \infty]$, is pure of weight $-w+i$ for every $\fkp$ as above, cf. the proof of \cite[Rem. 7.2.5]{MorelBook}. In particular there is no cancellation between different $i$ in \eqref{eq:Rho2Plus}.
 On the other hand, part (2) of Lemma \ref{lem:H2=Hc} (with $x=\fkp$) implies that $\tau_\fkp|\simil|^{w/2}=\pi^\natural_\fkp|\simil|^{w/2}$ is tempered and unitary. Then the first part of the proposition implies that $\rho_2^\coh(\Frob_\fkp)_{\tu{ss}}$ has eigenvalues whose absolute values equal the $-w+\frac{n(n+1)}{2}$-th power of $(\# \cO_F/\fkp)^{1/2}$. We conclude that $\IH^i(\Sh, \cL_{\iota\xi})_{\tu{ss}}[\tau^ \infty]=0$ unless $i=n(n+1)/2$. The proof of the proposition is complete.
\end{proof}

\begin{remark}
The last paragraph in the above proof simplifies significantly if $F\neq \Q$, in which case $\Sh_K$ is proper thus the compact support cohomology coincides with the intersection cohomology. See the first paragraph in the proof of Lemma \ref{lem:H2=Hc} for another simplification. When $F=\Q$, as an alternative to studying various cohomology spaces, one can also argue by reducing to the $F\neq \Q$ case as follows. Consider many real quadratic extensions $E/\Q$, find a base-change automorphic representation $\pi_E$ of $\GSp_{2n}(\A_{E})$ by Proposition \ref{prop:BC-for-GSp}, and apply the so-called patching lemma \cite{Sorensen} to patch $\spin(\rho^C_{\pi_E})$ to obtain a Galois representation $\rho_2':\Gamma_{\Q}\ra \GL_{2^n}(\lql)$. The argument in the next section can be adapted with $\rho'_2$ in place of $\rho_2$ to construct $\rho_\pi^C: \Gamma_{\Q} \ra \GSpin_{2n+1}(\lql)$.
\end{remark}

\begin{corollary}\label{cor:Pi-is-in-Disc-series-Packet-2}
If $\tau \in A(\pi^\natural)$ then $\tau_ \infty$ belongs to the discrete series $L$-packet $\Pi^{G(F_ \infty)}_\xi$.
\end{corollary}
\begin{proof}
Since $\tau_ \infty$ is $\xi$-cohomological and unitary, $\tau_ \infty$ appears in the Vogan--Zuckerman classification of \cite{VoganZuckerman}. The preceding proof shows that $\uH_{(2)}^i(\Sh_K,\cL_{\xi})[\tau^ \infty]$ is nonzero only for $i=n(n+1)/2$. Hence $\uH^i(\Lie G(F_ \infty), K'_ \infty; \tau_ \infty^K \otimes \xi)$ does not vanish exactly for $i=n(n+1)/2$. This implies that $\tau_ \infty$ must be a discrete series representation through \cite[Thm.~5.5]{VoganZuckerman} (namely we must have $\mathfrak{l}\subset \mathfrak{k}$ in the notation there). Thus $\tau_ \infty \in \Pi^{G(F_ \infty)}_\xi$.
\end{proof}

\begin{corollary}\label{cor:A-is-integer}
If $\tau \in A(\pi^\natural)$ then $\tau^ \infty \tau'_ \infty \in A(\pi^\natural)$ for all $\tau'_ \infty \in \Pi^{G(F_ \infty)}_\xi$ and $m(\tau)=m(\tau^ \infty \tau'_ \infty)$. Moreover, $a(\pi^\natural)$ is a positive integer.
\end{corollary}

\begin{remark}
 A similar argument appears in \cite[Lem. 4.2]{KottwitzInventiones}.
\end{remark}

\begin{proof}
 We work with the trace formula with fixed central character, with respect to the datum $(\fkX, \chi)$, where $\fkX = Z(\A_F)$ and $\chi$ is the central character of $\pi$. (We diverge from the convention of \S\ref{sect:ShimuraDatum} only in this proof.) Let $G(F_{\vst})^1$ be as in the paragraph above \eqref{eq:fLef-truncated}, and $\eta:G(F_{\vst})/G(F_{\vst})^1\ra \C^\times$ the unique character such that $\pi_{\vst}$ is the $\eta$-twist of the Steinberg representation; in particular the restriction of $\eta$ to $Z(F_{\vst})$ equals $\chi$. Consider test functions $f=f^{ \infty,\vst} f_{\vst,\eta} f_{\tau_ \infty}$ and $f'=f^{ \infty,\vst} f_{\vst,\eta} f_{\tau'_ \infty}$, where
\begin{itemize}
 \item $f^{ \infty,\vst} \in \cH(G(\A_F^{ \infty,\vst}),(\chi^{ \infty,\vst})^{-1})$,
 \item $f_{\vst,\eta} \in \cH(G(F_{\vst}, \chi_{\vst}^{-1}))$ is the function $f^G_{\Lef,\eta}$ at $\vst$ as in Corollary \ref{cor:Lef-variant2}, and
 \item $f_{\tau_ \infty},f_{\tau'_ \infty}$ are pseudo-coefficients for $\tau_ \infty,\tau'_ \infty$, respectively.
\end{itemize}
By \cite[Lem.~3.1]{KottwitzInventiones}, $f_{\tau_ \infty}$ and $f_{\tau'_ \infty}$ have the same stable orbital integrals. Moreover $f_{\vst,\eta}$ is stabilizing by Corollary~\ref{cor:Lef-variant2}. The simple trace formula and its stabilization (Lemmas~\ref{lem:simple-TF} and~\ref{lem:stabilization}) imply that, by arguing similarly as in the proof of Proposition \ref{prop:JL-for-GSp},
\begin{equation}\label{eq:A-is-integer}
\sum_{\sigma} m(\sigma)\Tr (f_{\vst,\eta}|\sigma_{\vst})\Tr (f_{\tau_ \infty}|\sigma_ \infty) = \sum_{\sigma}m(\sigma) \Tr (f_{\vst,\eta}|\sigma_{\vst})\Tr (f_{\tau'_ \infty}|\sigma_ \infty) ,
\end{equation}
where both sums run over the set of $\sigma \in \cA_{\chi}(G)$ such that
\begin{itemize}
\item $\sigma^{ \infty,\vst}\simeq \tau^{ \infty,\vst}$,
\item $\sigma_{\vst}\simeq \tau _{\vst}$, by Corollary \ref{cor:Lef-variant2} (since $\sigma_{\vst}$ cannot be one-dimensional, as in the proof of Proposition \ref{prop:JL-for-GSp}),
 and
\item $\sigma_ \infty \in \Pi^{G(F_ \infty)}_\xi$. (Since $\Tr (f_{\tau_ \infty}|\sigma_ \infty) \neq 0$, the infinitesimal character of $\sigma_ \infty$ is the same as that of $\tau_ \infty$. By \cite[Thm. 1.8]{Salamanca-Riba}, we have that $\sigma_ \infty$ is $\xi$-cohomological. Thus $\sigma \in A(\pi^\natural)$, to which Corollary \ref{cor:Pi-is-in-Disc-series-Packet-2} applies.)
\end{itemize}
By the last condition, $\Tr (f_{\tau_ \infty}|\sigma_ \infty)$ equals $1$ if $\sigma_ \infty\simeq \tau_ \infty$ and $0$ otherwise; the obvious analogue also holds with $\tau'_ \infty$ in place of $\tau_ \infty$. Therefore, we have the equality
$$
m(\tau)=m(\tau^ \infty \tau'_ \infty).
$$

It remains to prove the second assertion. With $A(\pi^\natural)/\!\!\sim$ as in \eqref{eq:Rho2Plus}, we apply Corollary \ref{cor:Pi-is-in-Disc-series-Packet-2} to rewrite $a(\pi^\natural)$ defined in \eqref{eq:Defa0} as 
$$
a(\pi^\natural) = (-1)^{n(n+1)/2} N^{-1}_ \infty \sum_{\tau \in A(\pi^\natural)/\sim}
\sum_{\tau'_ \infty \in \Pi^{G(F_ \infty)}_\xi} m(\tau^ \infty \tau'_ \infty) \cdot \tu{ep}(\tau'_ \infty \otimes \xi).
$$
The inner sum is equal to 
$$
(-1)^{n(n+1)/2} \cdot 2 \cdot | \Pi^{G(F_ \infty)}_\xi|\cdot m(\tau)
$$
by what we just proved as well as Remark \ref{rem:K_infty}. Since $N_ \infty=2^n$ and $| \Pi^{G(F_ \infty)}_\xi|=2^{n-1}$, we conclude that $a(\pi^\natural)=\sum_{\tau \in A(\pi^\natural)/\sim} m(\tau) \in\Z_{>0}$. (Note $a(\pi^\natural)\ge m(\pi^\natural)>0$.)
\end{proof}

For $a \in \Z_{\ge 1}$ and $m \in \Z_{\ge1}$ let $i_a \colon \GL_m\ra \GL_{am}$ denote the block diagonal embedding.

\begin{corollary}\label{cor:lgc-unramified}
Let $\rho_2^\coh$ be as above. For almost all finite $F$-pla\-ces $v$ where $\pi^\natural$ is unramified and, $\rho_2^\coh(\Frob_v)_{\ssimple} \sim q_v^{n(n+1)/4}\cdot i_{a(\pi^\natural)} (\spin^\vee(\iota\phi_{\pi_v^\natural}))(\Frob_v)$ in $\GL_{a(\pi^\natural)2^n}$. 
\end{corollary}

\begin{proof} Write $\gamma_1 = \rho_2^\coh(\Frob_v)_{\ssimple}$ and $\gamma_2 =q_v^{n(n+1)/4}\cdot i_{a(\pi^\natural)} (\spin^\vee(\iota\phi_{\pi_v^\natural}))(\Frob_v)$. By Proposition \ref{prop:PointCounting} we get $\Tr (\gamma_1^j) = \Tr (\gamma_2^j)$ for $j$ sufficiently large. Consequently $\gamma_1$ and $\gamma_2$ are $\GL_{a(\pi^\natural)2^n}(\lql)$-conjugate.
\end{proof}

\section{Construction of the $\GSpin$-valued representation}\label{sect:Construction}

Let $\pi$ be a cuspidal $\xi$-cohomological automorphic representation of $\GSp_{2n}(\A_F)$ satisfying (St). Denote by $\omega_\pi$ its central character. We showed in Corollary \ref{cor:lgc-unramified} that there exists a $a(\pi^\natural) 2^n$-di\-men\-sio\-nal Galois representation $\rho_2^\coh \colon \Gamma \to \GL_{a(\pi^\natural)2^n}(\lql)$. In this section we show that the representation $\rho_2^\coh$ factors through the composition
$$
\GSpin_{2n+1}(\lql) \overset{\spin^\vee}\to \GL_{2^n}(\lql) \to \GL_{a(\pi^\natural) 2^n}(\lql), 
$$
to a representation $\rho_\pi^C \colon \Gamma \to \GSpin_{2n+1}(\lql)$.

Denote by $S_{\mathrm{bad}}$ the finite set of rational primes $p$ such that either $p=2$ or $p$ is ramified in $F$ or $\pi_v$ is ramified at a place $v$ of $F$ above $p$. (As we commented in introduction, it shouldn't be necessary to include $p=2$ in view of \cite{KimMadapusiPera}.) In the theorem below the superscript $C$ designates the $C$-normalization in the sense of \cite{BuzzardGee}. Statement (ii)' of the theorem will be upgraded in the next section to include \emph{all} $v$ which are not above $S_{\bad}\cup \{\ell\}$.

\begin{theorem}\label{thm:RhoPi0}
Let $\pi$ be a cuspidal $\xi$-cohomological automorphic representation of $\GSp_{2n}(\A_F)$ satisfying the conditions (St) and (L-coh). Let $\ell$ be a prime number and $\iota \colon \C \isomto \lql$ a field isomorphism. Then exists a representation
$$
\rho^C_{\pi} = \rho^C_{\pi, \iota} \colon \Gamma \to \GSpin_{2n+1}(\lql),
$$
unique up to conjugation by $\GSpin_{2n+1}(\lql)$, satisfying (iii.b), (iv) and (v) of Theorem \ref{thm:A} with $\rho^C_\pi$ in place of $\rho_\pi$, as well as the following (which are slightly modified from (i), (ii), (iii.a) of Theorem \ref{thm:A} due to a different normalization).
\begin{itemize}
\item[(\textit{i}')] For each automorphic $\Sp_{2n}(\A_F)$-subrepresentation $\pi^\flat$ of $\pi$, its associated Galois representation $\rho_{\pi^\flat}$ is isomorphic to $\rho^C_\pi$ composed with the projection $\GSpin_{2n+1}(\lql)\ra \SO_{2n+1}(\lql)$. The composition
$$
[\mathcal{N} \circ \rho^C_{\pi}] \colon \Gamma \to \GSpin_{2n+1}(\lql) \to \GL_1(\lql)
$$
corresponds to $\omega_\pi|\cdot|^{n(n+1)/2}$  by global class field theory.
\item[(\textit{ii}')] There exists a finite set of prime numbers $S$ with $S \supset \Sbad$ such that for all finite $F$-places $v$ which are not above $S\cup\{ \ell\}$, $\rho^C_{\pi, v}$ is unramified and the element $\rho^C_{\pi}(\Frob_v)_{\ssimple}$ is conjugate to the Satake parameter of $\iota\pi_v |\simil|^{n(n+1)/4}$  in $\GSpin_{2n+1}(\lql)$.
\item[(\textit{iii.a}')] $\mu_{\tu{HT}}(\rho^C_{\pi, v}, \iota y) = \iota \mu_{\tu{Hodge}}(\xi_y)-\frac{n(n+1)}{4}\simil$. (See Definitions \ref{def:Hodge-Tate} and \ref{def:Hodge}.)\footnote{Even though $\mu_{\tu{HT}}(\rho^C_{\pi, v}, \iota y)$ and $\iota \mu_{\tu{Hodge}}(\xi_y)$ are only \emph{conjugacy classes} of (possibly half-integral) cocharacters, the equality makes sense since $\frac{n(n+1)}{4}\simil$ is a central (possibly half-integral) cocharacter.}
\end{itemize}
\end{theorem}
\begin{proof}
We have the automorphic representation $\pi$ of $\GSp_{2n}(\A_F)$. Consider
\begin{itemize}
\item $\pi^\flat \subset \pi$ a cuspidal automorphic $\Sp_{2n}(\A_F)$-subrepresentation from Lemma \ref{lem:LabesseSchwermer};
\item $\pi^\natural$ a transfer of $\pi$ to the group $G(\A_F)$ from Proposition \ref{prop:JL-for-GSp};
\item $\rho_{\pi^\flat} \colon \Gamma \to \SO_{2n+1}(\lql)$ the Galois representation from Theorem \ref{thm:ExistGaloisRep};
\item $\wt{\rho_{\pi^\flat}}$ a lift of $\rho_{\pi^\flat}$ to the group $\GSpin_{2n+1}(\lql)$.
To prove the existence of this lift, consider for $s \in \Z_{\geq 1}$ the group $\GSpin^{(s)}_{2n+1}(\lql)$ of $g \in \GSpin_{2n+1}(\lql)$ such that $\cN(g)^s =1$. Using the exact sequence $\mu_{2s} \injects \GSpin^{(s)}_{2n+1} \surjects \SO_{2n+1}$ and the vanishing of $\textup{H}^2(\Gamma, \Q/\Z)$ we see that such a lift indeed exists for $s$ sufficiently large and divisible.
\item $\rho_1 \colon \Gamma \to \GSpin_{2n+1}(\lql) \overset {i_{a(\pi^\natural)} \spin^\vee} \to \GL_{a(\pi^\natural)2^n}(\lql)$ the composition of $\wt{\rho_{\pi^\flat}}$ with the $a(\pi^\natural)$-th power of the $\spin$ representation.
\item $\rho_2 = \rho_2^\coh \colon \Gamma \to \GL_{a(\pi^\natural) 2^n}(\lql)$ a semi-simple representative of the virtual representation introduced in \eqref{eq:Rho2Plus} (see also Proposition \ref{prop:PointCounting});
\end{itemize}
(see also Figure 1 in the introduction). We have
\begin{itemize}
\item For all finite $F$-places $v$ away from $S$,
$$
\overline {\rho_1(\Frob_v)_{\ssimple}} \sim \overline{\rho_2(\Frob_v)_{\ssimple}} \in \PGL_{a(\pi^\natural)2^n}(\lql)
$$
by Proposition \ref{prop:PointCounting};
\item The representation $\li {\rho_1} \colon \Gamma \to \PGL_{a(\pi^\natural) 2^n}(\lql)$ has connected image. This is true because $\li {\rho_1}$ factors over $\SO_{2n+1}(\lql)$, has a regular unipotent in its image, and hence has connected image by Proposition \ref{prop:DetermineZariskiImage}.
\end{itemize}
By Proposition \ref{prop:WeaklyAcceptable}(1) and Example \ref{ex:PGL-example} there exists a $g \in \GL_{a(\pi^\natural) 2^n}(\lql)$ and a character $\chi \colon \Gamma \to \lql^\times$, such that
$$
\rho_2 = \chi g \rho_1 g\inv \colon \Gamma \to \GL_{a(\pi^\natural)2^n}(\lql).
$$
Without loss of generality, we replace $\rho_2$ with $g^{-1}\rho_2 g$ to assume that $\rho_2=\chi \rho_1$.
By construction, the representation $\rho_1$ has image inside $\GSpin_{2n+1}(\lql) \subset \GL_{a(\pi^\natural)2^n}(\lql)$, and consequently the representation $\rho_2^\coh$ has image in $\GSpin_{2n+1}(\lql)$ as well. Thus $\rho_2^\coh$ induces a representation
$$
\rho^C_{\pi} \colon \Gamma \to \GSpin_{2n+1}(\lql)
$$
such that $\rho_2 = i_{a(\pi^\natural)}\spin^\vee \circ \rho^C_{\pi}$. We now compute $\rho^C_\pi$ at Frobenius conjugacy classes. The uniqueness of $\rho^C_\pi$ will then follow from Proposition \ref{prop:SpinConjugacyA} and the fact that since $q\rho^C_\pi$ is conjugate to $\rho_{\pi^\flat}$ it has a regular unipotent element of $\GSpin_{2n+1}$ in its image.

We have the semisimple elements $\rho^C_{\pi}(\Frob_v)_{\ssimple}$ and $\iota\phi_{\pi_v}(\Frob_v)$ in $\GSpin_{2n+1}(\lql)$. At this point we know that 
\begin{itemize}
\item $\spin^\vee(\rho^C_{\pi}(\Frob_v)_{\ssimple})$ and $q_v^{n(n+1)/4}\cdot\spin^\vee(\iota\phi_{\pi_v}(\Frob_v))$ are conjugate in $\GL_{2^n}(\lql)$ by Proposition \ref{prop:PointCounting}.
\item $q(\rho^C_{\pi}(\Frob_v)_{\ssimple}) \sim q(\iota\phi_{\pi_v}(\Frob_v)) \in \GL_{2n+1}(\lql)$.
\end{itemize}
We claim that additionally we have 
\begin{equation}\label{eq:spinor-norm-equal}
\mathcal{N}(\rho^C_{\pi}(\Frob_v)_{\ssimple}) = \mathcal{N}( q_v^{-n(n+1)/4}  \iota\phi_{\pi_v}(\Frob_v))). 
\end{equation}
Once the claim is verified, by Lemma \ref{lem:GaugerSteinberg}, $\rho^C_{\pi}(\Frob_v)_{\ssimple}$ is conjugate to $q_v^{-n(n+1)/4}\cdot\iota\phi_{\pi_v}(\Frob_v)$ 
 in $\GSpin_{2n+1}(\lql)$ since $\{\spin, \std, \cN\}$ is fundamental set of representations of $\GSpin_{2n+1}$ (table above Lemma \ref{lem:GaugerSteinberg}). This proves statement (ii)'. Let us prove the claim. Possibly after conjugation by an element of $\GL_{2^n}(\lql)$, the image of $\spin\circ \rho_\pi$ lies in $\GO_{2^n}$ (resp. $\GSp_{2^n}$) if $n(n+1)/2$ is even (resp. odd), for some symmetric (resp. symplectic) pairing on the underlying $2^n$-dimensional space. Again by the same lemma, we may assume that $\spin(\iota\phi_{\pi_v}(\Frob_v))$ also belongs to $\GO_{2^n}$ (resp. $\GSp_{2^n}$). Hereafter we let the central characters $\omega_\pi$ or $\omega_{\pi_v}$ also denote the corresponding Galois characters via class field theory. For almost all $v$ we have the following isomorphisms (cf. Lemma~\ref{lem:similitude-of-spin-rep}) 
\begin{eqnarray}
(\spin((\rho^C_{\pi}|_{\Gamma_v})_{\ssimple}))^\vee 
& \simeq & |\cdot|^{-n(n+1)/4}\cdot (\spin(\iota\phi_{\pi_ v}))^\vee \\
 ~ & \simeq & ~ |\cdot|^{-n(n+1)/4}\cdot\spin(\iota\phi_{\pi_v})\otimes \omega^{-1}_{\pi_v} \nonumber \\ 
 ~ & \simeq & ~
|\cdot|^{-n(n+1)/2} (\spin((\rho^C_{\pi}|_{\Gamma_v})_{\ssimple}))\otimes \omega^{-1}_{\pi_v},\nonumber
\end{eqnarray}
The above isomorphisms imply that
\begin{equation}\label{eq:temporary}
\spin(\rho^C_{\pi})\simeq \spin(\rho^C_{\pi})^\vee \otimes |\cdot|^{n(n+1)/2}\omega_\pi.
\end{equation}
We have
$$
\spin(\rho^C_{\pi}) \simeq \spin(\rho^C_{\pi})^\vee \otimes \mathcal{N}(\spin(\rho^C_{\pi}))
$$
Recall the equality $\mathcal{N}(\phi_{\pi_v})=\omega_{\pi_v}$ (from functoriality of the Satake isomorphism with respect to the central embedding $\G_m\hra \GSp_{2n}$) and the isomorphism $\spin (\iota\phi_{\pi_v}) \isomto \spin (\iota\phi_{\pi_v})^{\vee} \otimes \simil(\phi_{\pi_v})^{-1}$ induced from the pairing $\langle \cdot , \cdot \rangle$ that defines $\GO_{2^n}$ (resp. $\GSp_{2^n}$). 
From Lemma \ref{lem:for-mult-one} we deduce
$$
\mathcal{N}(\spin(\rho^C_{\pi}))= |\cdot|^{n(n+1)/2}\omega_\pi.
$$
This proves the second part of (i'). Evaluating at unramified places, we obtain \eqref{eq:spinor-norm-equal}, finishing the proof of the claim.

We show statement (i)'. It only remains to check the first part. By Theorem \ref{thm:ExistGaloisRep} and the preceding proof of ({ii})', we have for almost all unramified places $v$ that
$$
\overline {\iota\phi_{\pi_v}(\Frob_v)} \sim \iota\phi_{\pi_v^\flat}(\Frob_v) \sim \rho_{\pi^\flat}(\Frob_v)_{\ssimple},
$$
where we also used that the Satake parameter of the restricted representation $\pi^\flat_v \subset \pi_v$ is equal to the composition of the Satake parameter of $\pi_v$ with the natural surjection $\GSpin_{2n+1}(\C) \to \SO_{2n+1}(\C)$ (cf. \cite[Lem. 5.2]{Xu}). Hence $ \overline {\rho^C_{\pi}(\Frob_v)_{\ssimple}} \sim \rho_{\pi^\flat}(\Frob_v)_{\ssimple}$ for almost all $F$-places $v$. Consequently $\std \circ \overline{\rho^C_{\pi}} \sim \std \circ \rho_{\pi^\flat} = \rho_{\pi^\sharp}$. By Proposition \ref{prop:ConjugacyStdRep} the representation $\overline {\rho^C_{\pi}}$ is $\SO_{2n+1}(\lql)$-conjugate to the representation $\rho_{\pi^\flat}$. This proves the first part of ({i})'.

We prove statement ({v}). The composition of $\rho_\pi$ with $\GSpin_{2n+1}(\lql) \surjects \SO_{2n+1}(\lql)$ is $\rho_{\pi^\flat}$. Hence statement (v) follows from Proposition \ref{prop:DetermineZariskiImage}.

The Galois representation in the cohomology $\uH^i(S_K, \cL)_{\tu{ss}}$ is potentially semistable by Kisin \cite[Thm. 3.2]{KisinPST}. The representation $\rho^C_{\pi}$ appears in this cohomology and is therefore potentially semistable as well (this uses that semisimplification preserves potential semistability). This proves the first assertion of statement ({iii}).

To verify ({iii.a})', let $\mu_{\tu{HT}}(\rho^C_{\pi}, i) \colon \G_{\tu{m}, \lql} \to \GSpin_{2n+1, \lql}$ be the Hodge-Tate cocharacter of $\rho^C_{\pi}$ (Definition \ref{def:Hodge-Tate}). Similarly, we let $\mu_{\tu{HT}}(\rho_{\pi^\flat}, i) \colon \G_{\tu{m}, \lql} \to \GO_{2n+1, \lql}$ be the Hodge-Tate cocharacter of $\rho_{\pi^\flat}$. We need to check that $\mu_{\tu{HT}}(\rho^C_{\pi, v}, \iota y) = \iota \mu_{\tu{Hodge}}(\xi_y)-\frac{n(n+1)}{4}\simil$. In fact it is enough to check the equalities
\begin{eqnarray}\label{eq:CheckHTcharacter1}
q\mu_{\tu{HT}}(\rho^C_{\pi, v}, \iota y) &= & q\iota \mu_{\tu{Hodge}}(\xi_y),\\
\mathcal{N} \circ \mu_{\tu{HT}}(\rho^C_{\pi, v}, \iota y))& = & \mathcal{N} \circ\iota \left(\mu_{\tu{Hodge}}(\xi_y)\right),
\label{eq:CheckHTcharacter2}
\end{eqnarray}
namely after applying the natural surjection $q': \GSpin_{2n+1} \surjects \GO_{2n+1} =\SO_{2n+1}\times \GL_1$ given by $(q,\mathcal{N})$, since $q'$ induces an injection on the set of conjugacy classes of cocharacters. (The map $X_*(\TGSpin) \to X_*(\TGO)$ induced by $q'$ is injective since $\TGSpin\to \TGO$ is an isogeny, so the map is still injective after taking quotients by the common Weyl group.)
Let $y \colon F \to \C$ be an embedding, such that $\iota y$ induces the place $v$. It is easy to see that $q
\mu_{\tu{HT}}(\rho^C_{\pi, v}, \iota y) = \mu_{\tu{HT}}(\rho_{\pi^\flat, v}, \iota y)$ and $q \mu_{\tu{Hodge}}(\xi_y) = \mu_{\tu{Hodge}}(\xi^\flat_y)$ from \eqref{eq:Functoriality-and-MuHodge} and Lemmas \ref{lem:restriction-real-parameter}) and \ref{lem:Hodge-Tate-cochar}. Thus \eqref{eq:CheckHTcharacter1} follows from Theorem \ref{thm:ExistGaloisRep} (iii). The proof of \eqref{eq:CheckHTcharacter2} is similar: we have the following equalities in $X_*(\G_m)=\Z$.
$$
\mathcal{N} \circ \mu_{\tu{HT}}(\rho^C_{\pi, v}, \iota y)) = \mu_{\tu{HT}}(\mathcal{N}\circ\rho^C_{\pi, v}, \iota y)) 
$$
$$
=\iota \mu_{\tu{Hodge}}(\omega_{\pi_y})-\tfrac{n(n+1)}{2}=\mathcal{N}\circ\iota (\mu_{\tu{Hodge}}(\xi_y) - \tfrac{n(n+1)}{4}\simil), 
$$
where the Hodge cocharacter of $\omega_\pi$ at $y$ is denoted by $\mu_{\tu{Hodge}}(\omega_{\pi_y})$. Indeed, the first, second, and third equalities follow from Lemma \ref{lem:Hodge-Tate-cochar}, part (i') of the current theorem (just proved above), and the fact that $\mathcal{N}\circ\phi_{\pi_y}=\phi_{\omega_{\pi_{\xi_y}}}$. The last fact comes from the description of the central character of an $L$-packet; see condition (ii) in \cite[\S3]{LanglandsRealClassification}.
The third equality also uses the easy observation that $\mathcal{N}\circ \simil=2$ under the canonical identification $X_*(\G_m)=\Z$.

We prove statement ({iii.b}). So we assume that $v|\ell$ and that $K_\ell$ is either hyperspecial or contains an Iwahori subgroup of $\GSp_{2n}(F\otimes \Q_\ell)$. We use the following proposition of Conrad:

\begin{proposition}[Conrad]\label{prop:ConradLiftHodgeTateProperty}
Consider a representation $\sigma \colon \Gamma_v \to \GO_{2n+1}(\lql)$ satisfying a basic $p$-adic Hodge theory property $P \in \{\tu{de Rham, crystalline, semistable}\}$. There exists a lift $\sigma' \colon \Gamma_v \to \GSpin_{2n+1}(\lql)$ that satisfies $P$ if $\sigma$ admits a lift $\sigma'' \colon \Gamma_v \to \GSpin_{2n+1}(\lql)$ which is Hodge-Tate.
\end{proposition}
\begin{proof}
(see also Wintenberger \cite{wintenberger}). Combine Proposition 6.5 with Corollary 6.7 of Conrad's article \cite{ConradLifting}. (Conrad has general statements for central extensions of algebraic groups of the form $[Z \injects H' \surjects H]$; we specialized his case to our setting.)
\end{proof}

Write $\omega'_{\pi}:=\omega_{\pi}|\cdot|^{n(n+1)/2}$. Let $\rec(\cdot)$ denote the $\ell$-adic Galois representation corresponding to a Hecke character of $\A_F^\times$ via global class field theory. The product $\rec(\omega'_\pi)\rho_{\pi^\flat} \colon \Gamma \to \GO_{2n+1}(\lql)=\GL_1(\lql)\times \SO_{2n+1}(\lql)$ is the Galois representation corresponding to the automorphic representation $\omega'_{\pi} \otimes \pi^\flat$ of $\A_F^\times \times \Sp_{2n}(\A_F)$, cf. Theorem \ref{thm:ExistGaloisRep}. By comparing Frobenius elements at the unramified places, and using Chebotarev and Brauer-Nesbitt, the representations $\std \circ q' \circ \rho^C_{\pi}$ and $\std \circ \rec(\omega'_{\pi})\rho_{\pi^\flat}$ are isomorphic. Hence $q' \circ \rho^C_{\pi}$ and $\rec(\omega'_\pi)\rho_{\pi^\flat}$ are $\GO_{2n+1}(\lql)$-conjugate by Proposition \ref{prop:ConjugacyStdRep}. We make two observations:
\begin{itemize}
\item We saw that $\rho^C_{\pi}$ is potentially semistable. In particular $\rec(\omega'_\pi)\rho_{\pi^\flat}$ has a lift which is a Hodge-Tate representation (namely $\rho^C_{\pi}$ is such a lift).
\item Under the assumption of ({iii.b}), $\rho_{\pi^\flat}$ is known to be crystalline (resp. semistable\footnote{If $\pi_v$ has a nonzero vector fixed under the standard Iwahori subgroup $I$ of $\GSp_{2n}(F_v)$, then the restriction $\pi_v^{\flat}$ has a nonzero fixed vector under an Iwahori subgroup of $\Sp_{2n}(F_v)$, which may be not conjugate to $I \cap \Sp_{2n}(F_v)$, but the Galois representation $\rho_{\pi^\flat}$ is still semi-stable at $v$.}) by ({iv}) and ({v})) of {Theo\-rem} \ref{thm:ExistGaloisRep}).
We have $K_\ell \cap Z(\A_F) = Z(\wh \cO_F)$ if $K_\ell$ contains an Iwahori subgroup and thus the character $\rec(\omega'_{\pi})$ is crystalline at all $v|\ell$. Therefore $\rec(\omega'_{\pi}) \rho_{\pi^\flat}$ is crystalline (resp. semistable) as well.
\end{itemize}

By Proposition \ref{prop:ConradLiftHodgeTateProperty} and Theorem \ref{thm:ExistGaloisRep}.(v), the local representation $\rec(\omega'_{\pi})\rho_{\pi^\flat}|_{\Gamma_v}$ has a lift $r_v := \wt {\rec(\omega'_{\pi})\rho_{\pi^\flat}} \colon \Gamma_v \to \GSpin_{2n+1}(\lql)$ which is crystalline/semistable. Since $r_v$ and $\rho^C_{\pi}|_{\Gamma_v}$ are both lifts of $\rec(\omega'_{\pi})\rho_{\pi^\flat}|_{\Gamma_v}$, the representations $r_v$ and $\rho^C_{\pi}|_{\Gamma_v}$ differ by a quadratic character $\chi_v$.
Hence statement ({iii.b}) follows.

We prove statement ({vi}). Let $\rho' \colon \Gamma \to \GSpin_{2n+1}(\lql)$ be semisimple and such that for almost all $F$-places $v$ where $\rho'$ and $\rho^C_\pi$ are unramified, $\rho'(\Frob_v)_{\textup{ss}}$ is conjugate to $\rho^C_{\pi}(\Frob_v)_{\textup{ss}}$. We know that $\rho_{\pi}^C$ has a regular unipotent element in its image (it suffices to check this after composing with $\GSpin_{2n+1}(\lql) \surjects \SO_{2n+1}(\lql)$, in which case it follows from Proposition \ref{prop:DetermineZariskiImage}). Hence Proposition \ref{prop:SpinConjugacyA} implies that the two representation are conjugate.

Statement ({iv}) reduces to checking that $\rho_{\pi^\flat}:\Gamma\to \GO_{2n+1}(\lql)$ is totally odd since the covering map $\GSpin_{2n+1}\to \GO_{2n+1}$ induces an isomorphism on the level of Lie algebras and the adjoint action of $\GSpin_{2n+1}$ factors through that of $\GO_{2n+1}$. We already know from Theorem \ref{thm:ExistGaloisRep}.({vi}) that $\rho_{\pi^\flat}$ is totally odd, so the proof is complete.
\end{proof}

\section{Compatibility at unramified places}\label{sect:CompatUnrPlaces}

Let $\pi$ be an automorphic representation of $\GSp_{2n}(\A_F)$, satisfying the same conditions as in Section \ref{sect:Construction}. In this section we identify the representation $\rho_{\pi, v}$ from Theorem \ref{thm:RhoPi0} at \emph{all} places $v$ not above $S_{\tu{bad}}\cup\{\ell\}$. 

\begin{proposition}\label{prop:AllUnramifiedPlaces}
Let $v$ be a finite $F$-place such that $p:=v|_\Q$ does not lie in $S_{\tu{bad}} \cup \{\ell\}$. Then $\rho^C_{\pi}$ is unramified at $v$. Moreover $\rho^C_{\pi}(\Frob_v)_{\ssimple}$ is conjugate to $\iota\phi_{\pi_v |\simil|^{n(n+1)/4}}(\Frob_v)$.
\end{proposition}
\begin{proof}
Let $\pi^\natural$ be a transfer of $\pi$ to the inner form $G(\A_F)$ of $\GSp_{2n}(\A_F)$ (Proposition \ref{prop:JL-for-GSp}). Let $B(\pi^\natural)$ be the set of cuspidal automorphic representations $\tau$ of $G(\A_F)$ such that
\begin{itemize}
\item $\tau_{\vst}$ and $\pi_{\vst}$ are isomorphic up to a twist by an unramified character,
\item $\tau^{ \infty, \vst,v}$ and $\pi^{\natural, \infty, \vst,v}$ are isomorphic,
\item $\tau_v$ is unramified,
\item $\tau_ \infty$ is $\xi$-cohomological.
\end{itemize}
To compare with the definition of $A(\pi^\natural)$, notice that the condition at $v$ is different. We define an equivalence relation $\approx$ on the set $B(\pi^\natural)$ by declaring that $\tau_1\approx \tau_2$ if and only if $\tau_2 \in A(\tau_1)$ (hence, $\tau_1 \approx \tau_2$ if and only if $\tau_{1, v} \simeq \tau_{2,v}$). Define a (true) representation of $\Gamma$ (see \eqref{eq:Rho2Plus} for $\rho_2^\coh(\tau)$, cf. Proposition \ref{prop:PointCounting} and Corollary \ref{cor:lgc-unramified})
\begin{equation}\label{eq:Rho3Plus}
\rho^\coh_3:=\sum_{\tau \in B(\pi^\natural)/{\approx}} \rho_2^\coh(\tau) =\sum_{\tau \in B(\pi^\natural)/{\approx}} i_{a(\tau)}\circ \rho_2(\tau).
\end{equation}
Recall the definition of $a(\tau)$ from \eqref{eq:Defa0}. Define $b(\pi^\natural) := \sum_{\tau \in B(\pi^\natural)/\approx} a(\tau)$.

Since $\rho_2(\pi^\natural)$ and $\rho_2(\tau)$ have the same Frobenius trace at almost all places for $\tau \in B(\pi^\natural)$, we deduce that $\rho_2(\pi^\natural)\simeq \rho_2(\tau)$. Hence
$$
\rho_3^\coh\simeq i_{b(\pi^\natural)}\circ\rho_2(\pi^\natural).
$$
We adapt the argument of Proposition \ref{prop:PointCounting} to the slightly different setting here. Consider the function $f$ on $G(\A_F)$ of the form $f = f_ \infty \otimes f_{\vst} \otimes \one_{K_v} \otimes f^{ \infty, \vst, v}$, where $f_ \infty$ and $f_{\vst}$ are as in the proof of that proposition, and $f^{ \infty,\vst, v}$ is such that, for all automorphic representations $\tau$ of $G(\A_F)$ with $\tau^{ \infty, K} \neq 0$ and $\Tr \tau_ \infty(f_ \infty) \neq 0$, we have (observe the slight difference between Equations \eqref{eq:AuxFunction} and \eqref{eq:AuxFunction2} at the place $v$):
\begin{equation}\label{eq:AuxFunction2}
\Tr \tau^{ \infty, \vst, v}(f^{ \infty, \vst, v}) = \begin{cases}
1 & \tu{if $\tau^{ \infty,\vst, v} \simeq \pi^{\natural, \infty,\vst,v}$} \cr
0 & \tu{otherwise.}
\end{cases}
\end{equation}
Arguing as in the paragraph between \eqref{eq:SpectralKottwitzFormula2} and \eqref{eq:LKPCF4}, but with $B(\pi^\natural)$, $b(\pi^\natural)$, and $\rho_3^\coh$ in place of $A(\pi^\natural)$, $a(\pi^\natural)$, and $\rho^\coh_2$, we obtain
$$
\Tr (\Frob_v^j, \rho_3^\coh) = \sum_{\tau \in B(\pi^\natural)/\approx} a(\tau) \Tr\tau_v(f_j) = \sum_{\tau \in B(\pi^\natural)/\approx} a(\tau)q_v^{jn(n+1)/4} \Tr(\spin^\vee(\iota\phi_{\tau_v}))(\Frob_v^j),
$$
where the last equality comes from \eqref{eq:Kottwitz-f_j}. Thus the proof boils down to the next lemma (Lemma \ref{lem:same-v-compo}). Indeed the lemma and the last equality imply that $\Tr \rho_2(\pi^\natural)(\Frob_v^j)$ is equal to $q_v^{n(n+1)/4} \Tr(\spin^\vee(\iota\phi_{\pi^\natural_v}))(\Frob_v^j)$. As in the proof of Theorem \ref{thm:RhoPi0}, since $\rho_2(\pi^\natural)=\spin^\vee \circ \rho_\pi^C$ by construction, the semisimple parts of $\rho_\pi^C(\Frob_v)$ and ${q_v}^{-n(n+1)/4}\iota\phi_{\pi^\natural_v}(\Frob_v)$ are conjugate. (Note that ${q_v}^{-n(n+1)/4}\phi_{\pi^\natural_v}(\Frob_v)=\phi_{\pi_v |\simil|^{n(n+1)/4}}(\Frob_v)$.)
\end{proof}

\begin{lemma}\label{lem:same-v-compo}
With the above notation, if $\tau \in B(\pi^\natural)$ then $\tau_v\simeq \pi^\natural_v$.
\end{lemma}

\begin{proof}
Let $\tau^*$ and $\pi^*$ denote transfers of $\tau$ and $\pi^\natural$ from $G$ to $\GSp_{2n}$ via part (2) of Proposition \ref{prop:JL-for-GSp}; the assumption there is satisfied by Corollary \ref{cor:Pi-is-in-Disc-series-Packet-2} (in fact we can just take $\pi^*$ to be $\pi$). In particular $\tau^*_x\simeq \tau_x$ and $\pi^*_x\simeq \pi^\natural_x$ at all finite places $x$ where $\tau_x$ and $\pi^\natural_x$ are unramified). In particular this is true for $x=v$, so it suffices to show that $\tau^*_v\simeq \pi^*_v$.

By \cite[Thm. 1.8]{Xu} we see that $\tau^*$ and $\pi^*$ belong to global $L$-packets $\Pi_1=\otimes_x \Pi_{1,x}$ and $\Pi_2=\otimes_x \Pi_{2,x}$ (as constructed in that paper), respectively, such that $\Pi_1=\Pi_2\otimes \omega$ for a quadratic Hecke character $\omega:F^\times \bs \A_F^\times \ra \{\pm1\}$ (which is lifted to a character of $\GSp_{2n}(\A_F)$ via the similitude character). Since each local $L$-packet has at most one unramified representation by \cite[Prop. 4.4.(3)]{Xu} we see that for almost all finite places $x$, we have $\tau^*_x\simeq \pi^*_x \otimes \omega_{x}$. Hence $\pi^\natural_x\simeq \pi^\natural_x\otimes \omega_x$. (Recall that $\tau_x\simeq \pi^\natural_x$ by the initial assumption at almost all $x$.) This implies through (ii') of Theorem \ref{thm:RhoPi0} that $\spin(\rho^C_\pi)\simeq \spin(\rho^C_\pi)\otimes \omega$ (viewing $\omega$ as a Galois character via class field theory). Since $\rho^C_\pi$ has a regular unipotent element in its image, it follows that $\omega=1$ (by Lemma \ref{lem:for-mult-one}). Hence the unramified representations $\tau^*_v$ and $\pi^*_v$ belong to the same local $L$-packet, implying that $\tau^*_v\simeq \pi^*_v$ as desired.
\end{proof}

We are now ready to collect everything and prove the main result.

\begin{theorem}\label{thm:ThmAisTrue}
Theorem \ref{thm:A} is true.
\end{theorem}
\begin{proof}
Let us caution the reader that the normalization for $\pi$ changes. Namely given $\pi$ as in Theorem \ref{thm:A} (so that $\pi|\simil|^{-n(n+1)/4}$ is $\xi$-cohomological), the preceding discussions in Sections \ref{sect:Construction} and \ref{sect:CompatUnrPlaces} apply to $\pi^C:=\pi|\simil|^{-n(n+1)/4}$. We thus define $\rho_\pi := \rho^C_{\pi^C}$.

Theorem \ref{thm:RhoPi0} proves (iii.b), (iv) and (v) of Theorem \ref{thm:A}. By Proposition \ref{prop:AllUnramifiedPlaces} we also have (ii). Statements (i) and (iii.a) follow from Theorem \ref{thm:RhoPi0}.(i') and \ref{thm:RhoPi0}.(iii.a'), with evident changes due to the different normalization. It remains to verify (iii.c) and (vi) of Theorem \ref{thm:A}. Item (vi) follows from Proposition \ref{prop:SpinConjugacyA} since $\rho_\pi$ contains a regular unipotent element in its image (see the proof of Proposition \ref{prop:DetermineZariskiImage}). Item (iii.c) is about the crystallineness of $\rho_\pi$. By definition, it is enough to prove that $\spin (\rho^C_{\pi^C})$ is crystalline, and hence it is enough to prove that the compact support cohomology of the Shimura variety $\Sh_K$ with coefficients in $\cL_{\xi}$ is crystalline. In case $F \neq \Q$ the variety $\Sh_K$ is projective, and the crystallineness follows from the thesis of Tom Lovering \cite{Lovering}. If $F = \Q$, then the Shimura variety $S_K$ is the classical Siegel space and in this case the representation was shown to be crystalline by Faltings-Chai \cite[VI.6]{FaltingsChai}. Alternatively we may reduce to this case using base change (Proposition \ref{prop:BC-for-GSp}) to a real quadratic extension $F/\Q$ in which $\ell$ is completely split.
\end{proof}

\begin{remark}\label{rem:odd}
 In the introduction we claimed that we prove \cite[Conj. 5.16]{BuzzardGee} up to Frobenius semisimplification in the case at hand. Our Theorem \ref{thm:A} covers everything except the verification that the $\GSpin_{2n+1}(\lql)$-conjugacy class of $\rho_\pi(c_v)$ coincides with the one predicted by \emph{loc. cit.} (Our theorem only shows that $\rho_\pi(c_v)$ is odd.) Let $\alpha_v \in\GSpin_{2n+1}(\C)$ be as in \cite[Conj. 5.16]{BuzzardGee}, and let $\rho_\pi=\rho_{\pi,\iota}$ with $\iota:\C\simeq \lql$. Let us sketch the argument that $\rho_\pi(c_v)$ is conjugate to $\iota(\alpha_v)$. By Lemma \ref{lem:GaugerSteinberg}, it is enough to prove the conjugacy in $\GL_1$ (under spinor norm), $\SO_{2n+1}$, and $\GL_{2^n}$ (under $\spin$). The conjugacy in $\GL_1$ is checked using the second part of Theorem \ref{thm:A} (i). The conjugacy in $\SO_{2n+1}$ follows from the uniqueness of odd conjugacy class in $\SO_{2n+1}$ (or in any adjoint group, cf. \cite[\S2]{GrossOdd}, since both conjugacy classes are odd in $\SO_{2n+1}$ (in view of Lemma \ref{lem:odd-GO} for $\rho_\pi(c_v)$; by direct computation for $\iota(\alpha_v)$). Then $\ol{\spin}( \rho_\pi(c_v))$ and $\ol{\spin}(\iota(\alpha_v))$ are conjugate in $\PGL_{2^n}$; moreover they are (up to conjugation) images of the diagonal matrix in $\GL_{2^n}$ with each of $1$ and $-1$ appearing exactly $2^{n-1}$ times. Since $\rho_\pi(c_v),\spin(\iota(\alpha_v) \in \GL_{2^n}(\lql)$ are elements of order 2 whose images in $\PGL_{2^{n-1}}$ are determined as such, we conclude that $\rho_\pi(c_v)\sim \spin(\iota(\alpha_v)$ (with eigenvalues $1$ and $-1$, with multiplicities $2^{n-1}$ each). We are done with showing that $\rho_\pi(c_v)\sim\iota(\alpha_v)$ in $\GSpin_{2n+1}(\lql)$.

\end{remark}

\section{Galois representations for the exceptional group $G_2$}\label{sec:G_2}

As an application of our main theorems we realize some instances of the global Langlands correspondence for $G_2$ in the cohomology of Siegel modular varieties of genus 3 via theta correspondence, following the strategy of Gross--Savin \cite{GrossSavin}. In particular the constructed Galois representations will be motivic\footnote{In this paper ``motivic'' means that it appears in the \'etale cohomology of a smooth quasi-projective variety over a number field.} and come in compatible families as such. We work over $F=\Q$ (as opposed to a general totally real field) mainly because this is the case in \cite{GrossSavin}.

More precisely we are writing $G_2$ for the split simple group of type $G_2$ defined over $\Z$. Denote by $G_2^c$ the inner form of $G_2$ over $\Q$ which is split at all finite places and such that $G_2^c(\R)$ is compact.
The dual group of $G_2$ is $G_2(\C)$ and fits in the diagram
$$
\xymatrix{ \PGL_2(\C) \ar@{^(->}[r] & G_2(\C) \ar@/^2pc/[rr]_{\eta} \ar@{^(->}[r] \ar@{^(->}^-{\zeta}[d] & \SO_7(\C) \ar@{^(->}[r] \ar@{^(->}[d] & \GL_7(\C) \ar@{^(->}[d] \\ &\Spin_7(\C) \ar@{^(->}[r] \ar@/_2pc/[rr]^{\spin} & \SO_8(\C) \ar@{^(->}[r] & \GL_8(\C) }
$$
such that $G_2(\C)=\SO_7(\C)\cap \Spin_7(\C)$. The subgroup $\PGL_2(\C)$ is given by a choice of a regular unipotent element of $\SO_8(\C)$. See \cite[pp.169-170]{GrossSavin} for details. Note that the spin representation of $\Spin_7$ is orthogonal and thus factors through $\SO_8$. The $8$-dimensional representation $G_2(\C)\hra \GL_8(\C)$ decomposes into 1-dimensional and 7-dimensional irreducible pieces. The former is the trivial representation. The latter factors through $\SO_7(\C)$. Evidently all this is true with $\lql$ in place of $\C$.

The (exceptional) theta lift from each of $G_2$ and $G_2^c$ to $\PGSp_6$ uses the fact that $(G_2,\PGSp_6)$ and $(G^c_2,\PGSp_6)$ are dual reductive pairs in groups of type $E_7$ \cite{GRS-G2,GrossSavin}. In this section we concentrate on the case of $G_2^c$, only commenting on the case of $G_2$ at the end.
Every irreducible admissible representation of $G^c_2(\R)$ is finite-dimensional, and both ($C$-)cohomological and $L$-cohomological since the half sum of all positive roots of $G_2$ is integral.
Note that an automorphic representation $\pi$ of $\PGSp_6(\A)$ is the same as an automorphic representation of $\GSp_6(\A)$ with trivial central character, so we will use them interchangeably. For such a $\pi$ the subgroup $\rho_\pi(\Gamma)$ of $\GSpin_7(\lql)$ is contained in $\Spin_7(\lql)$ by (i) of Theorem \ref{thm:A}.

\begin{theorem}\label{thm:G2-A}
Let $\sigma$ be an automorphic representation of $G_2^c(\A)$. Assume that
\begin{itemize}
\item $\sigma$ admits a theta lift to a cuspidal automorphic representation $\pi$ on $\PGSp_6(\A)$.
\item $\sigma_\vst$ is the Steinberg representation at a finite place $\vst$.
\end{itemize}
Then for each prime $\ell$ and $\iota:\C\simeq \lql$, there exists a representation $\rho_\sigma=\rho_{\sigma,\iota}: \Gamma\ra G_2(\lql)$ such that
\begin{enumerate}
\item For every finite place $v\neq \ell$ where $\sigma$ is unramified, $\rho_\sigma$ is unramified at $v$. Moreover $ (\rho_\sigma |_{W_{\Q_v}})_{\ssimple}\simeq \iota\phi_{\sigma_v}$ as unramified $L$-parameters for $G_2$.
\item $\rho_{\sigma_\ell}$ is de Rham with $\mu_{\HT}(\rho_{\sigma_\ell})=\iota\mu_{\Hodge}(\sigma^\vee_ \infty)$.\footnote{Note that $\sigma_ \infty$ is $\xi$-cohomological for $\xi=\sigma^\vee_ \infty$.}
\item If $\sigma_\ell$ is unramified then $\rho_{\sigma_\ell}$ is crystalline.
\item $\zeta\circ \rho_\sigma\simeq \rho_\pi$.
\end{enumerate}
\end{theorem}

Before starting the proof, we recall the basic properties of the theta lift $\pi$. We see from \cite[\S4 Prop. 3.1, \S4 Prop. 3.19]{GrossSavin} that $\pi_\vst$ is the Steinberg representation, and that $\pi_v$ is unramified whenever $\sigma_v$ is unramified at a finite place $v$ and the unramified $L$-parameters are related via
\begin{equation}\label{eq:Satake-under-theta}
\zeta \phi_{\sigma_v}\simeq \phi_{\pi_v}.
\end{equation}
Furthermore $\pi_ \infty$ is an $L$-algebraic discrete series representation whose parameter can be explicitly described in terms of $\sigma$ (\S3 Cor. 3.9 of \emph{loc. cit.}).

\begin{proof}
We apply Theorem \ref{thm:A} to the text above $\pi$ to obtain a representation $\rho_\pi:\Gamma\ra \Spin_7(\lql)$. Note that $\spin\circ \rho_\pi$ is a semisimple representation by construction. Since the image of $\rho_{\pi^\flat}$ contains a regular unipotent element of $\SO_7(\lql)$, it follows that $\rho_\pi(\Gamma)$ contains a regular unipotent of $\Spin_7(\lql)$ and also that of $\SO_8(\lql)$. (The representation $\spin \colon \Spin_7 \to \GL_{2^3}$ induces an inclusion $\Spin_7 \hookrightarrow \SO_8$, under which a regular unipotent maps to a regular unipotent.) By \eqref{eq:Satake-under-theta} and the Chebotarev density theorem, the image of $\rho_\pi$ is locally contained in $G_2$ in the terminology of Gross-Savin, so \cite[\S2 Cor. 2.4]{GrossSavin} implies that $\rho_\pi(\Gamma)$ is contained in $G_2(\lql)$ (given as $\SO_7(\lql)\cap \Spin_7(\lql)$ for a suitable choice of the embedding $\SO_7\hra \SO_8$; here $\spin \colon \Spin_7\hra \SO_8$ is fixed). Hence we have $\rho_\sigma$ satisfying (4), namely that $\rho_\pi\simeq \zeta\circ \rho_\sigma$.

Assertions (1)-(3) of the theorem follow from Theorem \ref{thm:A} and the fact that the set of Weyl group orbits on the maximal torus (resp. on the cocharacter group of a maximal torus) for $G_2$ maps injectively onto that for $\Spin_7$. (The latter can be checked explicitly.)
\end{proof}

\begin{example}\label{ex:nonzero-theta}
There is a unique automorphic representation $\sigma$ of $G_2^c(\A)$ unramified outside $5$ such that $\sigma_5$ is the Steinberg representation and $\sigma_ \infty$ is the trivial representation \cite[\S1 Prop. 7.12]{GrossSavin}. Proposition 5.8 in \S5 of \emph{loc. cit.} (via a computer calculation due to Lansky and Pollack) tells us that $\sigma$ admits a nontrivial cuspidal theta lift to $\PGSp_6(\A)$. Proposition 5.5 in the same section gives another example of nontrivial theta lift but we will not consider it here.
\end{example}

We confirm the prediction of Gross--Savin that a rank 7 motive whose motivic Galois group is $G_2$ is realized in the middle degree cohomology of a Siegel modular variety of genus 3.

\begin{corollary}\label{cor:G2}
Let $\sigma$ be as in Example \ref{ex:nonzero-theta}. Write $\pi$ for its theta lift. Then $\rho_\sigma$ has Zariski dense image in $G_2(\lql)$. Moreover $\spin\circ \rho_\pi$ is isomorphic to the direct sum of $\eta\circ \rho_\sigma$ and the trivial representation. In particular $\eta\circ \rho_\sigma$ and the trivial representation appear in the $\pi^ \infty$-isotypic part in $\uH^6_c(
\Sh,\lql)(3)$, where $\Sh$ is the tower of Siegel modular varieties for $\GSp_6$ and $(3)$ denotes the Tate twist (i.e. the cube power of the cyclotomic character).
\end{corollary}
\begin{proof}
If the image is not dense in $G_2(\lql)$ then the proof of \cite[\S2 Prop. 2.3]{GrossSavin} shows that the Zariski closure of $\rho_\pi(\Gamma)$ is $\PGL_2$. However we see from the explicit computation of Hecke operators at $2$ and $3$ on $\sigma$ carried out by Lansky--Pollack \cite[p.45, Table V]{LanskyPollack} that the Satake parameters at $2$ and $3$ do not come from $\PGL_2$.\footnote{The authors check \cite[\S4.3]{LanskyPollack} that the Satake parameters do not come from $\SL_2$ via the map $\SL_2(\C)\ra G_2(\C)$ induced by a regular unipotent element. But the latter map factors through the projection $\SL_2\ra \PGL_2$.} We conclude that $\rho_\sigma(\Gamma)$ is dense in $G_2(\lql)$. The second assertion of the corollary is clear from (4) of Theorem \ref{thm:G2-A} and the construction of $\rho_\pi$.
\end{proof}

\begin{remark}
 Corollary \ref{cor:mult-one-cohomology} implies that the $\pi^ \infty$-isotypic part in $\uH^6_c(
\Sh,\lql)(3)$ is 8-dimensional, consisting of $\eta\circ \rho_\sigma$ and the trivial representation without multiplicity.
\end{remark}

\begin{remark}\label{remark:Tate}
The Tate conjecture predicts the existence of an algebraic cycle on $\Sh$ which should give rise to the trivial representation in the corollary. Gross and Savin suggest that it should come from a Hilbert modular subvariety of $\Sh$ for a totally real cubic extension of $\Q$. See \cite[\S6]{GrossSavin} for details.
\end{remark}

\begin{remark}
Ginzburg--Rallis--Soudry \cite[Thm. B]{GRS-G2} showed that every globally generic automorphic representation $\sigma$ of $G_2(\A)$ admits a theta lift to $\PGSp_6(\A)$. (The result is valid over every number field.) Using this, Khare--Larsen--Savin \cite{KLS2} established instances of global Langlands correspondence for $G_2(\A)$, the analogue of Theorem \ref{thm:G2-A} with $G_2$ in place of $G_2^c$, under a suitable local hypothesis. (See Section 6 of their paper for the hypothesis. They prescribe a special kind of supercuspidal representation instead of the Steinberg representation.) In our notation, their $\rho_\sigma$ is constructed inside the $\SO_7$-valued representation $\rho_{\pi^\flat}$, where the point is to show that the image is contained and Zariski dense in $G_2$ \cite[Cor. 9.5]{KLS2}.\footnote{Thereby they give an affirmative answer to Serre's question on the motivic Galois group of type $G_2$, since it is well known that $\rho_{\pi^\flat}$ appears in the cohomology of a unitary PEL-type Shimura variety after a quadratic base change, along the way to proving a result on the inverse Galois problem. Sometimes this contribution of \cite{KLS2} is overlooked in the literature.}
\end{remark}

\section{Automorphic multiplicity}\label{sect:AutomorphicMultiplicity}

In this section we prove multiplicity one results for automorphic representations of $\GSp_{2n}(\A_F)$ and the inner form $G(\A_F)$ under consideration. For $\GSp_{2n}$ we deduce this from Bin Xu's multiplicity formula using the following property (cf. Lemma \ref{lem:for-mult-one})
$$
\forall \omega \colon \Gamma \to \lql^\times: \rho_\pi \approx \rho_\pi \otimes \omega \Rightarrow \omega = 1
$$
of the associated Galois representations $\rho_\pi$. The result is then transferred to $G(\A_F)$ via the trace formula (\S\ref{sect:TraceFormula}). This is standard except when the highest weight of $\xi$ is not regular: In that case we need the input from Shimura varieties that the automorphic representations of $G(\A_F)$ of interest are concentrated in the middle degree.

\begin{theorem}\label{thm:AutomorphicMultiplicity}
Let $n\ge 2$. Let $\pi$ be a cuspidal automorphic representation of $\GSp_{2n}(\A_F)$ such that conditions (St) and (L-Coh) hold.
The automorphic multiplicity $m(\pi)$ of $\pi$ is equal to $1$, i.e. Theorem \ref{thm:IntroAutomMult} is true.
\end{theorem}
\begin{proof}
By Xu \cite[Prop. 1.7]{Xu} we have the formula $$m(\pi) = m(\pi^\flat) |Y(\pi)/\alpha(\cS_\phi)|,$$ where $m(\pi^\flat)=1$ in our case by Arthur \cite[Thm. 1.5.2]{ArthurBook}. The group $Y(\pi)$ is equal to the set of characters $\omega \colon \GSp_{2n}(\A_F) \to \C^\times$ which are trivial on $\GSp_{2n}(F)\A_F^\times\Sp_{2n}(\A_F) \subset \GSp_{2n}(\A_F)$ and are such that $\pi \simeq \pi \otimes \omega$. The definition of the subgroup $\alpha(\cS_\phi)$ of $Y(\pi)$ is not important for us: We claim that $Y(\pi) = 1$. Let $\omega \in Y(\pi)$ and let $\chi \colon \Gamma \to \lql^\times$ be the corresponding character via class field theory. We get $ \rho_\pi \simeq \rho_\pi\otimes \chi$ from Theorem \ref{thm:A}, (ii) and (vi). By Lemma \ref{lem:for-mult-one}, it follows that $\chi = 1$ and hence $\omega$ is trivial as well.
\end{proof}

We now prove the analogue of Theorem \ref{thm:AutomorphicMultiplicity} for an inner form of $\GSp_{2n, F}$. Since the analogue of results by Arthur and Xu has not been worked out yet for non-quasi-split inner forms,\footnote{Certain inner forms of $\Sp_{2n,F}$ and special orthogonal groups have been treated by Ta\"ibi \cite{TaibiTransfer}.} we only obtain a partial result. Let $G$ be an inner form of $\GSp_{2n, F}$ in the construction of Shimura varieties, cf. \S\ref{sect:ShimuraDatum}.

\begin{theorem}\label{thm:AutomorphicMultiplicity2}
Let $\pi$ be a cuspidal automorphic representation of $G(\A_F)$ satisfying conditions (L-coh) and (St). Then the automorphic multiplicity of $\pi$ is equal to $1$.
\end{theorem}
\begin{proof}
Let $f^{G}_ \infty:=|\Pi^G_{\xi}|^{-1}(-1)^{q(G_ \infty)} f^G_{\Lef, \infty}$ and $f^{G^*}_ \infty:=|\Pi^{G^*}_{\xi}|^{-1} (-1)^{q(G^*_ \infty)} f^{G^*}_{\Lef, \infty}$, where $f^G_{\Lef, \infty}$ and $f^{G^*}_{\Lef, \infty}$ denote the Lefschetz functions on $G(F_ \infty)$ and $G^*(F_ \infty)$ as in Equation \eqref{eq:AvgLef}), respectively. At the place $\vst$, we consider the truncated Lefschetz functions $f_{\Lef,\vst}^G$ and $f_{\Lef,\vst}^{G^*}$ as introduced in \eqref{eq:fLef-truncated}, and put $f^G_{\vst}:=(-1)^{q(G_{\vst})} f_{\Lef,\vst}^G$ and $f^{G^*}_{\vst}:=(-1)^{q(G^*_{\vst})} f_{\Lef,\vst}^{G^*}$. Let $f^{ \infty,\vst} \in \cH(G(\A_F^{ \infty,\vst}))$ be arbitrary. Let $\pi$ (resp. $\pi^*$) be a discrete automorphic representation of $G$ (resp. $G^*$).
Note that $f^G_v$ and $f^{G^*}_v$ are associated at $v \in \{\vst, \infty\}$ by Lemmas \ref{lem:LefschetzMatching-finite} and \ref{lem:LefschetzMatching} (with the choice of Haar measures and transfer factors explained there). Hence $f^{ \infty,\vst} f^G_{\vst} f^G_ \infty$ and $f^{ \infty,\vst} f^{G^*}_{\vst} f^{G^*}_ \infty$ are associated. Implicitly we are choosing the local transfer factor at each place $v$ to be identically equal to the sign $e(G_v)$ (whenever nonzero)\footnote{In the special case of inner forms, the definition of transfer factors in \cite{LanglandsShelstad} simplifies to yield a nonzero constant, independent of conjugacy classes, whenever two conjugacy classes are matching.} so that the global transfer factor always equals one (whenever nonzero).

Comparing the trace formulas for $G$ and $G^*$ and arguing as in the proof of Proposition \ref{prop:JL-for-GSp}, we obtain
$$
\sum_\tau m(\tau)\Tr \tau(f^{ \infty,\vst} f^G_{\vst} f^G_ \infty) = \sum_{\tau^*} m(\tau^*)\Tr\tau^*(f^{ \infty,\vst} f^{G^*}_{\vst} f^{G^*}_ \infty),
$$
where $\tau$ (resp. $\tau^*$) runs over discrete automorphic representations of $G(\A_F)$ (resp. $G^*(\A_F)$), and $m(\ )$ denotes the automorphic multiplicity as usual.
By Lemma \ref{lem:DiscAutomRepWithSteinbComponent} the functions $f_{\vst}^G$ (resp. $f_{\vst}^{G^*}$) have non-zero trace against $\pi_{\vst}$ (resp. $\pi_{\vst}^*$) only if $\pi_{\vst}$ (resp. $\pi_{\vst}^*$) is an unramified twist of the Steinberg representation. Therefore
\begin{equation}\label{eq:M1_SimpleTF}
|\Pi^G_{\xi}|^{-1} \sum_{\tau} m(\tau) (-1)^{q(G_ \infty)}\ep(\tau_ \infty \otimes \xi) = |\Pi^{G^*}_\xi|^{-1} \sum_{\tau^*} m(\tau^*)(-1)^{q(G^*_ \infty)} \ep(\tau^*_ \infty \otimes \xi),
\end{equation}
where $\Pi^G_{\xi}$ and $\Pi^{G^*}_\xi$ are discrete $L$-packets. In Equation \eqref{eq:M1_SimpleTF}, each sum ranges over the set of $\xi$-cohomological discrete automorphic representations which are Steinberg (up to unramified twist) at $\vst$ and isomorphic to $\pi^{\vst, \infty}$ away from $\vst$ and $ \infty$.

For the quasi-split group $G^*$, Corollary \ref{cor:Pi-is-in-Disc-series-Packet} shows that any $\tau^*$ in Equation \eqref{eq:M1_SimpleTF} is tempered at infinity, thus $\tau^*_ \infty$ is a discrete series (and belongs to $\Pi^{G^*}_\xi$). However Corollary \ref{cor:Pi-is-in-Disc-series-Packet} does not imply the analogue for $\tau$. For $\tau_ \infty$, Corollary \ref{cor:Pi-is-in-Disc-series-Packet-2} implies that $\tau_ \infty$ must be a discrete series representation. Hence $\ep(\tau_ \infty \otimes \xi)=(-1)^{q(G_ \infty)}$ and $\ep(\tau^*_ \infty \otimes \xi)=(-1)^{q(G^*_ \infty)}$ above, allowing us to simplify \eqref{eq:M1_SimpleTF} to
\begin{equation}\label{eq:M1_SimpleTF-2}
| \Pi^G_\xi |^{-1} \sum_\tau m(\tau) = | \Pi^{G^*}_\xi |^{-1} \sum_{\tau^*} m(\tau^*),
\end{equation}
where each sum runs over discrete automorphic representations which are Steinberg (up to unramified twist at $\vst$), belong to the specified discrete $L$-packet at infinity, and isomorphic to $\pi^{\vst, \infty}$ away from $\vst$ and $ \infty$.

At this point it is useful to show the following lemma by the trace formula argument.

\begin{lemma}\label{lem:MT-1-independent}
Suppose that $F\neq \Q$.
We have $m(\tau)=m(\tau^ \infty \tau'_ \infty)$ and $m(\tau^*)=m(\tau^{*, \infty} \tau^{*,\prime}_ \infty)$, where $\tau'_ \infty$, $\tau^{*,\prime}_ \infty$ are arbitrary members of $\Pi^G_\xi$ or $\Pi^{G^*}_\xi$.
\end{lemma}
\begin{proof}
The assertion that $m(\tau)=m(\tau^ \infty \tau'_ \infty)$ was already shown in
Corollary \ref{cor:A-is-integer}. (We have $\tau \in A(\pi)$ in the notation there.)
The same argument for $G^*$ proves $m(\tau^*)=m(\tau^{*, \infty} \tau^{*,\prime}_ \infty)$, with the only change occurring at $ \infty$. Namely, instead of appealing to Corollary \ref{cor:Pi-is-in-Disc-series-Packet}, we know that the analogue of \eqref{eq:A-is-integer} for $G^*$ receives contributions only from those representations whose components at $ \infty$ belong to $\Pi^{G^*}_\xi$, by Corollary \ref{cor:Pi-is-in-Disc-series-Packet}. (We already made this observation in the discussion above \eqref{eq:M1_SimpleTF-2}.)
\end{proof}

We continue the proof of Theorem \ref{thm:AutomorphicMultiplicity2}.
The representation $\pi$ as in the statement of the theorem appears in the left sum of \eqref{eq:M1_SimpleTF-2}. So both sides are positive in that equation. In particular there exists $\pi^*$ contributing to the right hand side such that $m(\pi^*)>0$. Any other $\tau^*$ in the sum is isomorphic to $\pi^*$ at all finite places away from $\vst$. We also know that $\tau^*_{\vst}$ and $\pi^*_{\vst}$ differ by an unramified character and that $\tau_ \infty$ and $\pi_ \infty$ belong to the same $L$-packet. Now we claim that $\tau^*_{\vst}\simeq \pi^*_{\vst}$. To see this, we apply \cite[Thm. 1.8]{Xu} to deduce that $\tau^*$ and $\pi^*$ belong to the same global $L$-packet as in that paper, by the same argument as in the proof of Lemma \ref{lem:same-v-compo}. Since the local $L$-packet for $\GSp_{2n}(F_{\vst})$ of (any character twist of) the Steinberg representation is a singleton by \cite[Prop. 4.4, Thm. 4.6]{Xu}, the claim follows.\footnote{Lemma \ref{lem:ComponentIsSteinberg} implies that the $L$-parameter $W_{F_{\vst}}\times \SU_2(\R)\ra \SO_{2n+1}(\C)$ for the Steinberg representation of $\Sp_{2n}(F_{\vst})$ restricts to the principal representation of $\SU_2(\R)$ in $\SO_{2n+1}(\C)$. A lift of this to a $\GSpin_{2n+1}(\C)$-valued parameter is attached to the Steinberg representation of $\GSp_{2n}(F_{\vst})$ in Xu's construction. This is enough to imply that the group $\mathcal{S}_{\tilde\phi}$ in \cite[Prop. 4.4]{Xu} is trivial.}

We have shown that the right hand side of \eqref{eq:M1_SimpleTF-2} may be summed over $\tau^*$ which is isomorphic to $\pi^*$ away from infinity and belongs to $\Pi^{G^*}_\xi$ at infinity. Each $\tau^*$ has automorphic multiplicity one by Theorem \ref{thm:AutomorphicMultiplicity} and Lemma \ref{lem:MT-1-independent}. (Without the lemma, some $m(\tau^*)$ could be zero.) So \eqref{eq:M1_SimpleTF-2} comes down to
$$
\sum_\tau m(\tau) = | \Pi^G_\xi |.
$$
Recall that the sum runs over $\tau$ such that $\tau^{ \infty,\vst} \simeq \pi^{ \infty,\vst}$, $\tau_ \infty \in \Pi^G_\xi$, and $\tau_\vst\simeq \pi_{\vst}\otimes\epsilon$ for an unramified character $\epsilon$. By Lemma \ref{lem:MT-1-independent}, the contribution from $\tau$ with $\tau^ \infty \simeq \pi^ \infty$ is already $| \Pi^G_\xi | m(\pi)$, and the other $\tau$ (if any) contributes nonnegatively. We conclude that $m(\pi)=1$. (Moreover, all $\tau$ in the sum with $m(\tau)>0$ should be isomorphic to $\pi$ at $\vst$; this gets used in the next corollary.)
\end{proof}

\begin{corollary}\label{cor:mult-one-cohomology}
The integer $a(\pi^\natural)$ in Corollary \ref{cor:A-is-integer} (defined in \eqref{eq:Defa0}) is equal to one.
\end{corollary}

\begin{proof}
Consider $\pi := \pi^\natural |\cdot |^{n(n+1)/4}$, which satisfies (St) and (L-coh). Applying the last paragraph in the proof of the preceding theorem to $\pi$, we see that for every $\tau \in A(\pi^\natural)$, we have an isomorphism $\tau_{\vst}\simeq \pi^\natural_{\vst}$. Since $\tau_ \infty \in \Pi^G_\xi$ for every $\tau \in A(\pi^\natural)$ by Corollary \ref{cor:Pi-is-in-Disc-series-Packet-2}, it follows from Theorem \ref{thm:AutomorphicMultiplicity2}, Lemma \ref{lem:MT-1-independent}, and Remark \ref{rem:K_infty} that
$$
a(\pi^\natural) = (-1)^{n(n+1)/2} N^{-1}_ \infty \sum_{\tau \in A(\pi^\natural)} m(\tau) \cdot \tu{ep}(\tau_ \infty \otimes \xi)
= |\Pi^G_\xi|^{-1}\sum_{\tau\atop \tau^ \infty\simeq \pi^ \infty,~\tau_ \infty \in \Pi^G_\xi} m(\tau) = 1.
$$
\end{proof}

\section{Meromorphic continuation of the Spin $L$-function}\label{sec:meromorphic}

Let $\pi$ be a cuspidal automorphic representation of $\GSp_{2n}(\A_F)$ unramified away from a finite set of places $S$. The partial spin $L$-function for $\pi$ away from $S$ is by definition
$$
L^S(s,\pi,\spin):=\prod_{v\notin S} \frac{1}{\det(1-q_v^{-s}\spin(\phi_{\pi_v}(\Frob_v)))}.
$$
Various analytic properties of this function would be accessible if the Langlands functoriality conjecture for the $L$-morphism $\spin$ were known. However we are far from it when $n\ge 3$. In particular no results have been known about the meromorphic (or analytic) continuation of $L^S(s,\pi,\spin)$ when $n\ge 6$ (see introduction for some results when $2\le n\le 5$).

The aim of this final section is to establish Theorem \ref{thm:D} on meromorphic continuation for $L$-algebraic $\pi$ under hypotheses (St), (L-coh), and (spin-REG) by applying a potential automorphy theorem, namely Theorem A of \cite{BLGGT14-PA}. Let $S_{\bad}$ be as in Theorem \ref{thm:A}. A character $\mu:\Gamma\ra\GL_{1}(\ol{M}_{\pi,\lambda})$ is considered totally of sign $\epsilon \in \{\pm 1\}$ if $\mu(c_v)=\epsilon$ for all infinite places $v$ of $F$. More commonly a character totally of sign $+1$ or $-1$ is said to be totally odd or even, respectively. For each $\Q$-embedding $w:F\hra \C$ define the cocharacter $\mu_{\Hodge}(\phi_{\pi_w}):\GL_1(\C)\ra \GSpin_{2n+1}(\C)$ by restricting the $L$-parameter $\phi_{\pi_w}:W_{F_w}\ra {}^L \Sp_{2n}$ to $W_\C=\GL_1(\C)$ (via $w$). Condition (L-coh) ensures that $\mu_{\Hodge}(\phi_{\pi_w})$ is an algebraic cocharacter.

\begin{proposition}\label{prop:weakly-compatible}
Suppose that $\pi$ satisfies (St), (L-coh), and (spin-REG).
There exist a number field $M_\pi$ and a representation
$$R_{\pi,\lambda}: \Gamma\ra \GL_{2^n}(\ol{M}_{\pi,\lambda})$$
for each finite place $\lambda$ of $M_{\pi}$ such that the following hold. (Write $\ell$ for the rational prime below $\lambda$.)
\begin{enumerate}
\item At each place $v$ of $F$ not above $S_{\bad}\cup \{\ell\}$, we have $$\mathrm{char}(R_{\pi,\lambda}(\Frob_v))=\mathrm{char}(\spin(\iota\phi_{\pi_v}(\Frob_v))) \in M_{\pi}[X].$$
\item $R_{\pi,\lambda}|_{\Gamma_v}$ is de Rham for every $v|\ell$. Moreover it is crystalline if $\pi_v$ is unramified and $v\notin S_{\bad}$.
\item For each $v|\ell$ and each $w:F\hra \C$ such that $\iota w$ induces $v$, we have $\mu_{\HT}(R_{\pi,\lambda}|_{\Gamma_v},\iota w)=\iota(\spin\circ \mu_{\Hodge}(\phi_{\pi_w}))$. In particular $\mu_{\HT}(R_{\pi,\lambda}|_{\Gamma_v},\iota w)$ is a regular cocharacter for each $w$.
\item $R_{\pi,\lambda}$ is irreducible.
\item $R_{\pi,\lambda}$ maps into $\GSp_{2^n}(\ol{M}_{\pi,\lambda})$ (resp. $\GO_{2^n}(\ol{M}_{\pi,\lambda})$) for a suitable nondegenerate alternating (resp. symmetric) pairing on the underlying $2^n$-dimensional space over $\ol{M}_{\pi,\lambda}$ if ${n(n+1)/2}$ is odd (resp. even). The multiplier character $\mu_\lambda:\Gamma\ra\GL_{1}(\ol{M}_{\pi,\lambda})$ (so that $R_{\pi,\lambda}\simeq R_{\pi,\lambda}^\vee\otimes \mu_\lambda$) is totally of sign $(-1)^{n(n+1)/2}$.
\end{enumerate}
\end{proposition}

\begin{remark}
Parts (1)-(3) of the lemma imply that the family $\{R_{\pi,\lambda}\}$ is a weakly compatible system of $\lambda$-adic Galois representations, cf. \cite[\S5.1]{BLGGT14-PA}.\end{remark}

\begin{remark}\label{rem:coefficients}
Although we do not need it, we can choose $M_{\pi}$ such that $R_{\pi,\lambda}$ is valued in $\GL_{2^n}(M_{\pi,\lambda})$ for every $\lambda$. Concretely we may take $M_\pi$ to be the field of definition for the $\pi^ \infty$-isotypic part in the compact support Betti cohomology with $\Q$-coefficients (with respect to the local system arising from $\xi$). When the coefficients are extended to $M_{\pi,\lambda}$ the $\pi^ \infty$-isotypic part becomes a $\lambda$-adic representation of $\Gamma$ via \'etale cohomology, which is isomorphic to a single copy of $R_{\pi,\lambda}$ if the coefficients are further extended to $\ol M_{\pi,\lambda}$ by Corollary \ref{cor:mult-one-cohomology}.
\end{remark}

\begin{proof}
The first three assertions are immediate from Theorem \ref{thm:RhoPi0}. For the fourth assertion, observe that the image of $\overline \rho_\pi$ in $\SO_{2n+1}$ is either $\PGL_2$, $G_2$ ($n=3$) or $\SO_{2n+1}$, due to the Steinberg hypothesis (and local global compatibility at $v_{\tu{St}}$).

We first exclude the case $G_2$ ($n=3$). In this case $\spin_{7}|_{G_2}$ decomposes as a direct sum of the trivial representation and an odd-dimensional self-dual representation. In particular, up to twist the weight $0$ occurs in $\rho_\pi$ with multiplicity at least $2$. This contradicts however the regularity in item (3). We now exclude the case $\PGL_2$ in a similar way. Write $\psi \colon \SL_2 \to \Spin_{2n+1}$ for the principal $\SL_2$, and decompose $r := \spin \circ \psi = \bigoplus_{i=1}^t r_i$ into irreducible representations $r_i$ of $\SL_2$ (cf. \cite[Prop. 6.1]{GrossMinuscule}). The dimension of the representations $r_i$ all have the same parity: $\dim(r_i) \equiv \dim(r_j) \mod 2$ for all $i,j$. If the dimension of the $r_i$ is odd, the weight $0$ appears in them, and in the even case, the weight $1$ appears in each $r_i$. As $n > 2$, $r$ is reducible, so (up to twist) either the weight $0$ or the weight $1$ appears with multiplicity $t \ge 2$ in $\rho_\pi$.

Thus, if the Galois representation has image in these smaller groups, it can't be regular, which contradicts (3). In particular statement (4) follows.

The first part of (5) is clear from Lemma \ref{lem:similitude-of-spin-rep}. Let us check the second part of (5). Consider the diagram
$$\xymatrix{ \Gamma \ar[r]^-{\rho_\pi} \ar[rd]^-{\mu_\lambda} & \GSpin_{2n+1}(\ol{M}_{\pi,\lambda}) \ar[d]^-{\mathcal{N}} \ar[r]^-{\spin} & \GSp_{2^n}(\ol{M}_{\pi,\lambda})\mbox{~~or~~}\GO_{2^n}(\ol{M}_{\pi,\lambda}) \ar[dl]^-{\simil} \\ & \GL_1(\ol{M}_{\pi,\lambda}) }.$$
The outer triangle commutes by the definition of $\mu_\lambda$. The right triangle also commutes by Lemma \ref{lem:similitude-of-spin-rep}.
Hence $\mu_\lambda=\mathcal{N}\circ \rho_\pi$. By (L-coh) we know that $\varpi_{\pi_v|\cdot|^{n(n+1)/4}}$ is the inverse of the central character of some irreducible algebraic representation $\xi_v$ at each infinite place $v$. Hence $\varpi_{\pi_v|\cdot|^{n(n+1)/4}}:\R^\times\ra \C^\times$ is the $w$-th power map, where $w \in \Z$ is as in (cent) of \S\ref{sect:ShimuraDatum} (in particular $w$ is independent of $v$). Hence $\varpi_{\pi|\cdot|^{n(n+1)/4}}=\varpi_{\pi}|\cdot|^{n(n+1)/2}$ corresponds via class field theory and $\iota$ to an even Galois character of $\Gamma$ (regardless of the parity of $w$). On the other hand, $\mu_\lambda$ corresponds to $\varpi_\pi$ via class field theory and $\iota$ in the $L$-normalization, cf. Theorem \ref{thm:A}. We conclude that $\mu_{\lambda,v}(c_v)=(-1)^{n(n+1)/2}$ for each $v| \infty$.
\end{proof}

Now that we proved the lemma, Theorem A of \cite{BLGGT14-PA} implies

\begin{corollary}\label{cor:ThmDisTrue}
Theorem \ref{thm:D} is true.
\end{corollary}

\begin{proof}
The conditions of the theorem in \textit{loc. cit.} are verified by the above lemma with the following additional observation: the characters $\mu_\lambda$ forms a weakly compatible system since they are associated to the central character of $\pi$ (which is an algebraic Hecke character).
\end{proof}

\begin{remark}
For the lack of precise local Langlands correspondence for general symplectic groups (however note that Bin Xu established a slightly weaker version in \cite{Xu}), we cannot extend the partial spin $L$-function for $\pi$ to a complete $L$-function by filling in the bad places. However the method of proof for the corollary yields a finite alternating product of completed $L$-functions made out of $\Pi$ as in Theorem \ref{thm:D}, by an argument based on Brauer induction theorem as in the proof of \cite[Thm. 4.2]{HSBT10}; this alternating product should be equal to the completed spin $L$-function as this is indeed true away from $S$.
\end{remark}

\appendix

\section{Lefschetz functions}\label{sec:Lefschetz}

We collect some results on Lefschetz functions that are used in the text. In this appendix, $F$ is a characteristic-zero non-archimedean local field of residue characteristic $p>0$, except in Lemma \ref{lem:DiscAutomRepWithSteinbComponent} at the end, where $F$ is global. We collect the required results from the literature and prove some lemmas to deal with small technical difficulties (non-compact center, twisted group).

To help the readers we clarify some terminology. There are three names for the function whose trace computes the Euler-Poincar\'e characteristic or the Lefschetz number of the group cohomology (resp. Lie algebra cohomology) for a given reductive $p$-adic (resp. real) group: they are called \emph{Euler-Poincar\'e functions}, \emph{Kottwitz functions}, or \emph{Lefschetz functions}. In the real case one can consider twisted coefficients by local systems. There are small differences between the three functions. Euler-Poincar\'e and Lefschetz functions are considered on either $p$-adic or real groups, and can be described in terms of pseudo-coefficients for certain discrete series representations. The functions may not be unique but their orbital integrals are well defined. A Kottwitz function mainly refers to a particular function on a $p$-adic group and gives pseudo-coefficients for the Steinberg representations. It is not just characterized by their traces but can be given by an explicit formula, cf. \cite[\S2]{KottwitzTamagawa}. A generalization of Kottwitz functions on a $p$-adic group is given in \cite{SS97} but we will not need it in this paper.

We recall a result of Casselman and Kottwitz. Let $G$ be a connected reductive group and write $q(G)$ for the $F$-rank of the derived subgroup of $G$.

\begin{proposition}\label{prop:ClassicalLefschetzFunction}
Suppose that the center of $G$ is anisotropic over $F$. Then there exists a locally constant compactly supported function $f_{\Lef}=f^G_{\Lef}$ on $G(F)$ such that
\begin{itemize}
\item[(\textit{i})] Assume that the adjoint group of $G$ is simple. For all irreducible admissible unitary representations $\tau$ of $G(F)$ the trace $\Tr \tau(f_{\Lef})$ vanishes unless $\tau$ is either the trivial representation or the Steinberg representation, in which case the trace is equal to $1$ or $(-1)^{q(G)}$ respectively. (If $q(G)=0$ then the trivial and Steinberg representations coincide.)
\item[(\textit{ii})] Let $\gamma \in G(F)$ be a semisimple element, and $I$ the neutral component of its centralizer; we give $I(F) \backslash G(F)$ the Euler-Poincar\'e measure. Then the orbital integral
$$
\uO_{\gamma}(f_{\Lef}) = \int_{I(F) \bs G(F)} f_{\Lef}(g^{-1} \gamma g) \dd g
$$
is non-zero if and only if $I(F)$ has a compact center, in which case $\uO_{\gamma}(f_{\Lef})=1$.
\end{itemize}
\end{proposition}
\begin{proof}
Kottwitz constructed in \cite[Sect.~2]{KottwitzTamagawa} \emph{Lefschetz functions $f_\Lef$} on $G(F)$. These are functions such that
\begin{equation}\label{eq:LefschetzFunctions}
\Tr \tau(f_\Lef) = \sum_{i=0}^ \infty (-1)^i \dim \uH^i_{\tu{cts}}(G(F), \tau),
\end{equation}
for all irreducible representations $\tau$ of $G(F)$. The cohomology in Equation \eqref{eq:LefschetzFunctions} is the continuous cohomology, \ie the derived functors of $\tau \mapsto \tau^{G(F)}$ on the category of smooth representations of finite length. By the computation of the cohomology spaces $\uH^i_{\tu{cts}}(G(F), \tau)$ in \cite[Thm. XI.3.9]{BorelWallach}, for unitary $\tau$, the spaces $\uH^i_{\tu{cts}}(G(F), \tau)$ vanish unless $i=0$ and $\tau$ is the trivial representation, or $i= q(G)$ and $\tau$ is the Steinberg representation, in which case $\dim \uH^{q(G)}_{\tu{cts}}(G(F), \tau) = 1$.
\end{proof}

We now consider the group $\GL_{2n+1}$ equipped with the involution $\theta$ defined by $g \mapsto J {}^{\tu{t}}g^{-1} J$ with $J$ the $(2n+1)\times (2n+1)$ matrix with all entries $0$ except those on the antidiagonal, where we put $1$. Define $\GL_{2n+1}^+ := \GL_{2n+1} \ltimes \langle \theta \rangle$. In \cite{BorelLabesseSchwermer} Borel, Labesse, and Schwermer introduced twisted Lefschetz functions. Below we will cite mostly from the more recent article of Chenevier--Clozel \cite[Sect.~3]{ChenevierClozel}. The results we need from them are proven in a general twisted setup but specialize to our case as follows.

\begin{proposition}\label{prop:TiwstedLefschetzFunction}
There exists a function $\funcGLlef$ with compact support in the subset $\GL_{2n+1}(F) \theta$ of $\GL_{2n+1}^+(F)$, such that
\begin{itemize}
\item[(\textit{i})] For all unitary admissible representations $\tau$ of $\GL_{2n+1}^+(F)$ the trace $\Tr \tau(\funcGLlef)$ vanishes, unless $\tau$ is either the trivial representation or the Steinberg representation, in which case trace is equal to $1$ or
$2(-1)^{n}$ respectively;
\item[(\textit{ii})] Let $\gamma\theta \in G(F) \theta \in \GL_{2n+1}^+(F)$ be a semisimple element, and $I_{\gamma\theta}$ the neutral component of the centralizer of $\gamma\theta$; we give $I_{\gamma\theta}(F) \backslash G(F)$ the \emph{Euler-Poincar\'e measure} \cite[p.30]{ChenevierClozel} (cf. \cite[Sect.~1]{KottwitzTamagawa}). Then the twisted orbital integral
$$
\uO_{\gamma\theta}(\funcGLlef) = \int_{I_{\gamma\theta}(F) \backslash G(F)} \funcGLlef(g^{-1} \gamma \theta(g)) \dd g,
$$
is non-zero if and only if $I_{\gamma \theta}(F)$ has a compact center. Moreover in case $I_{\gamma \theta}(F)$ has compact center, the integral $\uO_{\gamma\theta}(\funcGLlef)$ is equal to $1$.
\end{itemize}
\end{proposition}
\begin{proof}
We deduce this from Proposition 3.8 and Corollary 3.10 of \cite{ChenevierClozel}.
In fact, \cite{ChenevierClozel} states that $\Tr \tu{St}(f_{\Lef}^{\GL_{2n+1}(F)\theta}) = (-1)^{q(G)} P(1)/\eps(\theta)$. In our case one computes $q(G) = 2n$, $\eps(\theta) = (-1)^n$ and $P(1) = 2$, hence the statement of our proposition.
\end{proof}

\newcommand{\funcSPlef}{f_{\Lef}^{\Sp_{2n}(F)}}

It is well known (cf. \cite[1.2]{ArthurBook} or \cite[1.8]{WaldspurgerFacteurs}) that $\Sp_{2n, F}$ is a $\theta$-{twis\-ted} endoscopic group of $\GL_{2n+1, F}$.
Let $\funcSPlef$ be the Lefschetz function on $\Sp_{2n}(F)$ from Proposition \ref{prop:ClassicalLefschetzFunction}.
We introduce the notion of associated functions. Let $G_0$ be a reductive group over $F$. Suppose that $G'_0$ is an endoscopic group for $G_0$ (i.e. $G'_0$ is part of an endoscopic datum for $G_0$). We say that the functions $f_0 \in \cH(G'_0(F))$ and $f'_0 \in \cH(H'_0(F))$ are \emph{associated} if they have matching orbital integrals in the sense of \cite[(5.5.1)]{KottwitzShelstad}. The same definition carries over to the archimedean and adelic setup.

\begin{lemma}\label{lem:FuncSPGLassociated}
Define $C = |F^{\times}/F^{\times 2}|^{-1}$. Then the functions $(\funcSP, C \cdot \funcGLlef)$ are associated.
\end{lemma}
\begin{proof}[Proof of Lemma \ref{lem:FuncSPGLassociated}]
We show that for each $\gamma \in \Sp_{2n}(F)$ strongly regular, semisimple and elliptic, we have
\begin{equation}\label{eq:CheckAssociated}
\SO_\gamma(f^{\Sp_{2n}}_{\Lef}) = \sum_{\delta} \Delta(\gamma, \delta)\SO_{\delta\theta}(C \funcGLlef)
\end{equation}
where $\delta$ ranges over a set of representatives of the twisted conjugacy classes in $\GL_{2n+1}(F)$ which are associated to $\gamma$. Note that the stable orbital integral $\SO_{\gamma}(f^{\Sp_{2n}}_{\Lef})$ is equal to the number of conjugacy classes in the stable conjugacy class of $\gamma$. Our first claim (1) is that $\Delta(\gamma, \delta) = 1$ for all elliptic $\delta$ associated to $\gamma$. Assuming that this claim is true, the right hand side of Equation \eqref{eq:CheckAssociated} is the stable twisted orbital integral $\SO_{\delta\theta}(C \funcGLlef)$, which equals, up to multiplication by $C$, the number of twisted conjugacy classes in the stable twisted conjugacy class of $\delta$. Our second claim is that $\SO_{\gamma}(f^{\Sp_{2n}}_{\Lef}) = \SO_{\delta\theta}(C \funcGLlef)$. Clearly the lemma follows from claims (1) and (2).

Let us now check the two claims. We begin with claim (1). For this we use the formulas in the article of Waldspurger \cite[Sect.~10]{WaldspurgerFacteurs}. The factors $\Delta(\gamma, \delta)$ are complicated and involve many notation to introduce properly, for which we do not have the space to introduce. Thus we use, without introducing, the notation from \cite{WaldspurgerFacteurs}. By \cite[Prop. 1.10]{WaldspurgerFacteurs}
\begin{equation}\label{eq:TransferFactor}
\Delta_{\Sp_{2n}, \wt G}(y, \wt x) = \chi(\eta x_D P_I(1) P_{I^{-}}(-1)) \prod_{i \in I^{-*}} \tu{sgn}_{F_i/F_{\pm i}}(C_i).
\end{equation}
Since our conjugacy class is elliptic we have $H = \Sp_{2n}$, the group $H^-$ is trivial and the sets $I^-$, $I^{-*}$ are empty. A priori, $\chi$ is any quadratic character of $F^\times$, and each choice defines a different isomorphism class of twisted endoscopic data. The choice affects the L-morphism:
\begin{align*}
\eta_\chi \colon {}^L H = \SO_{2n+1}(\C) \times W_F \to {}^L G(\C) = \GL_{2n+1}(\C) \times W_F, \ (g, w) \lmapsto (\chi(w) g, w).
\end{align*}
Since the Steinberg $L$-parameter for $\GL_{2n+1}(\C)$, $W_F \times \SU_2(\R) \to {}^L \GL_{2n+1}(\C)$, is trivial on $W_F$ (and $\tu{Sym}^{2n}$ on $\SU_2(\R)$), it factors through $\eta_\chi$ only for $\chi = 1$ (and not for non-trivial $\chi$). Consequently the transfer factor in Equation \eqref{eq:TransferFactor} is equal to $1$, and claim (1) is true.

We check claim (2). In Sections~1.3 and~1.4 of \cite{WaldspurgerFacteurs}, Waldspurger describes the (strongly) regular, semisimple (stable) conjugacy classes of $\Sp_{2n}(F)$ in terms of $F$-algebras. More precisely, a regular semisimple conjugacy class in $\Sp_{2n}(F)$ is given by data $(F_i, F_{\pm i}, c_i, x_i, I)$, where
$I$ is a finite index set;
for each $i \in I$, $F_{\pm i}$ is a finite extension of $F$;
for each $i \in I$, $F_i$ is a commutative $F_{\pm_ i}$-algebra of dimension $2$;
we have $\sum_{i \in I} [F_i:F] = 2n$;
$\tau_i$ is the non-trivial automorphism of $F_i/F_{\pm i}$;
for each $i \in I$, we have an an element $c_i \in F_i^\times$ such that $\tau_i(c_i) = -c_i$;
$x_i \in F_i^\times$ such that $x_i \tau_i(x_i) = 1$;
to the data $(F_i, F_{\pm i}, c_i, x_i, I)$, Waldspurger attaches a conjugacy class \cite[Eq. (1)]{WaldspurgerFacteurs} and this conjugacy class is required to be regular.

The data $(F_i, F_{\pm i}, c_i, x_i, I)$ should be taken up to the following equivalence relation: The index set $I$ is up to isomorphism; the triples $(F_i, F_{\pm i}, x_i)$ are up to isomorphism; the element $c_i \in F_i^\times$ is given up to multiplication by the norm group $\tu{N}_{F_i/F_{\pm i}}(F_i^\times)$. The stable conjugacy class is obtained from $(F_i, F_{\pm i}, c_i, x_i, I)$ by forgetting the elements $c_i$ \cite[Sect.~1.4]{WaldspurgerFacteurs}, and keeping only $(F_i, F_{\pm i}, x_i, I)$.

According to \cite[Sect.~1.3]{WaldspurgerFacteurs} a strongly regular, semisimple twisted conjugacy class in $\GL_{2n+1}$ is given by data $(L_i, L_{\pm i}, y_i, y_D, J)$ where\footnote{To avoid a collision of notation in our exposition, we write $L_{\pm i}$, $L_i$, $y_i$, $J$, $y_D$ where Waldspurger writes $F_{\pm i}$, $F_i$, $x_i$, $I$, $x_D$.} a finite index set $J$; for each $i \in J$, $L_{\pm i}$ is a finite extension of $F$; for each $i \in J$, $L_i$ is a commutative $L_{\pm i}$-algebra of dimension $2$ over $L_{\pm i}$; for each $i \in J$, an element $y_i \in L_i^\times$; $2n+1 = \sum_{i \in J} [L_i : F]$; $y_D$ an element of $F^\times$; to the data $(L_i, L_{\pm i}, y_i, y_D, J)$ Waldspurger attaches a twisted conjugacy class \cite[p.45]{WaldspurgerFacteurs} and this class is required to be strongly regular.

The data $(L_i, L_{\pm i}, y_i, J)$ should be taken up to the following equivalence relation: $(L_i, L_{\pm i}, J)$ are under the same equivalence relation as before for the symplectic group; the elements $y_i$ are determined up to multiplication by $\tu{N}_{L_i/L_{\pm i}}(L_i^\times)$; $y_D$ is determined up to squares $F^{\times 2}$. The stable conjugacy class of $(L_i, L_{\pm i}, J, y_i, y_D)$ is obtained by taking $y_i$ up to $L_{\pm i}^\times$ and forgetting the element $y_D$.

By \cite[Sect.~1.9]{WaldspurgerFacteurs} the stable (twisted) conjugacy classes $(L_i, L_{\pm i}, x_i, J)$ and $(F_i, F_{\pm i}, y_i, I)$ correspond if and only if $(L_i, L_{\pm i}, x_i, J) = (F_i, F_{\pm i}, y_i, I)$.

The conjugacy class $\gamma$ (resp. $\delta$) is elliptic if the algebra $F_i$ (resp $L_i$) is a field. Let's assume that this is the case (otherwise there is nothing to prove, because the equation $\SO_{\gamma}(f^{\Sp_{2n}}_{\Lef}) = \STO_{\delta}(C \funcGLlef)$ reduces to $0 = 0$). By the description above, to refine the \emph{stable} conjugacy class $(F_i, F_{\pm i}, x_i, I)$ to a conjugacy class, is to give elements $c_i \in F_i^\times / \tu{N}_{F_i/F_{\pm i}} F_i^\times$ such that $\tau_i(c_i) = -c_i$. For each $i$ there are $2$ choices for this, so we get $2^{\# I}$ conjugacy classes inside the stable conjugacy class.

To refine the stable twisted conjugacy class $(L_{\pm i}, L_i, y_i, J)$ to a twisted conjugacy class, is to lift $y_i$ under the mapping $L_i^\times / \tu{N}_{L_i/L_{\pm i}}(L_i^\times) \surjects L_i^\times/L_{\pm i}^\times$. There are $\# L_{\pm i}^\times/\tu{N}_{L_i/L_{\pm i}}(L_i^\times) = 2$ choices for such a lift for each $i \in J$. We get $2^{\# J}$ choices to refine the collection $(y_i)_{i \in J}$. Finally we also have to choose an element $y_D \in F^\times/F^{\times 2}$. In total we have $|F^\times/F^{\times 2}|2^{\# J}$ twisted conjugacy classes inside the stable twisted conjugacy class. Thus if we take $C = |F^\times/F^{\times 2}|^{-1}$ the lemma follows.
\end{proof}

We also need Lefschetz functions on the group $\GSp_{2n}(F)$ and its nontrivial inner form $G$. Unfortunately Proposition \ref{prop:ClassicalLefschetzFunction} does not apply since the center of $\GSp_{2n}(F)$ is not compact. We follow Labesse to construct Kottwitz functions generally for an arbitrary reductive group $G$ over $F$. Let $A$ denote the maximal split torus in the center of $G$. Define $\nu \colon G(F)\ra X^*(A)\otimes_\Z \R$ as in \cite[3.9]{LabesseStabilization} and put $G(F)^1:=\ker \nu$. Consider the exact sequence
$$
1\ra A\ra G \stackrel{\varsigma}{\ra} G':=G/A \ra 1
$$
and a Lefschetz function $f'_\Lef \in \cH(G')$ as in Proposition \ref{prop:ClassicalLefschetzFunction}. The pullback $f'_\Lef\circ \varsigma$ is a function on $G(F)$. Multiplying with the characteristic function $1_{G(F)^1}$, we obtain
\begin{equation}\label{eq:fLef-truncated}
f_{\Lef}=f^G_{\Lef}:=1_{G(F)^1}\cdot (f'_\Lef\circ \varsigma) \in \cH(G(F)).
\end{equation}
Let $G^*$ denote a quasi-split inner form of $G$. Suppose that the Haar measures on $G'(F)$ and $(G^*)'(F)$ are chosen compatibly in the sense of \cite[p.631]{KottwitzTamagawa}. Given a Haar measure on $A(F)$ this determines Haar measures on $G(F)$ and $G^*(F)$. Let us normalize the transfer factor between $G$ and $G^*$ to be (whenever nonzero) the Kottwitz sign $e(G)$, which is equal to $(-1)^{q(G^*)-q(G)}$ by \cite[pp.296-297]{KottwitzSign}.

\begin{lemma}\label{lem:LefschetzMatching-finite}
With the above choices, $(-1)^{q(G)}f^G_{\Lef}$ and $(-1)^{q(G^*)}f^{G^*}_{\Lef}$ are associated.
\end{lemma}

\begin{remark}
Our lemma specifies the constant in \cite[Prop. 3.9.2]{LabesseStabilization} when $\theta=1$ and $H=G^*$.
\end{remark}

\begin{proof}
The proof is quickly reduced to the case when the center is anisotropic, the point being that $G(F)^1$ and $G^*(F)^1$ are invariant under conjugation. From Proposition \ref{prop:ClassicalLefschetzFunction} we deduce that $SO^G_\gamma(f^G_{\Lef})$ is zero if $\gamma$ is non-elliptic and equals the number of $G(F)$-conjugacy classes in the stable conjugacy class of $\gamma$ if $\gamma$ is elliptic. Since the same is true for $G^*$ it is enough to show that the number of conjugacy classes in a stable conjugacy class is the same for $\gamma$ and $\gamma^*$ when they are strongly regular and have matching stable conjugacy classes. This follows from the $p$-adic case in the proof of \cite[Thm. 1]{KottwitzTamagawa}.
\end{proof}

\begin{definition}[\!\!{\cite[Def. 3.8.1, 3.8.2]{LabesseStabilization}}]\label{def:cuspidal-stabilizing}
Let $\phi \in \cH(G(F))$. We say that $\phi$ is \emph{cuspidal} if the orbital integrals of $\phi$ vanish on all regular non-elliptic semisimple elements, and \emph{strongly cuspidal} if the orbital integrals of $\phi$ vanish on all non-elliptic elements and if the trace of $\phi$ is zero on all constituents of induced representations from unitary representations on proper parabolic subgroups. The function $\phi$ is said to be \emph{stabilizing} if $\phi$ is cuspidal and if the $\kappa$-orbital integrals of $\phi$ vanish on all semisimple elements for all nontrivial $\kappa$.
\end{definition}

\begin{lemma}\label{lem:LefschetzIsStabilizing}
The function $f^G_{\Lef}$ is strongly cuspidal and stabilizing. If $\Tr \pi(f^G_\Lef)\neq 0$ for an irreducible unitary representation $\pi$ of $G(F)$ then $\pi$ is an unramified character twist of either the trivial or the Steinberg representation.
\end{lemma}
\begin{proof}
The first assertion is \cite[Prop. 3.9.1]{LabesseStabilization}. By construction $f^G_\Lef$ is constant on $Z(F)\cap G(F)^1$, which is compact.
Since $\Tr \pi(f^G_\Lef)$ is the sum of $\Tr \pi^1(f^G_\Lef)$ over irreducible constituents $\pi^1$ of $\pi|_{G(F)^1}$, we may choose $\pi^1$ such that $\Tr \pi^1(f^G_\Lef)\neq 0$. The nonvanishing implies that $\pi^1$ has trivial central character on $Z(F)\cap G(F)^1$. Let $\omega_\pi$ denote the central character of $\pi$. Then $\omega_\pi$ is trivial on the subgroup $A(F)\cap G(F)^1 \subset A(F)$, and hence induces a morphism
$$\omega_\pi|_{A(F)}:A(F)/(A(F) \cap G(F)^1) \to \C^\times.$$
We have a short exact sequence
$$1\ra G(F)^1 \ra G(F) \ra G(F)/G(F)^1 \ra 1$$
and $G(F)/G(F)^1$ is a lattice of $X^*(A)_{\R}$ via $\nu$. The image of $A(F)$ in $X^*(A)_{\R}$, namely $A(F)/A(F)\cap G(F)^1$, is a sublattice of $G(F)/G(F)^1$ since the map clearly factors through $\nu:G(F)\ra X(A)_{\R}$. Thus $A(F)G(F)^1$ is a subgroup of $G(F)$ with a finite abelian quotient. We can think of $A(F)G(F)^1/A(F)=G(F)^1/A(F)\cap G(F)^1$ as a subgroup of $G(F)/A(F)=G'(F)$ with a finite abelian quotient as well.

Since $\C^\times$ is divisible, we can extend $\omega_\pi|_{A(F)}$ to a character $\omega:G(F)/G(F)^1\ra\C^\times$. Twisting by $\omega^{-1}$, we may assume that $\omega_\pi|_{A(F)}$ is trivial.
Let us write $f$ for $f^G_{\tu{Lef}}$. Now we compute
\begin{eqnarray}
\Tr \pi(f)&=& \int_{g \in G(F)/A(F)} \int_{a \in A(F)} \theta_\pi(ag) f(ag)da dg \nonumber\\
&=& \vol(A(F)\cap G(F)^1) \int_{ G(F)^1/A(F)\cap G(F)^1} \theta_{\pi}(g)f(g) dg \nonumber\\
&=&\vol(A(F)\cap G(F)^1) \int_{ G(F)^1/A(F)\cap G(F)^1} \theta_{\pi}(g)f'(g) dg,\nonumber
\end{eqnarray}
where $\pi$ is viewed as a representation of $G'(F)$ (since $\omega_\pi|_{A(F)}=1$), in which $G(F)^1/A(F)\cap G(F)^1$ is a finite-index subgroup, and where $f'$ is a Lefschetz function for $G'(F)$.
Notice that the integral after the second equality is well defined as $f$ and $\pi$ are invariant under $A(F)\cap G(F)^1$.

Write $X$ for the finite set of characters of $G(F)/A(F)G(F)^1 = G'(F)/G(F)^1$. They are unramified characters of $G'(F)$. Then
$$\Tr \pi(f)=\vol(A(F)\cap G(F)^1)|X|^{-1}\sum_{\chi \in X} \int_{ G'(F)} \chi(g)\theta_{\pi}(g)f'(g) dg$$
$$= \vol(A(F)\cap G(F)^1)|X|^{-1}\sum_{\chi \in X} \Tr (\pi\otimes \chi)(f').$$
Thus if $\Tr \pi(f)\neq 0$ then $\pi\otimes \chi$ is the trivial or Steinberg representation of $G'(F)$. We conclude that $\pi$ is the trivial or Steinberg representation of $G(F)$ up to an unramified twist.
 \end{proof}

When the center of $G$ is not compact, it is also convenient to work with functions with fixed central character. We adopt the notation and convention from \S\ref{sect:TraceFormula}. We record a result in the following, where a representation is said to be essentially unitary if its character twist is unitary.

\begin{corollary}\label{cor:Lef-variant2}
Let $\eta \colon G(F) \to \C^\times$ be a character, and write
$\omega := \eta|_{A(F)}$.
Define a function $f^G_{\Lef, \eta} \in \cH(G(F),\omega^{-1})$
 by $f^G_{\Lef, \eta}(g) := \eta^{-1}(g) f^{G/A}_{\Lef}(\li g)$, where $\li g \in G(F)/A(F)$ denotes the
image of $g$ under the quotient map. Then $f^G_{\Lef, \eta}$ has the following properties:
\begin{enumerate}
\item $f^G_{\Lef, \eta}$ is stabilizing (in particular cuspidal).
\item Let $\pi$ be an irreducible essentially unitary $G(F)$-representation
whose central character on $A(F)$ coincides with $\omega$.
Then
$\Tr(f^G_{\Lef, \eta}, \pi) = 0$ unless $\pi \in \{\eta, \St_{G(F)}\otimes
\eta\}$.
\item $f^G_{\Lef, \eta}(zg) = \eta\inv (g) f^G_{\Lef, \eta}(g)$ for all $z \in Z_G(F)$ and $g \in G(F)$.
\end{enumerate}
\end{corollary}
\begin{proof}
Let us write $\li G:=G/A$ in this proof.
The second point quickly follows from the analogue for $f^{\li G}_{\Lef}$ on $\li G(F)=G(F)/A(F)$. Let us show the first point. Let $g \in G(F)$ be a semisimple element, and $\li g \in \li G(F)$ its image. Write $I$ (resp.~$\li I$) for the connected centralizer of $g$ in $G$ (resp.~$\li g$ in $\li G$). Then the natural map $I\ra \li I$ induces $I/A\simeq \li I$, as in the proof of \cite[Lem.~3.1(1)]{KottwitzRational}.
Equipping quotient measures on $\li I$ and $\li G$ with respect to those on $A$, $I$, and $G$, we have
\begin{equation}\label{eq:eta-twist-orbital}
O_g(f^G_{\Lef, \eta})=\eta(g) O_g(f^G_{\Lef,1}) = \eta(g) O_{\li g}(f^{\li G}_{\Lef}).
\end{equation}
This implies that $f^G_{\Lef, \eta}$ is cuspidal since $f^{\li G}_{\Lef}$ is cuspidal.

It remains to verify that the $\kappa$-orbital integrals $O^\kappa_g(f^G_{\Lef, \eta})$ vanish for each semisimple $g$ and $\kappa\neq 1$.
The first observation is that $I\twoheadrightarrow \li I$ induces a canonical isomorphism $\mathfrak{K}(\li I/F)\stackrel{\sim}{\ra} \mathfrak{K}(I/F)$ in the notation of \cite[\S4]{KottwitzEllipticSingular}. To see this, one can dually check that the induced map $\mathfrak{E}(I/F)\ra \mathfrak{E}(\li I/F)$ (using the notation therein) is an isomorphism. This follows from the following commutative diagram with exact rows, coming from the long exact sequence associated with $0\ra A \ra I \ra \li I \ra 0$, since $\mathfrak{E}(I/F)$ and $\mathfrak{E}(\li I/F)$ are the kernels of the first and second vertical maps, respectively. (We have used Hilbert 90 for the split torus $A$.)
$$\xymatrix{
 0 \ar[r] & H^1(F,I) \ar[r] \ar[d] & H^1(F,\li I) \ar[d] \ar[r] & H^2(F,A) \ar@{=}[d] \\
 0 \ar[r] & H^1(F, G) \ar[r] & H^1(F,\li G) \ar[r] & H^2(F,A)
}$$
The second observation is that the map $G\twoheadrightarrow\li G$ induces a surjection from the set of conjugacy classes in the stable conjugacy class of $g$ onto the analogous set for $\li g$, where the cardinality of fibers remains constant, as explained on pp.\,611--612 of \cite{HainesBCFL}.

From these two observations combined with \eqref{eq:eta-twist-orbital}, we deduce that $O^\kappa_g(f^G_{\Lef, \eta})=0$ if and only if $O^\kappa_{\li g}(f^{\li G}_{\Lef})=0$ for each $\kappa\neq 1$ (viewed in either $\mathfrak{K}(\li I/F)$ or $\mathfrak{K}(I/F)$ via the isomorphism above).
Since $f^{\li G}_{\Lef}$ is stabilizing, $O^\kappa_{\li g}(f^{\li G}_{\Lef})$ vanishes, therefore so does $O^\kappa_g(f^G_{\Lef, \eta})$.

The last assertion of the corollary boils down to the statement that $f^{G/A}_{\Lef}(gz) = f^{G/A}_{\Lef}(g)$ for all $g \in G/A(F)$ and $z \in Z_{G/A}(F)$. This follows from Harish-Chandra's Plancherel theorem \cite[Thm.~VIII.1.1 (3)]{WaldspurgerPlancherel} (applied with $f=f^{G/A}_{\Lef}$) since the right hand side of that theorem does not change when $f^{G/A}_{\Lef}$ is translated by $z$; the point is that the only irreducible tempered representation of $(G/A)(F)$ with nonzero trace against $f^{G/A}_{\Lef}$ is the Steinberg representation, which has the trivial central character.
\end{proof}

Assume from now on that $G$ is a non-split inner form of a quasi-split group $G^*$ over $F$. Consider a finite cyclic extension $E/F$ with $\theta$ generating the Galois group. Put $\tilde G^*:=\Res_{E/F} G^*$ equipped with the evident $\theta$-action. (The case $G^*=\GSp_{2n}$ is used in the main text.)
\begin{lemma}\label{lem:SteinbergAssociated} The function $f^{\tilde G^*}_{\Lef}$ is strongly cuspidal and stabilizing on the twisted group $\tilde G^* \theta$. There exist constants $c \in \C^\times$ such that
the functions $\tilde c f^{G^*}_\Lef$ and $f^{\tilde G^*}_\Lef$ are associated (via base change).
\end{lemma}

\begin{proof}
\cite[Prop. 3.9.2]{LabesseStabilization}. (Take $\theta=1$, $H=G^*$ in \textit{loc. cit.} for the first assertion.)
\end{proof}

Now we turn our attention to Lefschetz functions (also called Euler-Poincar\'e functions) on connected reductive groups $G$ over $F=\R$. We assume that $G(\R)$ has discrete series representations. Let $\xi$ be an irreducible algebraic representation of $G\times_\R \C$. Denote by $\chi_\xi:Z(\R)\ra \C^\times$ the restriction of $\xi^\vee$ to the center $Z(\R)$. Write $f_\xi=f^{G}_\xi \in \cH(G(\R),\chi_\xi^{-1})$ for an Euler-Poincar\'e function associated to $\xi$, characterized by the identity
$$
\Tr \pi(f_\xi)=\ep_{K_ \infty}(\pi\otimes \xi)\stackrel{\mathrm{def}}{=}\sum_{i=0}^ \infty (-1)^i \dim \uH^i(\Lie G(\R), K_ \infty; \pi \otimes \xi)
$$
for every irreducible admissible representation $\pi$ of $G(\R)$ with central character $\chi_\xi$, where $K_ \infty$ is a finite index subgroup in the group generated by the center and a maximal compact subgroup of $G(\R)$. (A main example for us is $\tilde K_H$ in Definition \ref{def:cohomological}.)
The formula does not determine $f_\xi$ uniquely but the orbital integrals of $f_\xi$ are well defined. The function $f_\xi$ exists by Clozel and Delorme and can be constructed as the average of pseudo-coefficients as follows. Let $\Pi^G_\xi$ denote the $L$-packet consisting of (the isomorphism classes of) irreducible discrete series representations $\pi$ of $G(\R)$ whose central character and infinitesimal character are the same as $\xi^\vee$. Write $f_{\pi} \in \cH(G(\R),\chi_\xi^{-1})$ for a pseudo-coefficient of $\pi$, and $q(G) \in \Z$ for the real dimension of $G(\R)/K_ \infty$. Then $f_\xi$ can be defined by
\begin{equation}\label{eq:AvgLef}
f^G_\xi:=(-1)^{q(G)} \sum_{\pi \in \Pi^G_\xi} f_{\pi}.
\end{equation}
Let $G^*$ denote a quasi-split inner form of $G$ over $\R$. Then the same $\xi$ gives rise to the functions $f^{G^*}_\xi$. Note that Definition \ref{def:cuspidal-stabilizing} makes sense verbatim when $F$ is archimedean and the function has central character.

\begin{lemma}\label{lem:Lefschetz}
Any pseudo-coefficient $f_\pi$ for a discrete series representation $\pi$ is cuspidal. In particular $f^{G}_\xi$ and $f^{G^*}_\xi$ are cuspidal.
\end{lemma}

\begin{proof}
It suffices to check that $f_\pi$ is cuspidal, or equivalently that the trace of every induced representation from a proper Levi subgroup vanishes against $f_\pi$, cf. \cite[p.538]{ArthurInvariantII}, but this is true by the construction of pseudo-coefficients. (The cuspidality of $f_\pi$ also follows from the proof of \cite[Lem. 3.1]{KottwitzInventiones}.)
\end{proof}

Let $A$ denote the maximal split torus in the center of $G$ (hence also in $G^*$). Equip $G(\R)/A(\R)$ and $G^*(\R)/A(\R)$ with Euler-Poincar\'e measures and $A(\R)$ with the Lebesgue measure so that the Haar measures on $G(\R)$ and $G^*(\R)$ are determined. Define $q(G)$ (resp. $q(G^*)$) to be the real dimension of the symmetric space associated to the derived subgroup of $G(\R)$ (resp. $G^*(\R)$). Normalize the transfer factor between $G(\R)$ and $G^*(\R)$ to be $e(G)=(-1)^{q(G^*)-q(G)}$. Write $\Pi^G_\xi$ (resp. $\Pi^{G^*}_\xi$) for the discrete series $L$-packet for $G(\R)$ (resp. $G^*(\R)$) associated to $\xi$, cf. Example below Definition \ref{def:cohomological}.

\begin{lemma}\label{lem:LefschetzMatching}
The functions $(-1)^{q(G)} |\Pi^G_\xi|^{-1} f^{G}_\xi$ and $(-1)^{q(G^*)} |\Pi^{G^*}_\xi|^{-1} f^{G^*}_\xi$ are associated.
\end{lemma}

\begin{proof}
This follows from the computation of stable orbital integrals in \cite[Lem. 3.1]{KottwitzInventiones}; see also \cite[Prop. 3.3]{CL11} when $A=1$.
\end{proof}

A similar construction works in the base change context, cf. \cite{CL11}. We are only concerned with a particular case that $\tilde G^*=G^*\times G^*\times \cdots \times G^*$ (the number of copies is $d \in \Z_{\ge 1}$) and $\theta$ is the automorphism $(g_1,...,g_d)\mapsto (g_2,...,g_d,g_1)$. Write $\tilde \xi:=\xi\otimes \cdots \otimes \xi$. A function $f_{\tilde \xi} \in \cH(\tilde G^*(\R),\chi_\xi^{-1})$ is said to be a (twisted) Lefschetz function for $\tilde \xi$ if
$$
\Tr \tilde \pi( f_{\tilde\xi})=\sum_{i=0}^ \infty (-1)^i \Tr(\theta\,|\, \uH^i\left(\Lie G(\R)^d, K^d_ \infty; \pi \otimes \xi\right)
$$
for every irreducible admissible representation $\pi$ of $G(\R)$ with central character $\chi_\xi$. Definition \ref{def:cuspidal-stabilizing} carries over to this base change setup as in \cite[Def. 3.8.1,~3.8.2]{LabesseStabilization}.

\begin{lemma}\label{lem:TwistedLefschetz}
The function $f_{\tilde\xi}$ is cuspidal.\footnote{When $H^1(\R,G_{\mathrm{sc}})=1$ (assuming $A=1$), Clozel and Labesse prove that $f_{\tilde\xi}$ is also stabilizing in \cite[Thm. A.1.1]{LabesseStabilization}. However we do not use it in the main text but instead appeal to the fact that the (twisted) Lefschetz function at a finite place is stabilizing.} Moreover $f_{\tilde\xi}$ and $\tilde c f^{G^*}_\xi$ are associated for some $\tilde c \in \C^\times$.
\end{lemma}
\begin{proof}
The first assertion is \cite[Prop. 3.3]{CL11}. The second assertion follows from \cite[Prop. 3.3, Thm. 4.1]{CL11}. (The reference assumes that $A=1$ but the arguments can be adapted to the case of nontrivial $A$ as in the untwisted setup above).
\end{proof}

We end this appendix with a global result. We change notation. Let $F$ be a totally real number field and $\vst$ a finite $F$-place. Let $G$ be an inner form of either the group $\GSp_{2n, F}$ or the group $\Sp_{2n, F}$.

\begin{lemma}\label{lem:DiscAutomRepWithSteinbComponent}
Assume that $n > 1$. Let $\pi$ be a non-abelian discrete automorphic representation of $G(\A_F)$, and assume that $\Tr \pi_{\vst}(\wt f^G_{\vst}) \neq 0$. Then $\pi_{\vst}$ is an unramified twist of the Steinberg representation.
\end{lemma}
\begin{remark}
The lemma is false for $n = 1$.
\end{remark}
\begin{proof}
We give the argument only in case $G$ is an inner form of $\GSp_{2n, F}$; the argument for inner forms of $\Sp_{2n, F}$ is similar. By the assumption $\Tr \pi_{\vst}(\wt f^G_{\vst}) \neq 0$, $\pi_{\vst}$ is either a twist of the Steinberg representation or a twist of the trivial representation (Lemma \ref{lem:LefschetzIsStabilizing}). We assume that we are in the second case; after twisting we may assume that $\pi_{\vst}$ is the trivial representation. Let $G_1 \subset G$ be the kernel of the factor of similitudes. By strong approximation the subset $G_1(F) G_1(F_{\vst}) \subset G_1(\A_F)$ is dense if $G_{1, F_{\vst}}$ is non-anisotropic. Let us assume this for now. Let $f \in \pi$. Since $\pi_{\vst}$ is the trivial representation, $f$ is invariant under $G(F_{\vst})$. Thus $f$ is $G_1(F) G_1(F_{\vst})$-invariant, and hence $G_1(\A_F)$ invariant. This implies that $\pi$ is abelian. It remains to check that $G_{1, F_{\vst}}$ is non-anisotropic. In the split case, we have $G_{1, F_{\vst}} \simeq \Sp_{2n, F_{\vst}}$. In the non-split case, the group $G_{1,F_{\vst}}$ is of the following form. Let $D / F_{\vst}$ be the quaternion algebra, and consider the involution on $D$ defined by $\li x = \Tr(x) - x$ where $\Tr$ is the reduced trace. Then $G(1,F_{\vst}) \simeq \Sp_n(D)$ is the group of $g \in \GL_n(D)$ such that $g A_n {}^\tu{t} \li g = c(g) A_n$ for some $c(g) \in F^\times_{\vst}$, where $A_n$ is the matrix with 1's on the anti-diagonal, and 0's everywhere else \cite[item (3), p.92]{PlatonovRapinchuk}. For $n > 1$ the group $\Sp_n(D)$ has a strict parabolic subgroup and thus $G_{1, F_{\vst}}$ is not anisotropic.
\end{proof}

\section{Conjugacy in the standard representation}

Let $V$ be a finite dimensional $\lql$-vector space and $\langle \cdot, \cdot \rangle$ a non-degenerate, symmetric or skew-symmetric bilinear form on $V$. Let $H \subset \GL(V)$ be the subgroup of elements that preserve $\langle \cdot, \cdot \rangle$ (resp. preserve it up to scalar). Thus, $H(\lql)$ is either an orthogonal group or a symplectic group (resp. of similitude). Write $\simil \colon H(\lql) \to \lql^\times$ for the factor of similitude. Note $\simil$ is non-trivial only for the groups $\GO(V, \langle \cdot, \cdot \rangle)$ and $\GSp(V, \langle \cdot, \cdot \rangle)$.

The results of this appendix largely overlap with \cite[Prop A (and Prop 2.1)]{Wang2015}, although he works in a slightly different setting.

\begin{proposition}[{Larsen, cf. proof of \cite[Prop. 2.3, Prop. 2.4]{LarsenConjugacy}}]\label{prop:ConjugacyStdRep}
Let $\Gamma$ be a topological group. Consider two (continuous) representations $\phi_1, \phi_2 \colon \Gamma \to H(\lql)$ such that $\phi_1$ is semisimple. Then $\phi_1, \phi_2$ are $H(\lql)$-conjugated if and only if $\simil \circ \phi_1 = \simil \circ \phi_2$ and $\std \circ \phi_1$, $\std \circ \phi_2$ are $\GL_m(\lql)$-conjugate.
\end{proposition}
\begin{remark}
The conclusion of Proposition \ref{prop:ConjugacyStdRep} fails in general for the special orthogonal group in even dimension. In odd dimension the group $\uO_{2n+1}(\lql)$ equals $\{\pm 1\} \times \SO_{2n+1}(\lql)$ and so the proposition is true for $\SO_{2n+1}(\lql)$ as well.
\end{remark}
\begin{proof}[Proof of Proposition \ref{prop:ConjugacyStdRep}]
We consider the group $H = \GO(V, \langle \cdot, \cdot \rangle)$ with $\langle \cdot, \cdot \rangle$ symmetric and non-degenerate, the other groups are treated in a similar fashion. Fix a morphism $\phi \colon \Gamma \to \GO(V, \langle \cdot, \cdot \rangle)$ with $\std \circ \phi$ semisimple. Write $\chi = \simil \circ \phi \colon \Gamma \to \lql^\times$. Consider the set $X_\chi(\phi)$ of $\GO(V, \langle \cdot, \cdot \rangle)$-conjugacy classes of morphisms $\phi' \colon \Gamma \to \GO(V, \langle \cdot, \cdot \rangle)$ such that $\std \circ \phi ' \simeq \std \circ \phi$ and $\chi = \simil \circ \phi'$. We view the space $V$ as a $\Gamma$-representation via $\phi$. We have the injection
$$
X_\chi(\phi) \injects \Isom_\Gamma(V, V^* \otimes \chi) / \Aut_\Gamma(V), \quad [\phi', \rho \colon \std \circ \phi' \isomto_\Gamma \std \circ \phi] \mapsto \rho_* \langle \cdot, \cdot \rangle,
$$
whose image is the set of non-degenerate symmetric pairings taken modulo $\Aut_{\Gamma}(V)$, where the automorphisms $\sigma \in \Aut_{\Gamma}(V)$ act on $\Isom_\Gamma(V, V^* \otimes \chi)$ via $\rho \mapsto (\sigma^* \cdot )\circ \rho \circ (\sigma \cdot)$. Using the pairing $\langle \cdot, \cdot \rangle$ on $V$ we can further identify,
$$
\Isom_\Gamma(V, V^* \otimes \chi) / \Aut_\Gamma(V) \isomto \Aut_\Gamma(V) / \Aut_\Gamma(V)\textup{-congruence},
$$
where by congruence we mean the action $\tau \mapsto \sigma^\tu{t} \tau \sigma$ for $\sigma \in \Aut_\Gamma(V)$; here the transpose is defined using $\langle \cdot, \cdot \rangle$. Since $(V, \phi)$ is semisimple we can consider the isotypical decomposition $V = \bigoplus_{i=1}^t V_i^{d_i}$, and so by Schur's lemma $\End_\Gamma(V) = \prod_{i=1}^t \uM_{d_i}(\lql)$. We obtain an embedding
$$
X_\chi(\phi) \injects \prod_{j=1}^t \GL_{d_i}(\lql) / \GL_{d_i}(\lql)\tu{-congruence},
$$
where two matrices $X, Y \in \GL_{d_i}(\lql)$ are $\GL_{d_i}(\lql)$-congruent if there exists a third matrix $g \in \GL_{d_i}(\lql)$ such that $Y = g^{\tu{t}} X g$. The image of $X_\chi(\phi)$ in the set on the right hand side decomposes along the product and is in each $\GL_{d_j}(\lql)$-factor equal to the set of congruence classes of invertible symmetric matrices. Since $\lql$ is algebraically closed (of characteristic $\neq 2$), these classes have exactly one element.
\end{proof}

\bibliographystyle{amsalpha}

\providecommand{\bysame}{\leavevmode\hbox to3em{\hrulefill}\thinspace}
\providecommand{\MR}{\relax\ifhmode\unskip\space\fi MR }
\providecommand{\MRhref}[2]{%
 \href{http://www.ams.org/mathscinet-getitem?mr=#1}{#2}
}
\providecommand{\href}[2]{#2}



\end{document}